\documentclass[11pt]{article}

\usepackage{fullpage}
\usepackage[utf8]{inputenc}
\usepackage[T1]{fontenc}
\usepackage{amsmath,amsfonts,amssymb,amsthm,bbm}
\usepackage{stackrel}
\usepackage{enumerate}
\usepackage{graphicx}
\usepackage{hyperref}
\usepackage{subfigure}
\usepackage{color}

\newtheorem{theorem}{Theorem}

\newtheorem{proposition}[theorem]{Proposition}

\newtheorem{lemma}[theorem]{Lemma}
\newtheorem{definition}[theorem]{Definition}

\newtheorem{remark}[theorem]{Remark}

\newtheorem{question}{Question}

\numberwithin{theorem}{section}
\numberwithin{equation}{section}
\numberwithin{figure}{section}

\newcommand{\ve}{\varepsilon}
\newcommand{\ind}{\mathbbm{1}}

\newcommand{\ol}{\overline}
\newcommand{\ul}{\underline}
\newcommand{\dol}[1]{\ol{\ol{#1}}}
\newcommand{\dul}[1]{\ul{\ul{#1}}}

\newcommand{\calP}{\mathcal{P}}
\newcommand{\calD}{\mathcal{D}}

\newcommand{\ZZ}{\mathbb{Z}}
\newcommand{\RR}{\mathbb{R}}
\newcommand{\CC}{\mathbb{C}}

\newcommand{\PP}{\mathbb{P}}
\newcommand{\PPP}{\overline{\mathbb{P}}}
\newcommand{\EE}{\mathbb{E}}

\newcommand{\TT}{\mathbb{T}} 
\newcommand{\Ball}{B} 
\newcommand{\Ann}{A} 
\newcommand{\din}{\partial^{\textrm{in}}} 
\newcommand{\dout}{\partial^{\textrm{out}}} 
\newcommand{\cluster}{\mathcal{C}}
\newcommand{\Cinf}{\cluster_\infty} 
\newcommand{\lclus}{\cluster^{\textrm{max}}} 
\newcommand{\holew}{\overline{\mathcal{H}}} 
\newcommand{\holes}{\mathring{\mathcal{H}}} 
\newcommand{\Ch}{\mathcal{C}_H} 
\newcommand{\Cv}{\mathcal{C}_V} 
\newcommand{\circuit}{\mathcal{C}} 
\newcommand{\circuitevent}{\mathcal{O}} 
\newcommand{\circuitarm}{\mathcal{OA}} 
\newcommand{\colorseq}{\mathfrak{S}} 
\newcommand{\arm}{\mathcal{A}} 
\newcommand{\darm}{\overline{\mathcal{A}}} 
\newcommand{\marm}{\widehat{\mathcal{A}}} 

\newcommand{\net}{\mathcal{N}} 
\newcommand{\Dint}[2]{{#1}^{\textrm{int}(#2)}}
\newcommand{\Dout}[2]{{#1}^{\textrm{out}(#2)}}

\newcommand{\next}{\psi}
\newcommand{\deltaKMS}{\mathfrak{d}}

\newcommand{\lra}{\leftrightarrow}
\newcommand{\llra}{\longleftrightarrow}

\begin{document}

\title{A 2D forest fire process beyond the critical time}
\author{Jacob van den Berg\footnote{CWI and VU University Amsterdam; E-mail: \texttt{J.van.den.Berg@cwi.nl}.}, Pierre Nolin\footnote{City University of Hong Kong; E-mail: \texttt{bpmnolin@cityu.edu.hk}. Partially supported by a GRF grant from the Research Grants Council of the Hong Kong SAR (project CityU11307320).}}

\date{}

\maketitle

\begin{abstract}
We study forest fire processes in two dimensions. On a given planar lattice, vertices independently switch from vacant to occupied at rate $1$ (initially they are all vacant), and any connected component ``is burnt'' (its vertices become instantaneously vacant) as soon as its cardinality crosses a (typically large) threshold $N$, the parameter of the model. This process was considered by Brouwer and the first author in \cite{BB2006}.

Our analysis provides a detailed description, as $N \to \infty$, of the process near and beyond the critical time $t_c$ (at which an infinite cluster would arise in the absence of fires). In particular we prove a somewhat counterintuitive result: there exists $\deltaKMS > 0$ such that with high probability, the origin does not burn before time $t_c + \deltaKMS$. This provides a negative answer to Open Problem~4.1 of \cite{BB2006}. Informally speaking, the result can be explained in terms of the emergence of fire lanes, whose total density is negligible (as $N \to \infty$), but which nevertheless are sufficiently robust with respect to recoveries. We expect that such a behavior also holds for the classical Drossel-Schwabl model.

A large part of this paper is devoted to understanding the role played by recoveries during the time interval $[t_c, t_c + \deltaKMS]$. These recoveries do have a ``microscopic'' effect everywhere on the lattice, but it turns out that their combined influence on macroscopic scales (and in fact on relevant ``mesoscopic'' scales) vanishes as $N \to \infty$.

In order to prove this, we use key ideas from a paper by Kiss, Manolescu and Sidoravicius \cite{KMS2015}, introducing a suitable induction argument to extend and strengthen their results. We then use it to prove that a deconcentration result in our earlier joint work with Kiss \cite{BKN2018} on volume-frozen percolation (a model \emph{without} recoveries) also holds for the forest fire process. As we explain, significant additional difficulties arise here, since recoveries destroy the nice spatial Markov property of frozen percolation.

\bigskip

\textit{Key words and phrases: near-critical percolation, forest fires, self-organized criticality.}
\end{abstract}

\tableofcontents

\section{Introduction} \label{sec:intro}

\subsection{Background and motivation} \label{sec:backgr}

Consider the following process on a two-dimensional lattice $G = (V,E)$, such as the square lattice $\ZZ^2$ or the triangular lattice $\TT$. Initially, at time $t=0$, all vertices are vacant, and they become occupied (by a ``tree'')  at rate $1$ (the birth rate), forming occupied connected components (clusters) on $G$. Additionally, each occupied vertex is ``hit by lightning'' at rate $\zeta$, the parameter of the model: when this occurs, the entire occupied cluster of the vertex ``burns instantaneously'' i.e. all vertices in that cluster become vacant at once. All these burnt vertices then become occupied again (``recover'') at rate $1$, and so on.

This model is a (continuous-time) version of the Drossel-Schwabl forest fire model \cite{DrSc1992}, which has received much attention in the physics literature, as well as in papers on ecology and related fields (this classical process is often viewed as a paradigmatic situation where \emph{self-organized criticality} can be observed). In this paper we simply call it the Drossel-Schwabl model. The first mathematically rigorous paper where this process was studied in dimension $>1$ is, as far as we know, the paper \cite{BB2006} by one of us and Brouwer.

Note that if all ignitions are ignored, i.e. in the corresponding ``birth process'', then at each time $t$ we have a configuration where the sites of the lattice are, independently of each other, occupied with probability $p = p(t) := 1 - e^{-t}$, and vacant with probability $1-p$. The study of the size of occupied connected components (and other connectivity properties) in such random configurations is the subject of (Bernoulli) \emph{percolation theory}, a field which has witnessed impressive developments in the last decades (see Section~\ref{sec:percolation} for a brief introduction to percolation theory, and a list of results that will be used in this paper).

One of the first results in percolation theory was the existence of a \emph{critical value} $p_c$, strictly between $0$ and $1$, such that for $p > p_c$ there is (almost surely) an infinite occupied cluster, while for $p < p_c$ there are only finite occupied clusters. Related to this early percolation result, the above mentioned paper \cite{BB2006} raised a fundamental open problem on forest fires, where $t_c$ is the time at which an infinite cluster ``starts to form'' in the birth process, i.e. defined by the relation $1 - e^{- t_c} = p_c$. This open problem is essentially the following: is it true that, for all $t > t_c$, the probability that a given vertex, say $0$, burns before time $t$ does \emph{not} tend to $0$ as $\zeta \searrow 0$? 

The following intuitive argument ``by contradiction'' (taken roughly from \cite{BB2006}) suggests an affirmative answer to the above question:

\begin{quote}
``Suppose there is a $t > t_c$ for which the above mentioned probability does go to $0$. Take a $t' \in (t_c, t)$. Trivially the mentioned probability also tends to $0$ for $t'$. So, roughly speaking, as $\zeta \searrow 0$, the system at time $t'$ looks like ordinary Bernoulli percolation with parameter $1 - e^{-t'}$. Since this is larger than $p_c$, there is a positive probability that $0$ belongs to an infinite occupied cluster at time $t'$. However, such a cluster would burn immediately, hence before time $t$: contradiction.''
\end{quote}

As said in \cite{BB2006}, this type of reasoning is very shaky, but its conclusion turns out to be correct for the directed binary tree (see Lemma~4.5 in that paper). The same paper also stated a percolation-like conjecture (Conjecture 2.1 there). This conjecture says, roughly speaking, that, for Bernoulli percolation at $p = p_c$, if certain macroscopic occupied clusters in an $n$ by $n$ box are destroyed (made vacant), after which \emph{every} vacant vertex in the box gets, independently, an extra chance $\delta$ to become occupied, then this $\delta$ needs to be bounded away from $0$ as $n \rightarrow \infty$ to assure large-scale connectivity in the resulting configuration.

The paper also introduced a variant of the Drossel-Schwabl model where the ignition mechanism is not Poissonian, but where instead an occupied cluster is burnt immediately when its size (volume) is at least $N$, which is now the parameter of the model. From now on, this process is called the \emph{$N$-forest fire}, and we denote the corresponding probability measure by $\PP_N$. One may expect that for very large $N$, it behaves in many respects the same as the Drossel-Schwabl model with parameter $\zeta = \frac{1}{N}$, and the paper \cite{BB2006} stated similar open problems and results for this process as for the Drossel-Schwabl model. In particular, it contained this question:

\begin{quote}
\textbf{Open Problem 4.1 in \cite{BB2006}.} \emph{Is, for all $t > t_c$,}
\begin{equation}\label{eq-OP4.1-in-BB}
\limsup_{N \rightarrow\infty} \PP_N \big( 0 \textit{ burns before time } t \big) > 0 \text{?}
\end{equation}
\end{quote}

In addition, the following conditional result was proved, where we denote $\Ball_m := [-m,m]^2$.

\begin{quote}
\textbf{Theorem 4.2 in \cite{BB2006}.} \emph{If the above-mentioned percolation-like Conjecture~2.1 in \cite{BB2006} is true, then there exists a $t > t_c$ such that for all $m \geq 1$,}
\begin{equation}\label{eq:BB-N}
\liminf_{N \rightarrow \infty} \PP_N \big( \exists v \in \Ball_m \: : \: v \textit{ burns before time } t \big) \leq \frac{1}{2}.
\end{equation}
\end{quote}

Much of our later work in this field was motivated by the open problem that we just stated (and its analog for the Drossel-Schwabl model, Open Problem~1.1 in \cite{BB2006}). Our hope was that, by first studying similar problems for models with a more ``sober'' dynamics, in particular volume-frozen percolation, some key features would be discovered, and that the combination of these features would lead to a solution to Open Problem~4.1 in \cite{BB2006}. Volume-frozen percolation can roughly, but not exactly (see the discussion in Section~\ref{sec:main-geom}), be considered, interpreting freezing as burning, as an $N$-forest fire process without recoveries.

One of these features was a deconcentration phenomenon, proved for volume-frozen percolation in our joint work \cite{BKN2018} with Kiss. Let us briefly and very informally describe this property. At each time we can consider the connected component of non-frozen vertices containing $0$, that we call the ``hole'' or ``island'' of $0$ (note that at time $0$, this is the set of all vertices). At some point it becomes finite, but with a volume \emph{much} bigger than $N$, and then, step by step (where each step corresponds to a freezing event in the current island), it shrinks further. For $0$ itself to freeze, it ``must'' at some step be in an island with a size of order, and larger than, $N$: after the next, and last, step, $0$ will then (with high probability) either be and remain in an island of size $> 0$ but smaller than $N$, or freeze itself (which typically has a very small probability if the former island has size of order $N$ but much bigger than $N$).

The deconcentration phenomenon says that the island sizes are very sensitive for disturbances earlier in the sequence, so that the size of the one-but-last island mentioned above is highly unpredictable: the probability that its size is larger, but not much larger, than $N$ (and hence the probability that $0$ eventually freezes), is very small (tends to $0$ as $N \rightarrow \infty$).

In the meantime, Kiss, Manolescu and Sidoravicius had obtained in \cite{KMS2015} a clever proof of the percolation-like conjecture of \cite{BB2006} (or, rather, an equally suitable version of that conjecture). Hence the ``conditional'' Theorem 4.2 of \cite{BB2006}, stated above, became an ``unconditional'' one, i.e. the ``If'' part in that theorem could be removed. However, Open Problem~4.1 of \cite{BB2006} (also mentioned above) remained open.

The percolation-like conjecture in \cite{BB2006}, and the version proved in \cite{KMS2015}, suggest that, very roughly speaking, certain regions that have become disconnected from each other by ``barriers'' of burnt vertices, caused by a fire near (or beyond) time $t_c$, remain with high probability disconnected during a time $\deltaKMS$, for some $\deltaKMS > 0$ which does not depend on $N$: the recovery of barriers is slow. This gives hope that the islands of $0$, that we alluded to earlier, remain sufficiently intact if recoveries are allowed, so that the deconcentration result for $N$-volume frozen percolation should also hold for the $N$-forest fire model. 

In the present paper we show that this hope is indeed justified: the sequence of islands for the $N$-forest fire, and that for volume-frozen percolation (or, rather, the last ``many'' steps of these two sequences), can be compared sufficiently well, even though the spatial Markov property which holds for the sequence of islands in the frozen percolation model is no longer valid in the $N$-forest fire. This is achieved by adapting the reasonings in Sections~6 and 7 of \cite{BKN2018}, and by developing and applying suitable versions of the results in \cite{KMS2015}. We thus prove that the $N$-forest fire process satisfies a similar deconcentration phenomenon as $N$-frozen percolation, which leads to our main result, Theorem~\ref{thm:main_result} below, giving a \emph{negative} answer to Open Problem~4.1 of \cite{BB2006}.

\subsection{Main results}

Our analysis provides a detailed understanding of the asymptotic behavior of the $N$-forest fire process on a time interval $[0, t_c + \deltaKMS]$, for some positive number $\deltaKMS$ which does not depend on $N$. It can be summarized in particular by the following result. For technical reasons explained in Remark~\ref{rem:triang_lattice} below, we need to focus on the case where $G$ is the triangular lattice (even though we expect the same behavior to occur on other natural two-dimensional lattices such as $\ZZ^2$).

\begin{theorem} \label{thm:main_result}
Consider the $N$-forest fire process on the triangular lattice $\TT$, and the associated critical time $t_c = t_c(\TT)$. There exists $\deltaKMS > 0$ (universal) such that
\begin{equation} \label{eq:proba_burning}
\PP_N \big( 0 \text{ burns before time } t_c + \deltaKMS \big) \stackrel[N \to \infty]{}{\longrightarrow} 0.
\end{equation}
\end{theorem}

Hence,
\begin{equation}
\PP_N \big( 0 \text{ is occupied at time } t \big) \stackrel[N \to \infty]{}{\longrightarrow} p(t) = 1 - e^{-t}
\end{equation}
uniformly over $t \in [0, t_c + \deltaKMS]$. More generally, it implies immediately that for every $K \geq 0$,
\begin{equation} \label{eq:local_limit}
\PP_N \big( \exists v \in \Ball_K \: : \: v \text{ burns before time } t_c + \deltaKMS \big) \stackrel[N \to \infty]{}{\longrightarrow} 0.
\end{equation}
The local limit of the $N$-forest fire process is thus simply Bernoulli site percolation with parameter $p(t)$.

Theorem~\ref{thm:main_result} provides a negative answer to Open Problem~4.1 of \cite{BB2006} and, as pointed out in Section~\ref{sec:backgr}, may look counterintuitive since $p(t) > p_c$ for $t \in (t_c, t_c + \deltaKMS]$ (in other words, the density of occupied sites is ``supercritical''). The result \eqref{eq:local_limit} substantially improves Theorem~7 of \cite{KMS2015}, which states that the liminf of the l.h.s. is $\leq \frac{1}{2}$.


\begin{remark}
We believe that our proofs could provide an explicit upper bound in \eqref{eq:proba_burning}. Such a quantitative statement would allow us to reach the same conclusion as in \eqref{eq:local_limit}, while letting $K \to \infty$ as a (very slow) function of $N$.
\end{remark}

We can also consider the $N$-forest fire process in a finite subgraph $G$ of $\TT$, for which we use the notation $\PP_N^{(G)}$. A result similar to Theorem~\ref{thm:main_result} holds for sequences of such domains, provided they grow fast enough in $N$.

\begin{theorem} \label{thm:main_result_finite}
For some $\deltaKMS > 0$, the following holds. For all $\ve > 0$, there exists $\eta = \eta(\ve) > 0$ such that: if $m(N) \geq N^{\frac{48}{91} - \eta}$ for all sufficiently large $N$, then
$$\limsup_{N \to \infty} \PP_N^{(\Ball_{m(N)})} \big( 0 \text{ burns before time } t_c + \deltaKMS \big) \leq \ve.$$
\end{theorem}

The exponent $\frac{48}{91}$ may look somewhat intriguing at first sight. It is associated with a particular ``characteristic length'' $m_{\infty}(N)$, which is natural for the $N$-forest fire process. We explain in Section~\ref{sec:map_next} where it comes from (in \eqref{eq:def_t_inf} and below). Moreover, it can be interpreted as a limit of the \emph{exceptional scales} $m_k(N)$, $k \geq 1$, introduced in \cite{BN2017a} (see \eqref{eq:def_mk} and the subsequent explanation).

Given Theorem~\ref{thm:main_result}, one may wonder if, or even expect that, \eqref{eq:proba_burning} holds not only at a time $t_c + \deltaKMS$ for some $\deltaKMS > 0$, but even at any fixed time $t$. However, it is not clear how to iterate our arguments, and we state this as an open problem.

\begin{question}
Is it true that for \emph{all} $t \geq 0$,
\begin{equation}
\PP_N \big( 0 \text{ burns before time } t \big) \stackrel[N \to \infty]{}{\longrightarrow} 0?
\end{equation}
\end{question}

\subsection{Additional remarks on the proofs}

As we said in Section~\ref{sec:backgr}, the proofs of our main results rely heavily on two earlier works: our paper \cite{BKN2018} with Kiss on frozen percolation, and the paper \cite{KMS2015} by Kiss, Manolescu and Sidoravicius. We first discuss the relation with the latter in Section~\ref{sec:intro_sdp}, before emphasizing in Section~\ref{sec:main-geom} the impact of so-called boundary rules.

\subsubsection{Effect of recoveries: self-destructive percolation} \label{sec:intro_sdp}

We mentioned earlier that a crucial ingredient in the present paper is the percolation-like conjecture of \cite{BB2006}, of which a suitable version was established in \cite{KMS2015}. Here we write ``suitable'' because all the conditional results of \cite{BB2006}, whose proofs assumed the validity of Conjecture~2.1 in that paper, also follow from the alternative version in \cite{KMS2015}, as explained there. We now discuss these matters in a bit more detail, and explain in particular why we had to extend the work of \cite{KMS2015} further to prove our main results.

The crossing estimate of \cite{KMS2015} (Theorem~4 there) can be described in the following way. For $n \geq 1$, we consider the birth process at time $t_c$ (that is, Bernoulli percolation at criticality) in the rectangle $[-3n, 3n] \times [0, n]$. We burn all connected components which touch both the left and right sides of the rectangle, and we then let the birth process evolve further (i.e. vacant vertices independently become occupied at rate $1$). Then for some $\deltaKMS > 0$ which does not depend on $n$, the probability that in the beginning (i.e. in the Bernoulli percolation configuration at $p = p_c$) there is an occupied horizontal crossing of the rectangle, and in the end, at time $t_c + \deltaKMS$, there is a vertical occupied crossing of the subrectangle $[-2n, 2n] \times [0, n]$ tends to $0$ as $n \to \infty$. More precisely, this probability is at most $c n^{-\lambda}$, where $c$, $\lambda > 0$ are two universal constants.

Using arguments in \cite{BB2006}, this has several interesting consequences. In particular, it follows that the conditional result \eqref{eq:BB-N} of \cite{BB2006} is now an unconditional result (Theorem~7 of \cite{KMS2015}). In addition, the ``intuitively natural'' limiting ($N = \infty$) forest fire process, where infinite occupied clusters become vacant as soon as they appear, while finite ones keep growing as usual, does not exist (Theorem~6 of \cite{KMS2015}), a non-existence result which, again, highlights clearly the difficulty to analyze forest fires around time $t_c$.

We revisit some of the geometric ideas in the proof of the percolation-like conjecture in \cite{KMS2015}. Note that if initially there is an occupied horizontal crossing, its destruction also prevents from crossing the box vertically. A key observation in \cite{KMS2015} is that if a new top-bottom connection arises after some amount of time $\deltaKMS > 0$, then there exists a necklace of ``highly connected'' regions, separated by successive cut vertices (\emph{passage sites} in their terminology) which have recovered (after the burnings at time $t_c$). Each of these passage sites then provides locally what is known in percolation theory as a (possibly ``defective'') six-arm configuration: six disjoint paths (namely, two occupied paths, one of which may have a ``defect'', and four vacant ones), each connecting a neighbor of the passage site to some distance. After identifying the existence of such arm configurations, a major difficulty is then to sum them in an appropriate way, with a view to showing that the ``entropy'' contribution, coming from the many arm configurations, gets ``beaten'' by the probability that such a given configuration exists (taking into account that each passage site, since it recovers before time $t_c + \deltaKMS$, contributes, roughly, a ``cost'' $\deltaKMS$: see Propositions~14 and 15 of \cite{KMS2015}).

The summation argument in \cite{KMS2015} is very clever and well-suited for the purposes in that paper, but it also appears to be quite rigid and not very flexible for other situations. Here, we develop a different approach, based on a somewhat abstract and exotic looking six-arm event, whose definition (see Definition~\ref{def:6_arms_modified} below) includes the recovery of the passage sites themselves, and is inspired by the deterministic Lemma~16(i) in \cite{KMS2015}. As it turns out, this has the advantage of allowing for rather direct proofs by induction, and we believe that it gives interesting insight into how recoveries can combine together. We obtain in particular analogs of the above-mentioned Theorem~4 of \cite{KMS2015} which are applicable in a variety of geometric situations, for example when the outermost circuit in a given annulus is burnt (instead of all left-right connections in a rectangle). Moreover, generalizations can quite easily be obtained, for instance where the crossing event considered in the final configuration is not macroscopic (like the vertical crossing of the box) but ``mesoscopic'', or where the time of destruction is more flexible. For those purposes, we introduce the notion of \emph{geodesics} (with respect to recoveries). Finally, this strategy should be useful for potential extensions to the Drossel-Schwabl model, as we explain in Section~\ref{sec:intro_DS}.

\subsubsection{Effect of the boundary rules} \label{sec:main-geom}

We now mention an important geometric issue arising with volume-frozen percolation, related to how boundary vertices of frozen clusters behave. In Section~\ref{sec:backgr} we said that volume-frozen percolation can roughly be considered as an $N$-forest fire without recoveries, writing ``roughly'' because of the following. In \cite{KMS2015} we followed the usual convention, or rule, for frozen percolation that when a cluster freezes, vertices along its outer boundary, which are all (by definition) vacant at that time, then remain vacant forever. A consequence of this boundary rule is that, after each step in the sequence mentioned in Section~\ref{sec:backgr}, the new island can be considered as ``fresh territory''. This spatial Markov property was very useful for the analysis in \cite{KMS2015}, and it indicates, in addition to the absence of recoveries, an extra aspect of the difference between volume-frozen percolation and the $N$-forest fire model that needs to be addressed in the present paper.

In the paper~\cite{BN2017b} it was shown that a change of boundary rules has a substantial impact for another frozen percolation process, introduced in \cite{BLN2012}, where clusters are frozen according to their \emph{diameter} (instead of volume). In the current situation, we will see (and in fact obtain ``almost automatically'', as an auxiliary result from the arguments in the proof of our main theorem) that they do not have a significant effect, so that volume-frozen percolation has some sturdiness in this respect. This has to do with the particular geometry of percolation holes, and we comment more on it in Section~\ref{sec:geom_discussion}.

Finally, we want to stress that we do \emph{not} prove that the $N$-forest fire process is close to frozen percolation with the same parameter. Roughly speaking, what we show (and is sufficient for the corresponding part of our proof of Theorem \ref{thm:main_result}) is that in the sequence of nested circuits that burn around $0$ after time $t_c$, if we fix any $k \geq 1$, the last $k$ such circuits are not really affected by recoveries, in the limit as $N \to \infty$. However, the total number of nested circuits is known to tend to $\infty$, and our techniques to connect processes in finite domains and in the full lattice do not allow for a comparison for the whole sequence of circuits.

\begin{question}
Can the $N$-frozen percolation and $N$-forest fire processes be compared, as $N \to \infty$? For example, can they be coupled so that at some time $t_c + \deltaKMS$, the islands containing the origin in the two processes are close to each other?
\end{question}

\subsection{Related works}

\subsubsection{Drossel-Schwabl process} \label{sec:intro_DS}

We explained in Section \ref{sec:backgr} that our paper focuses on the $N$-forest fire process, a variant of the Drossel-Schwabl process (abbreviated as D-S process in the rest of this section). An auxiliary complication of the D-S process itself is that although, intuitively, burnt clusters ``typically'' have size of order $\frac{1}{\zeta}$ (so that one expects this process to be similar to the $N$-forest fire process with $N = \frac{1}{\zeta}$), one also has to control the effect of smaller fires on the connectivity properties (and hence on the development of later fires) in the graph.

In \cite{BN2018}, we developed a method to control such effects in a forest fire which, like the D-S model, has Poisson ignitions, but which does not have recoveries (i.e. vertices, once burnt, remain vacant forever). A key idea in that paper was to introduce a percolation process where first (possibly large) ``impurities'' are created randomly on the lattice, after which we then study Bernoulli site percolation with parameter $p$ on the remaining graph (i.e. on the ``perforated'' graph, after removal of the impurities). We showed that, for a particular ``heavy-tailed'' choice of the law of the impurities, the resulting percolation process is stochastically dominated by the state of the forest fire process slightly before $t_c$. In \cite{BN2018}, this was then used to prove that a ``rich'' degree of connectivity exists in some prescribed domains, with high probability as $\zeta \searrow 0$, just after time $t_c$ (within a near-critical time window).

An important intermediate step in \cite{BN2018} (see Section~4 there) was to show that, in the above mentioned percolation process with heavy-tailed impurities, certain four-arm events are not too much affected by the impurities. One may expect that the type of \emph{six}-arm events considered below, which play a major role in controlling the effect of recoveries, possess a similar stability property, even though the proof for four arms in \cite{BN2018} is too crude in the case of six-arm events. We hope that the proofs by induction that we develop in Section~\ref{sec:positive_rec_time} of the present paper, for recoveries in the $N$-forest fire model, can be further adapted and refined, allowing us then to control the \emph{combined} impact of recoveries \emph{and} ``small'' fires in the D-S model. We plan to study these ideas further in a future work, and hope to answer the following question.

\begin{question}
Prove analogous results to Theorems~\ref{thm:main_result} and \ref{thm:main_result_finite} for the Drossel-Schwabl model, as the ignition rate $\zeta \searrow 0$.
\end{question}

In \cite{BN2018}, we made use of the percolation process with impurities to highlight the connections and analogies between the Drossel-Schwabl and the $N$-forest fire processes, in particular to show the existence of exceptional scales for the former (so far, under the condition that recoveries are forbidden). This relation was further studied in \cite{LN2021} (by Lam and the second author), where it is shown that near-critical ``avalanches'' of a similar nature arise in the two cases (again, in the absence of recoveries).

\subsubsection{Other literature on these and related subjects}

In \cite{BLN2012}, another frozen-percolation process was considered, that we alluded to earlier, where occupied clusters freeze when their diameter (instead of their volume) crosses the threshold $N$. This diameter-frozen version was later studied extensively by Kiss in \cite{Ki2015}, where a quite detailed description is obtained.

We also mention works on frozen percolation and forest fires on various other graphs, starting with the case where the graph $G$ is a tree. In his influential paper \cite{Al2000}, Aldous constructed a frozen percolation process on the binary tree where a cluster freezes immediately when its size becomes infinite. Benjamini and Schramm \cite{BS1999} pointed out that the analog of Aldous' process does not exist on planar lattices. Aldous also obtained surprisingly explicit quantitative results for his model (i.e. on the binary tree). In our paper \cite{BKN2012} with Kiss, we showed that Aldous' process is, in some sense, the limit, as $N \rightarrow \infty$, of the frozen percolation process where clusters freeze as soon as they have size $N$. In particular, the probability that a given vertex is eventually on a microscopic island does not vanish as $N \rightarrow \infty$ (very differently from the behavior on the triangular lattice, described in our subsequent paper with Kiss \cite{BKN2018}).

A natural open problem which arose from the work of Aldous (see \cite{AB2005}), namely whether his process on the binary tree is \emph{endogenous} (i.e. entirely determined by the birth process) was finally solved, with a \emph{negative} answer, in \cite{RST2021}. See also \cite{RSS2021} for more recent work related to that question.

In the beginning of this paper, we started by ``describing'' a continuous-time version of the Drossel-Schwabl model (which, for simplicity, we from then on also called
Drossel-Schwabl model), namely a forest fire model where ignitions occur at a rate $\zeta$. The existence of this process does not follow (unlike the $N$-forest-fire model) from standard arguments, because the ``interaction'' is not finite-range. Indeed, the size of occupied clusters is not uniformly bounded. For this process, a clever existence proof was given by D\"urre \cite{Du2006}.

Forest fires, frozen percolation and related processes have also been studied on the complete graph (see \cite{RT2009}, \cite{Ra2009}, as well as the more recent paper \cite{MR2017}) and in one dimension (see e.g. the monograph \cite{BF2013}, and the references therein). Finally, we mention the paper \cite{ADKS2015}, which gives results suggesting that for high $d$, forest fires and frozen percolation on the hypercubic lattice $\ZZ^d$ behave in a substantially different way from those in the plane. For a more extensive list of references we refer the reader to the nice discussion in Section~1.7 of \cite{RST2021}.

Finally, we conclude this section by mentioning that in the statistical physics literature (as well as in Earth and life sciences), the Drossel-Schwabl model has been studied as a (supposedly) paradigmatic example where the phenomenon of \emph{self-organized criticality} arises. Most of that literature is of a very heuristic nature, or focuses on simulations. For instance, the well-known paper \cite{MMT1998} uses computer simulations for such models to ``explain'' the power-law behavior followed by sizes of real forest fires (see also e.g. the more recent paper \cite{ZJG2010}). There seems to be controversy in the community of physicists studying self-organized criticality over the nature of the (supposed) criticality of the Drossel-Schwabl model, see e.g. \cite{PJ2020} and the references therein.

\subsection{Organization of the paper} \label{sec:intro_org}

In Section~\ref{sec:percolation}, we first set notations about Bernoulli percolation in two dimensions, and recall its most important properties, especially at and near criticality. We then introduce in Section~\ref{sec:holes} \emph{percolation holes}, which play a central role in our reasonings, and discuss some of their geometric features. We define formally the processes under consideration in Section~\ref{sec:FP_FF}, and introduce an associated iteration map related to successive burnings around a vertex. In Section~\ref{sec:positive_rec_time}, we extend results of \cite{KMS2015} by introducing in particular various idiosyncratic six-arm events, which are well adapted to proofs by induction. That section, which is the core of the paper, derives estimates that show that the combined effect of recoveries remains, in some sense, negligible on large scales for some positive amount of time. These reasonings are used in Section~\ref{sec:stab_recov} to study the evolution of percolation holes under the effect of recoveries, and establish a form of robustness for them. This property is instrumental in the last part, Section~\ref{sec:dec_FF}, where we first analyze our forest fire process in finite domains which are ``neither too small nor too large'', by comparing it to volume-frozen percolation. In such domains, we obtain in this way, in Section~\ref{sec:dec_iteration}, a deconcentration property for forest fires. We then use this intermediate result in Section~\ref{sec:dec_proof} to obtain Theorems~\ref{thm:main_result} and \ref{thm:main_result_finite}. Finally, we provide the proof of some regularity result in an Appendix.

\section{Preliminaries: 2D percolation near criticality} \label{sec:percolation}

In this preliminary section, we first set notations about Bernoulli percolation in Section~\ref{sec:setting}, considering mostly the case of site percolation on the (planar) triangular lattice. In Section~\ref{sec:near_critical_perc}, we then list classical properties of this process, in particular pertaining to its near-critical behavior, which are used extensively later in the paper.

\subsection{Bernoulli site percolation} \label{sec:setting}

In the whole paper, we focus on the triangular lattice $\TT = (V, E)$. It is the graph with vertex set $V = V_{\TT} := \big\{ x + y e^{i \pi / 3} \in \CC \: : \: x, y \in \ZZ \big\}$ (where we use the identification $\CC \simeq \RR^2$), and edge set $E = E_{\TT} := \{ \{v, v'\} \: : \: v, v' \in V \text{ with } |v - v'| = 1 \}$. For any two vertices $v$, $v' \in V$, the notation $v \sim v'$ means that they are connected by an edge ($\{v, v'\} \in E$), in which case they are called neighbors. Let $A \subseteq V$ be a subset of vertices. Its volume is the number $|A|$ of vertices that it contains, and its inner and outer boundaries are defined, respectively, as $\din A := \{v \in A \: : \: v \sim v'$ for some $v' \in A^c\}$ and $\dout A := \{v \in A^c \: : \: v \sim v'$ for some $v' \in A\}$ ($= \din (A^c)$). Finally, we denote by $\Ball_n := [-n,n]^2$ the ball with radius $n \geq 0$ around $0$ for the $L^{\infty}$ norm $\|.\| = \|.\|_{\infty}$, and for $0 \leq n_1 < n_2$, by $\Ann_{n_1,n_2} := \Ball_{n_2} \setminus \Ball_{n_1}$ the annulus with radii $n_1$ and $n_2$ centered on $0$. We also write $\Ball_n(z) := z + \Ball_n$ and $\Ann_{n_1,n_2}(z) := z + \Ann_{n_1,n_2}$, for $z \in \CC$.

Next, we introduce Bernoulli site percolation of $\TT$ with parameter $p \in [0,1]$. It is obtained by designating each vertex $v \in V$ occupied with probability $p$, and vacant with probability $1-p$, independently of the other vertices. We obtain in this way a vertex configuration $\omega = (\omega_v)_{v \in V}$, and we denote by $\Omega := \{0,1\}^V$ the set of all such configurations. The corresponding (product) probability measure is denoted by $\PP_p$. Two vertices $v$, $v' \in V$ are said to be connected, which is denoted by $v \lra v'$, if for some $k \geq 0$, there exists a path of length $k$ from $v$ to $v'$, i.e. a sequence of vertices $v_0 \sim v_1 \sim \ldots \sim v_k$ with $v_0 = v$ and $v_k = v'$, along which all vertices, so in particular $v$ and $v'$, are occupied (such a path is said to be occupied). More generally, we say that two subsets of vertices $A$, $A' \subseteq V$ are connected if there exist $v \in A$ and $v' \in A'$ such that $v \lra v'$, and we use the notation $A \lra A'$. Maximal connected components of occupied vertices are called occupied clusters. Vacant paths and clusters are defined in a similar way, with vacant vertices instead.

Each occupied vertex $v \in V$ belongs to one occupied cluster, that we denote by $\cluster(v)$, setting $\cluster(v) = \emptyset$ when $v$ is vacant. We let $\Cinf$ be the union of all infinite occupied clusters, we write $v \lra \infty$ for the event that $v \in \Cinf$, and we define $\theta(p) := \PP_p(0 \lra \infty)$. Site percolation on $\TT$ displays a phase transition at the percolation threshold $p_c = p_c^{\textrm{site}}(\TT)$, which is known to be equal to $\frac{1}{2}$ \cite{Ke1980}: for all $p \leq p_c = \frac{1}{2}$, there is a.s. no infinite occupied cluster, while for all $p > \frac{1}{2}$, there exists a.s. a unique such cluster. For more background on percolation theory, we refer the reader to the classical references \cite{Ke1982, Gr1999}.

If we consider a rectangle of the form $R = [x_1,x_2] \times [y_1,y_2]$ ($x_1 < x_2$, $y_1 < y_2$), a horizontal crossing in $R$ is an occupied path which starts and ends on the left and right sides of $R$, respectively, while staying in $R$. We use the notation $\Ch(R)$ for the event that such a crossing exists. We define in a similar way vertical crossings in $R$ (with the top and bottom sides instead), denoting their existence by $\Cv(R)$, as well as vacant horizontal and vertical crossings, for which we use the notations $\Ch^*(R)$ and $\Cv^*(R)$, respectively.

If $A = \Ann_{n_1,n_2}(z)$ is an annulus as above, we let $\circuitevent(A)$ (resp. $\circuitevent^*(A)$) be the event that an occupied (resp. vacant) circuit exists in $A$, i.e. an occupied (resp. vacant) path which remains in $A$, starts and ends at the same vertex, and surrounds $\Ball_{n_1}(z)$ once. Let $k \geq 1$, and $\sigma \in \colorseq_k := \{o,v\}^k$ (where we write $o$ and $v$ for ``occupied'' and ``vacant'', resp.). We introduce the arm event $\arm_{\sigma}(A)$, which requires the existence of $k$ disjoint paths $(\gamma_i)_{1 \leq i \leq k}$ in $A$, in counter-clockwise order, each from a vertex of $\dout \Ball_{n_1}(z)$ to a vertex of $\din \Ball_{n_2}(z)$, and with ``type'' (i.e. whether it is an occupied or a vacant path) given by $\sigma_i$. We write $\pi_{\sigma}(n_1,n_2) := \PP_{p_c}( \arm_{\sigma}(\Ann_{n_1,n_2}) )$, and $\pi_{\sigma}(n) := \pi_{\sigma}(0,n)$. In the particular case when $\sigma =(ovo\ldots) \in \colorseq_k$, $k \geq 1$, is alternating, we use the notations $\arm_k$ and $\pi_k$.

\begin{remark} \label{rem:triang_lattice}
In this paper we restrict ourselves to site percolation on $\TT$, even though we expect most of our reasonings to work on any lattice with enough symmetries, such as the square lattice $\ZZ^2$ (and also for similar processes defined, instead, in terms of bond percolation). The reason why we do so is because the starting point of our proofs is a deconcentration property for volume-frozen percolation. This result was established in \cite{BKN2018}, using the construction of the scaling limit of near-critical percolation \cite{GPS2018a} as an important ingredient, in order to derive asymptotic formulas for $\theta(p)$, as well as $L(p)$ (defined in the beginning of Section~\ref{sec:near_critical_perc}), as $p \searrow p_c$.
\end{remark}

\subsection{Behavior at and near criticality: main results} \label{sec:near_critical_perc}

Our results rely heavily on a precise description of two-dimensional Bernoulli percolation around the critical point $p = p_c$, which allows us, by comparison, to analyze the near-critical behavior of forest fire processes. We now collect the classical properties that we use throughout the proofs. As usual, a key tool is the characteristic length $L$, defined as follows:
\begin{equation} \label{eq:def_L}
\text{for $p < p_c = \frac{1}{2}$,} \quad L(p) := \min \big\{ n \geq 1 \: : \: \PP_p \big( \Cv( [0,2n] \times [0,n] ) \big) \leq 0.001 \big\},
\end{equation}
and $L(p) := L(1-p)$ for $p > p_c$. It follows from classical results (the Russo-Seymour-Welsh bounds at $p = p_c$) that $L(p) \to \infty$ as $p \to p_c$, and we set $L(p_c) := \infty$.

The function $L$ is clearly piecewise constant, and as in earlier works, we rather use a regularized version $\tilde L$, obtained in a natural way: we let $\tilde L(0) = \tilde L(1) := 0$, we set $\tilde L(p) := L(p)$ at each point of discontinuity $p \in(0,p_c) \cup (p_c,1)$ of $L$, and we then extend linearly $\tilde L$ to $[0,1] \setminus \{p_c\}$. This new function $\tilde L$ has the nice property of being continuous and strictly increasing (resp. strictly decreasing) on $[0,p_c)$ (resp. $(p_c,1]$), so that it is a bijection from $[0,p_c)$ (resp. $(p_c,1]$) to $[0,\infty)$. In the following, we only use this regularized version $\tilde L$, and for simplicity we write $L$ instead of $\tilde L$.

\begin{enumerate}[(i)]

\item \emph{Russo-Seymour-Welsh-type bounds.} For all $K \geq 1$, there exists $\delta_4 = \delta_4(K) > 0$ such that the following holds. For all $p \in (0,1)$, and $1 \leq n \leq K L(p)$,
\begin{equation} \label{eq:RSW}
\PP_p \big( \Ch( [0,4n] \times [0,n] ) \big) \geq \delta_4 \quad \text{and} \quad \PP_p \big( \Ch^*( [0,4n] \times [0,n] ) \big) \geq \delta_4.
\end{equation}

\item \emph{Exponential decay property.} There exist $\kappa_1$, $\kappa_2 > 0$ such that (see e.g. Lemma~39 in \cite{No2008}): for all $p > p_c$, and $n \geq 1$,
\begin{equation} \label{eq:exp_decay}
\PP_p \big( \Ch( [0, 4n] \times [0,n] ) \big) \geq 1 - \kappa_1 e^{- \kappa_2 \frac{n}{L(p)}}.
\end{equation}

\item \emph{Extendability of arm events.} For all $k \geq 1$, and $\sigma \in \colorseq_k$, there exists $C = C(\sigma) > 0$ such that (see Proposition~16 in \cite{No2008}): for all $0 \leq n_1 < n_2$,
\begin{equation} \label{eq:extendability}
\pi_{\sigma} \bigg( \frac{n_1}{2}, n_2 \bigg), \: \pi_{\sigma}(n_1, 2 n_2) \geq C \pi_\sigma(n_1, n_2).
\end{equation}

\item \emph{Quasi-multiplicativity of arm events.} For all $k \geq 1$, and $\sigma \in \colorseq_k$, there exist $C_1, C_2 > 0$ (depending only on $\sigma$) such that (see Proposition 17 in \cite{No2008}): for all $0 \leq n_1 < n_2 < n_3$,
\begin{equation} \label{eq:quasi_mult}
C_1 \pi_{\sigma}(n_1, n_3) \leq \pi_{\sigma}(n_1, n_2) \pi_{\sigma}(n_2, n_3) \leq C_2 \pi_{\sigma}(n_1, n_3).
\end{equation}

\item \emph{Arm exponents at criticality.} For all $k \geq 1$, and $\sigma \in \colorseq_k$, there exists $\alpha_{\sigma} > 0$ such that
\begin{equation} \label{eq:arm_exponent}
\pi_{\sigma}(k,n) = n^{- \alpha_{\sigma} + o(1)} \quad \text{as $n \to \infty$.}
\end{equation}
Moreover, the value of $\alpha_{\sigma}$ is known, except in the so-called monochromatic case, where $k \geq 2$ and $\sigma \in \colorseq_k$ is constant. For $k=1$, $\alpha_{\sigma} = \frac{5}{48}$, and for each $k \geq 2$, and all non-constant $\sigma \in \colorseq_k$ (i.e. containing both types), $\alpha_{\sigma} = \frac{k^2-1}{12}$. These critical arm exponents were established in \cite{LSW2002, SW2001}, using the conformal invariance property of critical percolation \cite{Sm2001} and properties of the Schramm-Loewner Evolution (SLE) processes \cite{LSW_2001a, LSW_2001b}  (with parameter $\kappa = 6$ here).

\item \emph{Stability for arm events near criticality.} For all $k \geq 1$, $\sigma \in \colorseq_k$, and $K \geq 1$, there exist $C_1$, $C_2 > 0$ (depending on $\sigma$ and $K$) such that (see Theorem 27 in \cite{No2008}): for all $p \in (0,1)$, and $0 \leq n_1 < n_2 \leq K L(p)$,
\begin{equation} \label{eq:near_critical_arm}
C_1 \pi_{\sigma}(n_1, n_2) \leq \PP_p \big( \arm_{\sigma}(\Ann_{n_1, n_2}) \big) \leq C_2 \pi_{\sigma}(n_1, n_2).
\end{equation}

\item \emph{Near-critical formulas for $\theta$ and $L$.} We have (see Theorem 2 in \cite{Ke1987}, or (7.25) in \cite{No2008})
\begin{equation} \label{eq:equiv_theta}
\theta(p) \asymp \pi_1(L(p)) \quad \text{as $p \searrow p_c$,}
\end{equation}
and (see (4.5) in \cite{Ke1987}, or Proposition 34 in \cite{No2008})
\begin{equation} \label{eq:equiv_L}
\big| p - p_c \big| L(p)^2 \pi_4 \big( L(p) \big) \asymp 1 \quad \text{as $p \to p_c$.}
\end{equation}

\item \emph{A-priori bounds on $1$-arm events.} There exist $C_1$, $C_2$, $\beta > 0$ such that we have the following estimates. For all $p \in (0,1)$, and $0 \leq n_1 < n_2 \leq L(p)$,
\begin{equation} \label{eq:1arm}
C_1 \bigg( \frac{n_1}{n_2} \bigg)^{\frac{1}{2}} \leq \PP_p \big( \arm_1(\Ann_{n_1, n_2}) \big) \leq C_2 \bigg( \frac{n_1}{n_2} \bigg)^{\beta}
\end{equation}
(the lower bound follows from the van den Berg-Kesten inequality, while the upper bound is an immediate consequence of \eqref{eq:RSW}).

\item \emph{A-priori lower bound on $4$-arm events.} There exist $C$, $\beta > 0$ such that the following lower bound holds. For all $p \in (0,1)$, and $0 \leq n_1 < n_2 \leq L(p)$,
\begin{equation} \label{eq:4arm}
\PP_p \big( \arm_4(\Ann_{n_1, n_2}) \big) \geq C \bigg( \frac{n_1}{n_2} \bigg)^{2 - \beta}
\end{equation}
(this uses the fact that the ``universal'' arm exponent for $\arm_5$ is equal to $2$, see e.g. Theorem~24~(3) in \cite{No2008}, combined with \eqref{eq:1arm}).

\item \emph{A-priori upper bound on $6$-arm events.} There exist $C$, $\beta > 0$ such that the following upper bound is satisfied. For all $p \in (0,1)$, and $0 \leq n_1 < n_2 \leq L(p)$,
\begin{equation} \label{eq:6arm}
\PP_p \big( \arm_6(\Ann_{n_1, n_2}) \big) \leq C \bigg( \frac{n_1}{n_2} \bigg)^{2 + \beta}
\end{equation}
(similarly to \eqref{eq:4arm}, this uses the arm exponent for $\arm_5$ and \eqref{eq:1arm}).

\item \emph{Volume estimates.} Consider a sequence of integers $n_k \geq 1$, with $n_k \to \infty$ as $k \to \infty$, and $p_k \in (p_c,1)$ ($k \geq 1$). If they satisfy $L(p_k) \ll n_k$ as $k \to \infty$, then (see Theorem~3.2 in \cite{BCKS2001}):
\begin{equation} \label{eq:largest_cluster}
\text{for all $\ve > 0$,} \quad \PP_{p_k} \bigg( \frac{\big| \lclus_{\Ball_{n_k}} \big|}{\big| \Ball_{n_k} \big| \theta(p_k)} \notin (1 - \ve, 1 + \ve) \bigg) \stackrel[k \to \infty]{}{\longrightarrow} 0,
\end{equation}
where $|\lclus_{\Ball_{n_k}}|$ denotes the volume of the largest occupied cluster in $\Ball_{n_k}$.

\end{enumerate}

Finally, we use later the following notations in order to state results with some uniformity in the parameter $p$. For any $p > p_c$ and $K \geq 1$, we introduce $\dul{p}^{(K)} \geq \dol{p}^{(K)} > p_c$ defined via the relations
\begin{equation} \label{eq:def_pK}
L \big( \dul{p}^{(K)} \big) = \frac{1}{K} \cdot L(p) \quad \text{and} \quad L \big( \dol{p}^{(K)} \big) = K \cdot L(p).
\end{equation}

\subsection{Further results: volumes of clusters}

We now state two useful results about the volume of the largest connected component in a box. The first one is an estimate on quantiles which follows from classical reasonings, in particular the arguments in \cite{BCKS2001} leading to \eqref{eq:largest_cluster}. In addition, it turns out to be convenient to require the occurrence of the event $\circuitarm ( \Ann_{\frac{1}{2} n, n} \, | \, \lclus_{\Ball_n} )$ that $\lclus_{\Ball_n}$ contains both an occupied crossing (i.e. an arm) and an occupied circuit in $\Ann_{\frac{1}{2} n, n}$ (see e.g. Lemma~2.1 in \cite{LN2021}).

\begin{lemma} \label{lem:largest_cluster_quantiles}
For all $\ve > 0$, there exist $\ul{x}$, $\ol{x}$, $X > 0$ (depending only on $\ve$) such that: for all $p > p_c$, and $n \geq X L(p)$,
\begin{equation} \label{eq:largest_cluster_quantiles}
\PP_p \bigg( \bigg\{ \frac{\big| \lclus_{\Ball_n} \big|}{n^2 \theta(p)} \in (\ul{x}, \ol{x}) \bigg\} \cap \circuitarm \Big( \Ann_{\frac{1}{2} n, n} \, \big| \, \lclus_{\Ball_n} \Big) \bigg) \geq 1 - \ve.
\end{equation}
\end{lemma}

We will also need a more precise upper bound for the existence of an abnormally large cluster, stated as Lemma~4.4 in \cite{BKN2018} (where it is obtained quite directly from Proposition~4.3 (iii) of \cite{BCKS2001}).

\begin{lemma} \label{lem:largest_cluster_exp}
There exist $C_1$, $C_2$, $X > 0$ such that: for all $p > p_c$, $n \geq L(p)$, and $x \geq X$,
\begin{equation} \label{eq:largest_cluster_exp}
\PP_p \big( \big| \lclus_{\Ball_n} \big| \geq x n^2 \theta(p) \big) \leq C_1 e^{- C_2 x \frac{n^2}{L(p)^2}}.
\end{equation}
\end{lemma}

For $\kappa$ and $n \geq 1$, we denote by $\net(n,\kappa)$ the event that in each of the (horizontal and vertical) rectangles
$$\frac{\kappa}{2} \big( [-4, 4] \times [-1, 1] + (6i, 6j-3) \big) \quad \text{and} \quad \frac{\kappa}{2} \big( [-1, 1] \times [-4, 4] + (6i-3, 6j) \big)$$
($i$, $j$ integers) that intersect the box $\Ball_n$, there exists an occupied crossing in the ``difficult'' direction. These crossings together imply the existence of an occupied connected subset $\net$ of $\Ball_n$, called \emph{net with mesh $\kappa$}, having the property that all the connected components of its complement in $\Ball_n$ have a diameter at most $4 \kappa$ (so in particular, this is the case for all occupied and vacant connected components in $\Ball_n$, other than the occupied cluster of $\net$). In addition, we observe that $\net(n,\kappa)$ does not depend on the vertices in $\Ball_{\kappa}$. Lemma~\ref{lem:largest_cluster_exp} will be used in combination with the following result, which is an immediate consequence of \eqref{eq:exp_decay} (see e.g. Lemma~2.2 in \cite{BN2018}).

\begin{lemma} \label{lem:exist_net}
There exist $C_1$, $C_2 > 0$ such that: for all $p > p_c$, and $n \geq \kappa \geq 1$,
\begin{equation} \label{eq:exist_net}
\PP_p \big( \net(n,\kappa) \big) \geq 1 - C_1 \Big( \frac{n}{\kappa} \Big)^2 e^{-C_2 \frac{\kappa}{L(p)}}.
\end{equation}
\end{lemma}

\section{Geometric considerations on percolation holes} \label{sec:holes}

In this section, we discuss some geometric properties of a key object, called ``percolation hole''. More precisely, we first define two types of such holes in Section~\ref{sec:def_holes}, weak and strong holes, and we state some of their elementary properties. We then turn, in Section~\ref{sec:geom_approx}, to their approximations on a grid, which play a key role later in our proofs. Finally, we comment on some of the geometric difficulties in Section~\ref{sec:geom_discussion}, in particular coming from the fact that weak holes possess ``dangling ends'', and are thus not well-approximable on a grid.

\subsection{Definitions: weak and strong holes} \label{sec:def_holes}

We start with some geometric properties of the infinite occupied cluster, or, rather, of the connected components of $\TT$ which are left when one removes it. We know that in the supercritical regime $p > p_c$, there exists a.s. a unique infinite occupied cluster $\Cinf$, which strongly disconnects the lattice: its complement contains only finite connected components, that we call ``holes''. In other words, any given vertex either belongs to $\Cinf$, or it is surrounded by $\Cinf$ and belongs to a hole.

More precisely, we define the following two notions of ``holes'', shown on Figure~\ref{fig:holes}. First, the \emph{weak hole} of the origin $\holew$ is the connected component containing $0$ in $V \setminus \Cinf$. The \emph{strong hole} $\holes$ is the component of $0$ in $V \setminus (\Cinf \cup \dout \Cinf)$, i.e. when we also remove the (vacant) neighbors of $\Cinf$. Note that $\holew = \holes = V$ if $\Cinf = \emptyset$, that is, if there is no infinite occupied cluster. By convention, we set $\holew = \emptyset$ (resp. $\holes = \emptyset$) when $0 \in \Cinf$ (resp. $0 \in \Cinf \cup \dout \Cinf$). Clearly $\holes \subseteq \holew$, and they are both simply connected (see Figure~\ref{fig:holes} for an illustration). We sometimes use the terminology \emph{bottleneck sites} (in $\holew$) for vertices $v \in \holew$ such that $\holew \setminus \{v\}$ is no longer connected (note that necessarily, $v \in \din \holew$).

\begin{figure}[t]
\begin{center}

\includegraphics[width=.75\textwidth]{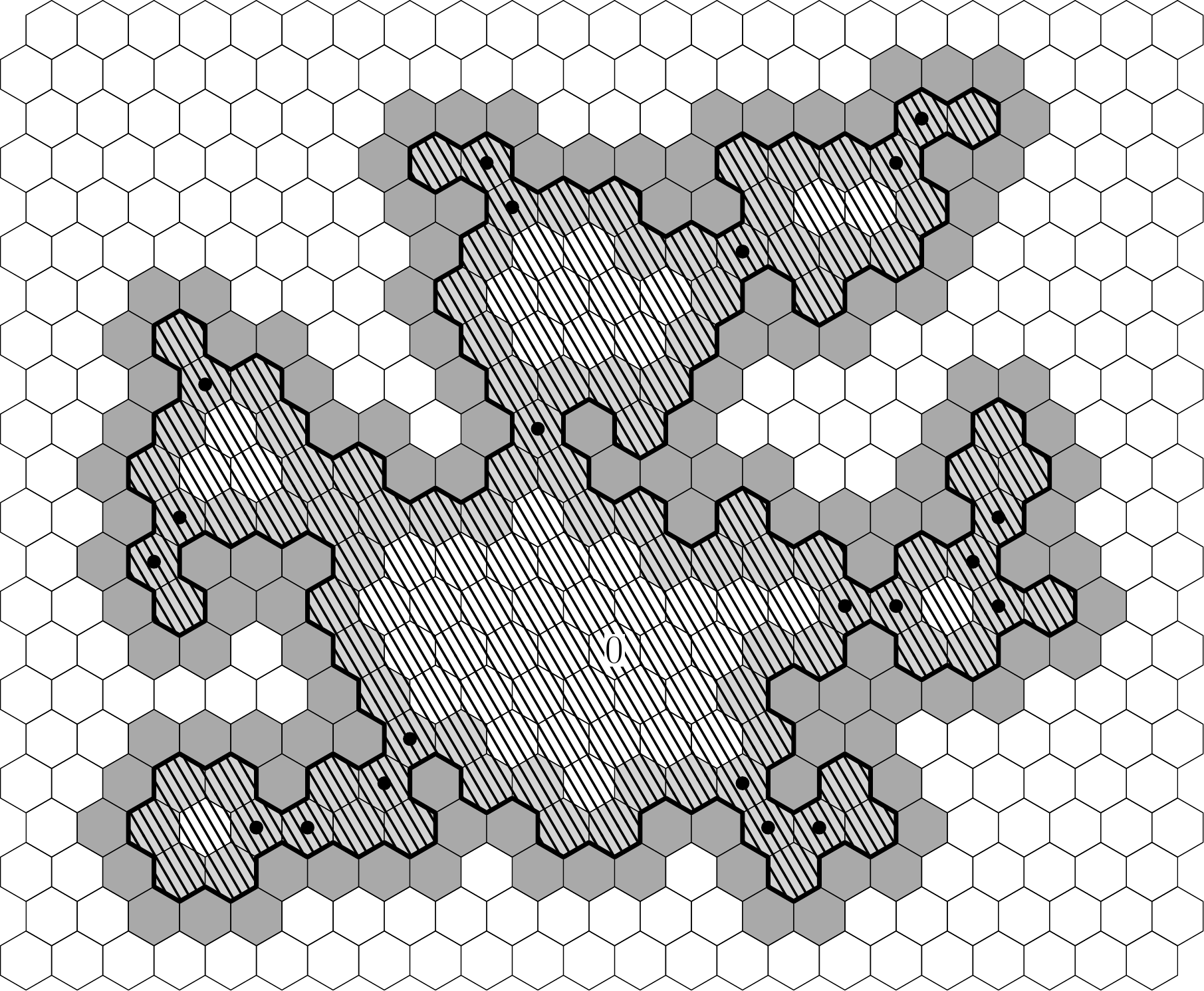}
\caption{\label{fig:holes} On this figure, site percolation on $\TT$ is depicted, as usual, as face percolation on the dual hexagonal lattice. The shaded region is the weak hole $\holew$ of the origin. The set $\dout \holew$, in dark gray, consists of occupied vertices belonging to $\Cinf$, while the vertices in $\din \holew$ ($\subseteq \dout \Cinf$), in light gray, are vacant. The strong hole $\holes$ of $0$ is the connected component of white sites containing this vertex. The vertices marked with a disk are bottleneck sites in $\holew$.}

\end{center}
\end{figure}

The weak hole is exactly the region around the origin left unaffected when the infinite cluster burns, while the strong hole contains the vertices about which the data of $\Cinf$ does not give any information, i.e. which are left untouched when $\Cinf$ is explored ``starting from $\infty$''. These objects play an important role in the proofs of Theorems~\ref{thm:main_result} and \ref{thm:main_result_finite}, when studying the successive burnings of clusters around $0$. More specifically, we follow closely the strong hole of $0$, but weak holes are not used explicitly, even though they may seem to be a more natural object at first sight.

\begin{remark} \label{rem:arms_holes}
Note that the vertices of $\din \holew$ and $\dout \holew$ are all vacant and all occupied, respectively. Hence, from any $v \in \din \holew \cup \dout \holew$, one can find two arms (both starting from a neighbor of $v$): one vacant and one occupied, extracted from $\din \holew$ and $\dout \holew$, respectively. On the other hand, $\dout \holes$ is a vacant circuit, and each of its vertices $v$ has an occupied neighbor $v'$ belonging to $\Cinf$. One can thus construct three arms starting from neighbors of $v$, two vacant arms (obtained by following $\dout \holes$ in both directions), and one occupied arm (using the path connecting $v'$ to $\infty$). The existence of three such arms along the boundary of $\holes$ was used in a crucial way in \cite{BKN2018} (see Remark~3.2 there).
\end{remark}

Combined with \eqref{eq:1arm} and Markov's inequality, the existence of arms as explained in Remark~\ref{rem:arms_holes} can be used to prove that the boundaries of $\holew$ and $\holes$ both have a negligible size. The following a-priori bounds on $\holew$ and $\holes$ will be useful later.

\begin{lemma} \label{lem:bound_holes}
There exist universal constants $c$, $\kappa > 0$ such that the following estimates hold for all $p > p_c$.
\begin{enumerate}[(i)]
\item For all $\lambda \geq 1$,
\begin{equation} \label{eq:bound_holes1}
\PP_p \big( \holew \subseteq \Ball_{\lambda L(p)} \big) \geq 1 - c \, e^{-\kappa \lambda}.
\end{equation}

\item For all $\lambda \leq 1$,
\begin{equation} \label{eq:bound_holes2}
\PP_p \big( \Ball_{\lambda L(p)} \subseteq \holes \big) \geq 1 - c \, \pi_1 \big( \lambda L(p), L(p) \big).
\end{equation}
\end{enumerate}
\end{lemma}

\begin{proof}[Proof of Lemma~\ref{lem:bound_holes}]
These two bounds follow quite directly from the standard properties of two-dimensional percolation recalled in Section~\ref{sec:near_critical_perc}, in particular \eqref{eq:RSW}, \eqref{eq:exp_decay}, \eqref{eq:extendability}, and \eqref{eq:near_critical_arm}. We refer the reader to Lemma~3.3 in \cite{BKN2018} for more details (note that the notion of holes considered in \cite{BKN2018} correspond to what we call strong holes, but the proof of \eqref{eq:bound_holes1} in this paper also applies to weak holes).
\end{proof}

Later, Lemma~\ref{lem:bound_holes} is often used through the following property, which is an immediate consequence of it (and \eqref{eq:1arm}). For all $\ve > 0$, there exist $\ol{\lambda} > \ul{\lambda} > 0$ (depending only on $\ve$) such that: for all $p > p_c$ close enough to $p_c$ (in terms of $\ve$),
\begin{equation} \label{eq:holes_apriori}
\PP_p \big( \Ball_{\ul{\lambda} L(p)} \subseteq \holes \subseteq \holew \subseteq \Ball_{\ol{\lambda} L(p)} \big) \geq 1 - \ve.
\end{equation}

We also consider analogous notions of holes in any simply connected domain $\Lambda \subseteq V$ containing $0$, obtained by replacing $\Cinf$ by the set of vertices connected to the boundary of $\Lambda$. More precisely, we introduce
\begin{equation} \label{eq:def_hole_domain}
\cluster_{\din \Lambda} := \big\{ v \in \Lambda \: : \: v \lra \din \Lambda \text{ in } \Lambda \big\},
\end{equation}
and then the weak and strong holes of the origin in $\Lambda$, denoted by $\holew_{\Lambda}$ and $\holes_{\Lambda}$, and defined as the connected component containing $0$ in $\Lambda \setminus \cluster_{\din \Lambda}$ and in $\Lambda \setminus (\cluster_{\din \Lambda} \cup \dout \cluster_{\din \Lambda} \cup \din \Lambda)$, respectively (each of them possibly taken to be $\emptyset$).

\begin{lemma} \label{lem:hole_domain}
There exist universal constants $c$, $\kappa > 0$ such that the following holds. For all $p > p_c$, and $\lambda \geq 1$, if $\Ball_{\lambda L(p)} \subseteq \Lambda$, then
\begin{equation}
\PP_p \big( \holew = \holew_{\Lambda} \text{ and } \holes = \holes_{\Lambda} \big) \geq 1 - c \, e^{-\kappa \lambda}.
\end{equation}
\end{lemma}

\begin{proof}[Proof of Lemma~\ref{lem:hole_domain}]
This follows immediately from \eqref{eq:bound_holes1}, see Remark~4.3 in \cite{BKN2018}.
\end{proof}

\subsection{Geometry of holes: bottlenecks and approximability} \label{sec:geom_approx}

Several regularity properties of the strong hole $\holes$, near criticality, were established in Section~3 of \cite{BKN2018}. They played a crucial role in the subsequent proofs in this paper, where we analyzed the clusters freezing successively around the origin, by following the holes that they created. These properties turn out to be useful in the present paper as well, and we remind them now. But before that, we need to give a few definitions.

The \emph{$\ell$-grid} is the square grid with side length $\ell > 0$, subdividing the plane into the squares $S^{(\ell)}_{i_1,i_2} = \ell \cdot ((i_1, i_2) + [-\frac{1}{2},\frac{1}{2}]^2)$, with $i_1$, $i_2 \in \ZZ$, that we call \emph{$\ell$-blocks}. We can talk about connected components of $\ell$-blocks, each $\ell$-block $S^{(\ell)}_{i_1,i_2}$ having four neighbors (namely, $S^{(\ell)}_{i_1 \pm 1, i_2}$ and $S^{(\ell)}_{i_1, i_2 \pm 1}$). For any simply connected domain $\Lambda \subseteq V$, we introduce its outer and inner approximations on the $\ell$-grid. For this purpose, we see $\Lambda$ as a subset of $\CC$ by drawing a hexagon with radius $\frac{1}{\sqrt{3}}$ around each vertex $v \in \Lambda$ (as on Figure~\ref{fig:holes}).
\begin{enumerate}[(i)]
\item The \emph{outer approximation} $\Dout{\Lambda}{\ell}$ of $\Lambda$ is the union of all $S^{(\ell)}_{i_1,i_2}$ ($i_1$, $i_2 \in \ZZ$) such that $S^{(\ell)}_{i_1,i_2} \cap \Lambda \neq \emptyset$.

\item Its \emph{inner approximation} (seen from $0$) $\Dint{\Lambda}{\ell}$ is obtained by considering all $\ell$-blocks $S^{(\ell)}_{i_1,i_2} \subseteq \Lambda$, and grouping them into connected components (as explained above): $\Dint{\Lambda}{\ell}$ is the union of all $\ell$-blocks in the connected component containing $S^{(\ell)}_{0,0} = [-\frac{\ell}{2}, \frac{\ell}{2}]^2$. If $S^{(\ell)}_{0,0} \nsubseteq \Lambda$ (we could even have $0 \notin \Lambda$), we set $\Dint{\Lambda}{\ell} = \emptyset$.
\end{enumerate}
Clearly, $\Dint{\Lambda}{\ell} \subseteq \Lambda \subseteq \Dout{\Lambda}{\ell}$.

A bounded, simply connected $\Lambda \subseteq \CC$ is said to be \emph{$(\ell, \eta)$-approximable} (for some $\ell > 0$, $\eta \in (0,1)$) if $|\Dout{\Lambda}{\ell} \setminus \Dint{\Lambda}{\ell} | \leq \eta |\Lambda|$. Note that this assumption implies in particular that $|\Dint{\Lambda}{\ell}| \geq (1 - \eta) |\Lambda|$ and $|\Dout{\Lambda}{\ell}| \leq (1 + \eta) |\Lambda|$. The first regularity property that we are going to use says that $\holes$ can be arbitrarily well approximated on the $(\delta L(p))$-grid, by choosing $\delta > 0$ small enough, uniformly in $p > p_c$ (and with high probability). Its proof is based on the $6$-arm estimate \eqref{eq:6arm}, to rule out the existence of ``macroscopic'' bottlenecks.

\begin{lemma}[Lemma~3.7 in \cite{BKN2018}] \label{lem:approx_hole}
For all $\ve$, $\eta > 0$, there exist $\delta$, $\ol{\delta} > 0$ (depending only on $\ve$, $\eta$) such that: for all $p_c < p < p_c + \ol{\delta}$,
\begin{equation}
\PP_p \big( \holes \text{ is } (\delta L(p), \eta)\text{-approximable} \big) \geq 1 - \ve.
\end{equation}
\end{lemma}

For $\lambda > 0$, the \emph{$\lambda$-shrinking} of a domain $\Lambda \subseteq \CC$ is defined as
\begin{equation} \label{eq:def_shrink}
[\Lambda]_{(\lambda)} := \{ z \in \Lambda \: : \: d(z, \CC \setminus \Lambda) \geq \lambda \},
\end{equation}
where $d$ is the distance induced by $\|.\|$. We use this notion only for domains $\Lambda$ obtained as connected components of $\ell$-blocks. In this case, observe that $[\Lambda]_{(\ve \ell)}$ is connected for all $\ve \in (0, \frac{1}{2})$, and that we have $(1 - 4 \ve) |\Lambda| \leq \big| [\Lambda]_{(\ve \ell)} \big| \leq |\Lambda|$ for all $\ve \in (0, \frac{1}{4})$, uniformly in $\ell$. Later, we also use the \emph{$\lambda$-thickening} of a domain $\Lambda \subseteq \CC$, which is
\begin{equation} \label{eq:def_thick}
\big[ \Lambda \big]^{(\lambda)} := \Lambda \cup \{ z \in \CC \setminus \Lambda \: : \: d(z, \Lambda) \leq \lambda \}.
\end{equation}
In the case when $\Lambda$ is a union of $\ell$-blocks, we have: for all $\ve \leq 1$,
\begin{equation} \label{eq:thick_bound}
|\Lambda| \leq \big| \big[ \Lambda \big]^{(\ve \ell)} \big| \leq (1+8 \ve) |\Lambda|.
\end{equation}

With these definitions at hand, we can state a continuity property for $\holes$ which is a key ingredient in \cite{BKN2018}, to compare the strong holes at various times. In this result, $\PP$ refers to the standard coupling of Bernoulli percolation, $\holes(p)$ denoting the strong hole at parameter $p$.

\begin{lemma}[Lemma~3.8 in \cite{BKN2018}] \label{lem:cont_hole}
For all $\ve > 0$, there exist $\delta$, $\ol{\delta} > 0$ (depending only on $\ve$) such that: for all $p_c < p < p' < p_c + \ol{\delta}$ with $\frac{L(p)}{L(p')} \leq 1 + \ol{\delta}$, we have
\begin{equation}
\PP \Big( \Big[ \Dint{\holes(p)}{\delta L(p)} \Big]_{(\ve \delta L(p))} \subseteq \holes(p') \subseteq \holes(p) \Big) \geq 1 - \ve.
\end{equation}
\end{lemma}

Note that by definition of the coupling, the second inclusion is obvious, and only the first one requires a proof. This result implies directly a continuity result for the volume of $\holes$, see Corollary~3.9 in \cite{BKN2018}.


In \cite{BKN2018}, Lemmas~\ref{lem:approx_hole} and \ref{lem:cont_hole} are in particular used in combination with an estimate on the largest cluster in a suitably regular domain, see Lemma~4.1 in that paper. In Section~\ref{sec:dec_iteration} we state and use a uniform version of such an estimate, see Lemma~\ref{lem:approx_BCKS_unif}.

%
%


The following strengthened version of Lemma~\ref{lem:approx_hole} will turn out to be useful toward the end of Section~\ref{sec:dec_proof}. This improved result provides a more uniform control on the percolation holes which appear successively around the origin. In order to state it, we use the notations $\dul{p}^{(K)}$, $\dol{p}^{(K)}$ from \eqref{eq:def_pK} ($p > p_c$, $K \geq 1$).

\begin{lemma} \label{lem:approx_hole_unif}
For all $\ve$, $\eta > 0$, and $K \geq 1$, there exist $\delta = \delta(\ve, \eta, K) > 0$ and $\ol{\delta} = \ol{\delta}(\ve, K) > 0$ such that: for all $p_c < p < p_c + \ol{\delta}$,
\begin{equation}
\PP \big( \text{for all } \tilde{p} \in [\dol{p}^{(K)}, \dul{p}^{(K)}], \: \holes(\tilde{p}) \text{ is } (\delta L(p), \eta)\text{-approximable} \big) \geq 1 - \ve.
\end{equation}
\end{lemma}

The proof of Lemma~\ref{lem:approx_hole_unif} follows roughly the same lines as the proof of Lemma~\ref{lem:approx_hole}. However, some non-trivial modifications are required, so we decided to provide a self-contained proof in Appendix.

\subsection{Comments on some geometric issues} \label{sec:geom_discussion}

Although this section is not needed in the proofs of our main results, it is meant to illustrate some geometric difficulties pertaining to weak holes. As we mentioned in Section~\ref{sec:def_holes}, intuitively, the weak holes seem to appear more naturally than the strong holes in the analysis of forest fires (as well as of frozen percolation with modified boundary rules, on which we briefly comment in Section~\ref{sec:FP_modified}).

The weak hole is not as well-behaved as the strong hole: macroscopic ``dangling ends'' may exist, so that contrary to the strong hole $\holes$, the weak hole $\holew$ is, in general, \emph{not} approximable. For example, using a construction similar in spirit to that of Figure~4 in \cite{BN2017b} (based on Lemma~2.4 in that paper, combined with \eqref{eq:RSW}), one could prove the following. There exists $c > 0$ small enough such that for all $\delta > 0$, for all $p > p_c$ sufficiently close to $p_c$, we have: with probability at least $c$, there is a region of the weak hole which has a volume at least $L(p)^2$, and which is separated from $0$ by at least $L(p)^{\frac{3}{4} - \delta}$ bottleneck sites. We provide a sketch of proof on Figure~\ref{fig:non_approximable}, but we do not give all details since this result is not used later in the paper.

\begin{figure}[t]
\begin{center}

\includegraphics[width=.85\textwidth]{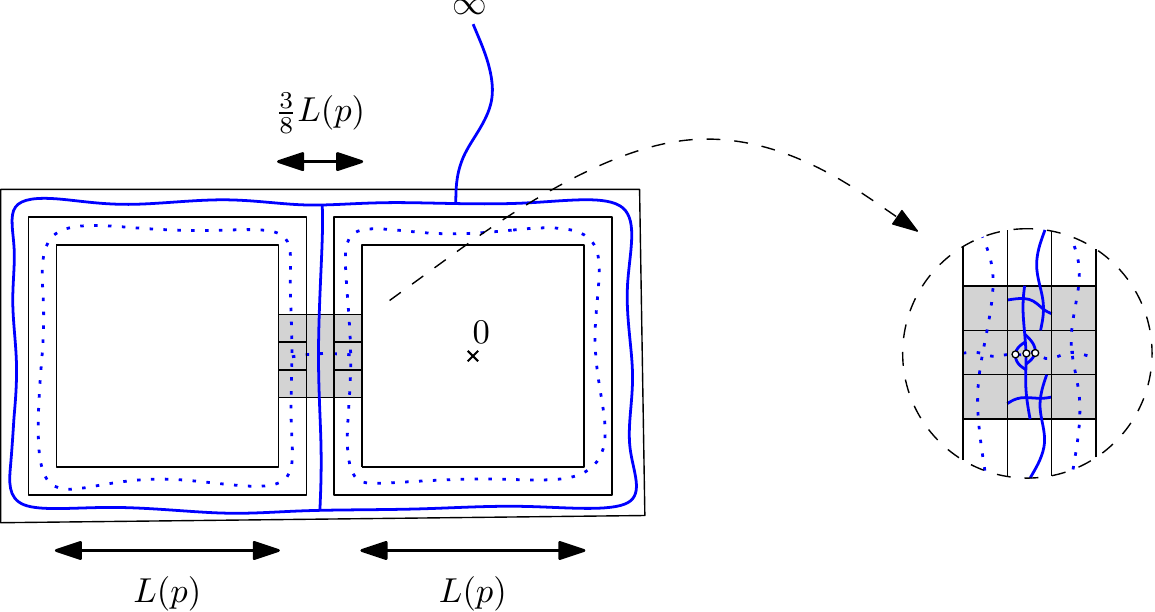}
\caption{\label{fig:non_approximable} The vacant paths are $p$-occupied, while the dotted ones are $p$-vacant. The grey region contains many (of order $L(p)^2 \pi_4(L(p)) = L(p)^{\frac{3}{4} + o(1)}$ as $p \searrow p_c$) bottleneck sites, each one separating the two ``chambers'', which both have a side length roughly $L(p)$.}

\end{center}
\end{figure}

\section{Frozen percolation and forest fires} \label{sec:FP_FF}

We now turn to the forest fire process studied in this paper. In Section~\ref{sec:processes}, we define it precisely, as well the volume-frozen percolation process to which it is compared later. We then introduce in Section~\ref{sec:map_next} the transformation $\next_N$, which will allow us to follow the successive burnings inside a finite domain.

\subsection{Definition of the processes} \label{sec:processes}

Consider a graph $G = (V_G, E_G)$, which is either the whole triangular lattice $\TT = (V_{\TT}, E_{\TT})$, or a finite simply connected subgraph of it. We consider processes on $G$, of the form $\eta = (\eta(t))_{t \geq 0}$, where for each $t \geq 0$, $\eta(t) = (\eta_v(t))_{v \in V_G} \in \{0, 1\}^{V_G}$ ($=: \Omega_G$) is a vertex configuration. For each of them, all vertices $v \in V_G$ are vacant at time $t = 0$ (i.e. $\eta_v(0) = 0$), and they can then become occupied at later times $t$ (that is, $\eta_v(t) = 1$). The occupied cluster of $v \in V_G$ at a time $t \geq 0$, i.e. in the configuration $\eta(t)$, is denoted by $\cluster_t(v)$.

Our processes will be compared to the pure birth process, for which each $v \in V_G$ becomes occupied at rate $1$, independently of the other vertices, and then simply remains in this state. Clearly, at each $t \geq 0$, the configuration $\eta(t)$ is distributed as Bernoulli site percolation with parameter
\begin{equation} \label{eq:relation_p_t}
p(t) := 1 - e^{-t}.
\end{equation}
We let
\begin{equation}
t_c := - \log(1 - p_c) = \log 2
\end{equation}
be the \emph{critical time} such that $p(t_c) = p_c = \frac{1}{2}$, after which an infinite occupied cluster arises.

In the following, we write (with a slight abuse of notation) $\PP_t = \PP_{p(t)}$ when referring to the pure birth process, and we use the shorthand notations $\theta(t) = \theta(p(t))$ and $L(t) = L(p(t))$ (for all $t \in [0,\infty)$). We also set $\theta(\infty) = \lim_{t \to \infty} \theta(t) = 1$ and $L(\infty) = \lim_{t \to \infty} L(t) = 0$.

\subsubsection*{Volume-frozen percolation}

First, volume-frozen percolation with parameter $N \geq 1$ is obtained by modifying the dynamics of the pure birth process in the following way. Again, we let each vertex $v \in V_G$ become occupied at rate $1$, but now only if none of its neighbors belongs to an occupied cluster with a volume $\geq N$, i.e. containing at least $N$ vertices. Otherwise, $v$ is kept vacant.

In other words, each occupied cluster $\cluster$ keeps ``absorbing'' new vertices until $|\cluster| \geq N$. Note that the latter may never happen, but if it does, the growth of $\cluster$ is stopped, and we say that this cluster \emph{freezes}. Its vertices, which are all occupied, are called \emph{frozen}. The vertices along its outer boundary $\dout \cluster$ are vacant, and they then remain in this state forever. Recall that we write $\PPP_N$ for the corresponding probability measure (or $\PPP_N^{(G)}$ when the graph that we use needs to be specified).

\subsubsection*{$N$-parameter forest fire process}

Next, we introduce the forest fire process which is the main focus of the present paper. It is defined in a similar way as volume-frozen percolation. But when a cluster reaches volume $N$, its vertices instantaneously become vacant. Such vertices are then allowed to ``recover'', i.e. become occupied again, at later times (still, at rate $1$). We use the notation $\PP_N^{(G)}$ (or simply $\PP_N$) for the probability measure governing this process.

\subsubsection*{Comments}

Observe that each of the two processes introduced above can be represented as a finite-range interacting particle system. Indeed, every vertex $v$ interacts only with the vertices in $\Ball_N(v)$. These processes can thus be constructed thanks to the general theory of interacting particle systems (we refer the reader to  \cite{Li2005} for a classical reference). In the remainder of the paper, all the processes studied are considered to be coupled by using the same family of Poisson processes $(\varsigma_v(t))_{t \in [0,\infty)}$, $v \in V_{\TT}$, of births (each with intensity $1$). We denote by $(\tau_v^i)_{i \geq 1}$, $v \in V_{\TT}$, the corresponding times (in increasing order).

\subsection{Successive burnings} \label{sec:map_next}

We now introduce a transformation $\next_N$ which will be useful to describe the $N$-parameter forest-fire process (and volume-frozen percolation with the same parameter $N$). For this purpose, we let
$$c_{\TT} := \frac{2}{\sqrt{3}}$$
be such that $|\Ball_n| \sim c_{\TT} \, n^2$ as $n \to \infty$. Roughly speaking, this map is used to predict (approximately) when the first burning occurs in a finite domain, in terms of its diameter. It was already considered, up to some minor changes, in our earlier work \cite{BKN2018} (together with Kiss).

Let $N \geq 1$. The function $\theta$ is continuous and strictly increasing on $[t_c,\infty)$, so it is a bijection from $[t_c,\infty)$ to $[0,1)$. In particular, this implies that for all $R > c_{\TT}^{-\frac{1}{2}} \sqrt{N}$, there exists a unique $t \in (t_c,\infty)$ such that
\begin{equation}
c_{\TT} R^2 \theta(t) = N.
\end{equation}
We denote this value by $\next_N(R)$. We also define the transformation $t \mapsto \hat{t} = \hat{t}(t,N) > t_c$ by $\hat{t} = \next_N(L(t))$ (for values $t > t_c$ such that $L(t) > c_{\TT}^{-\frac{1}{2}} \sqrt{N}$). It satisfies
\begin{equation} \label{eq:t_that}
c_{\TT} L(t)^2 \theta(\hat{t}) = N.
\end{equation}
Note that the map $\next_N$ is strictly decreasing, while $t \mapsto \hat{t}$ is strictly increasing. Moreover, $\hat{t} \to t_c$ as $t \searrow t_c$, and $\hat{t} \to \infty$ as $t \nearrow L^{-1}(c_{\TT}^{-\frac{1}{2}} \sqrt{N})$ (where $L^{-1}$ denotes the inverse function of $L$, seen as a function of time, on $(t_c, \infty)$). Hence, $t \mapsto \hat{t}$ is a bijection from $(t_c,L^{-1}(c_{\TT}^{-\frac{1}{2}} \sqrt{N}))$ to $(t_c, \infty)$. For convenience, we set $\next_N(R) = \infty$ for $R \in [0, c_{\TT}^{-\frac{1}{2}} \sqrt{N}]$, and $\hat{t} = \infty$ for $t \geq L^{-1}(c_{\TT}^{-\frac{1}{2}} \sqrt{N})$.

From \eqref{eq:equiv_theta}, \eqref{eq:equiv_L}, and the explicit values for the exponent $\alpha_{\sigma}$ appearing in \eqref{eq:arm_exponent} (in the particular cases $\sigma = (o)$ and $\sigma = (ovov)$), we can see that
\begin{equation} \label{eq:exp_theta_L}
\theta(t) = (t-t_c)^{\frac{5}{36} + o(1)} \: (t \searrow t_c) \quad \text{and} \quad L(t) = |t-t_c|^{-\frac{4}{3} + o(1)} \: (t \to t_c).
\end{equation}
In particular, $L(t)^2 \theta(t) \to \infty$ as $t \searrow t_c$, while clearly $L(t)^2 \theta(t) \to 0$ as $t \to \infty$. We deduce that the equation $\hat{t} = t$ has at least one solution, which leads us to introduce
\begin{equation} \label{eq:def_t_inf}
t_{\infty}(N) := \sup \{ t > t_c \: : \: \hat{t} = t \} \in (t_c, L^{-1}(c_{\TT}^{-\frac{1}{2}} \sqrt{N})),
\end{equation}
and then $m_{\infty}(N) := L(t_{\infty}(N))$. From \eqref{eq:t_that} and \eqref{eq:exp_theta_L}, we have
\begin{equation} \label{eq:exp_t_m_inf}
t_{\infty}(N) = t_c + N^{- \frac{36}{91} + o(1)} \quad \text{and} \quad m_{\infty}(N) = N^{\frac{48}{91} + o(1)} \quad \text{as $N \to \infty$.}
\end{equation}
Furthermore, it follows immediately from the definition of $t_{\infty}$ that: for all $t > t_{\infty}$, $\hat{t} > t$.

The following result allows one to estimate one iteration of the transformation $\next_N$ (see Lemma~3.1 in \cite{LN2021}).

\begin{lemma}
For all $\ve > 0$, there exist $0 < C_1 < C_2$ (depending on $\ve$) such that: for all $N \geq 1$ and $2 \sqrt{N} \leq r \leq R$,
\begin{equation}
C_1 \bigg( \frac{r}{R} \bigg)^{\frac{96}{5} + \ve} \leq \frac{L(\next_N(r))}{L(\next_N(R))} \leq C_2 \bigg( \frac{r}{R} \bigg)^{\frac{96}{5} - \ve}.
\end{equation}
\end{lemma}

Using $L(\next_N(m_{\infty}(N))) = L(t_{\infty}(N)) = m_{\infty}(N)$, it implies in particular: for all $2 \sqrt{N} \leq r \leq m_{\infty}(N)$,
\begin{equation} \label{eq:ratio_r_minf}
C_1 \bigg( \frac{r}{m_{\infty}(N)} \bigg)^{\frac{96}{5} + \ve} \leq \frac{L(\next_N(r))}{m_{\infty}(N)} \leq C_2 \bigg( \frac{r}{m_{\infty}(N)} \bigg)^{\frac{96}{5} - \ve}.
\end{equation}

We also define $t_k = t_k(N)$ for each $k \geq 0$ by induction as follows, using the earlier observation that the map $t \mapsto \hat{t}$ is a bijection (for every fixed $N \geq 1$). We let
\begin{equation} \label{eq:def_tk}
t_0(N) := 2 t_c, \: \text{ and for each } k \geq 0, \: t_{k+1}(N) \text{ s.t. } \widehat{t_{k+1}(N)} = t_k(N).
\end{equation}
Note that for all $k \geq 0$ and $N \geq 1$, $t_k(N) \geq t_{\infty}(N)$ and $t_{k+1}(N) \leq t_k(N)$. In addition, we denote
\begin{equation} \label{eq:def_mk}
m_k(N) := L(t_k(N)),
\end{equation}
and we have $m_k(N) \leq m_{\infty}(N)$ and $m_{k+1}(N) \geq m_k(N)$ ($k \geq 0$, $N \geq 1$). These are (up to minor modifications) the \emph{exceptional scales} uncovered in \cite{BN2017a}, and one can easily obtain from \eqref{eq:exp_theta_L} that each of them follows a power law in $N$: $m_k(N) = N^{\delta_k + o(1)}$ as $N \to \infty$, for some explicit sequence of exponents $(\delta_k)_{k \geq 0}$ which is strictly increasing, with $\delta_0 = 0$, $\delta_1 = \frac{1}{2}$, and $\delta_k \nearrow \delta_{\infty} = \frac{48}{91}$ (the exponent appearing in \eqref{eq:exp_t_m_inf}) as $k \to \infty$.

\section{Positive recovery time for forest fires} \label{sec:positive_rec_time}

In this section, we revisit the reasonings in \cite{KMS2015}. First, we establish in Section~\ref{sec:stab_6arms} a stability result, Theorem~\ref{thm:modif_6_arms}, for some peculiar six-arm events. These events incorporate the potential presence of ``passage sites'', which roughly correspond to recovered vertices in the forest fire process, i.e. occupied vertices which have burnt, and become again occupied not so long after. Such events play a central role (directly and indirectly) in the remainder of the paper. We then explain in Section~\ref{sec:appl_6arms} that a key crossing estimate from \cite{KMS2015} can be deduced easily from Theorem~\ref{thm:modif_6_arms}, which has several classical and important consequences. Finally, in Section~\ref{sec:gen_6arms}, we present generalizations of the earlier results which will be useful later to study forest fires.

\subsection{Stability for six-arm events with passage sites} \label{sec:stab_6arms}

We now prove a stability result for a modified $6$-arm event, denoted by $\marm_6$ below. Our motivation to introduce such a somewhat exotic event is twofold. First, it makes the summation argument in \cite{KMS2015} more transparent in our opinion, which allows us later to consider adaptations which are more suitable to our setting when successive circuits surrounding the origin burn. Moreover, and maybe more importantly, we believe that this strategy of proof can be adapted to the case of forest fires processes with Poisson ignition, i.e. the original Drossel-Schwabl model, which will be the focus of a separate work.

For an annulus $A = \Ann_{n_1,n_2}(z)$ with $0 \leq n_1 < n_2$ and $z \in \CC$, recall that $\arm_6(A)$ denotes the $6$-arm event in $A$ where the arms have alternating types, i.e. prescribed by $(ovovov)$. From now on, we rather consider the $6$-arm event $\arm_{\sigma}(A)$ corresponding to $\sigma = (ovvovv)$, which has, at $p = p_c$, a probability (uniformly) of order $\pi_6(n_1, n_2)$, by color switching. In addition, as in \cite{KMS2015}, we introduce a modification of $\arm_{\sigma}(A)$, denoted by $\darm_6(A)$, where one of the two occupied arms has possibly one defect with size at most $2$, i.e. this arm is allowed to contain either one single vacant vertex, or two successive vacant vertices. We will use that
\begin{equation} \label{eq:arm_defect}
\PP_{p_c} \big( \darm_6(A) \big) \leq C \bigg( 1 + \log \bigg( \frac{n_2}{n_1 \vee 1} \bigg) \bigg) \pi_6(n_1, n_2)
\end{equation}
for some universal constant $C$ (this is an almost immediate consequence of \eqref{eq:extendability} and \eqref{eq:quasi_mult}, see Proposition~18 in \cite{No2008}).

\begin{figure}[t]
\begin{center}

\includegraphics[width=.75\textwidth]{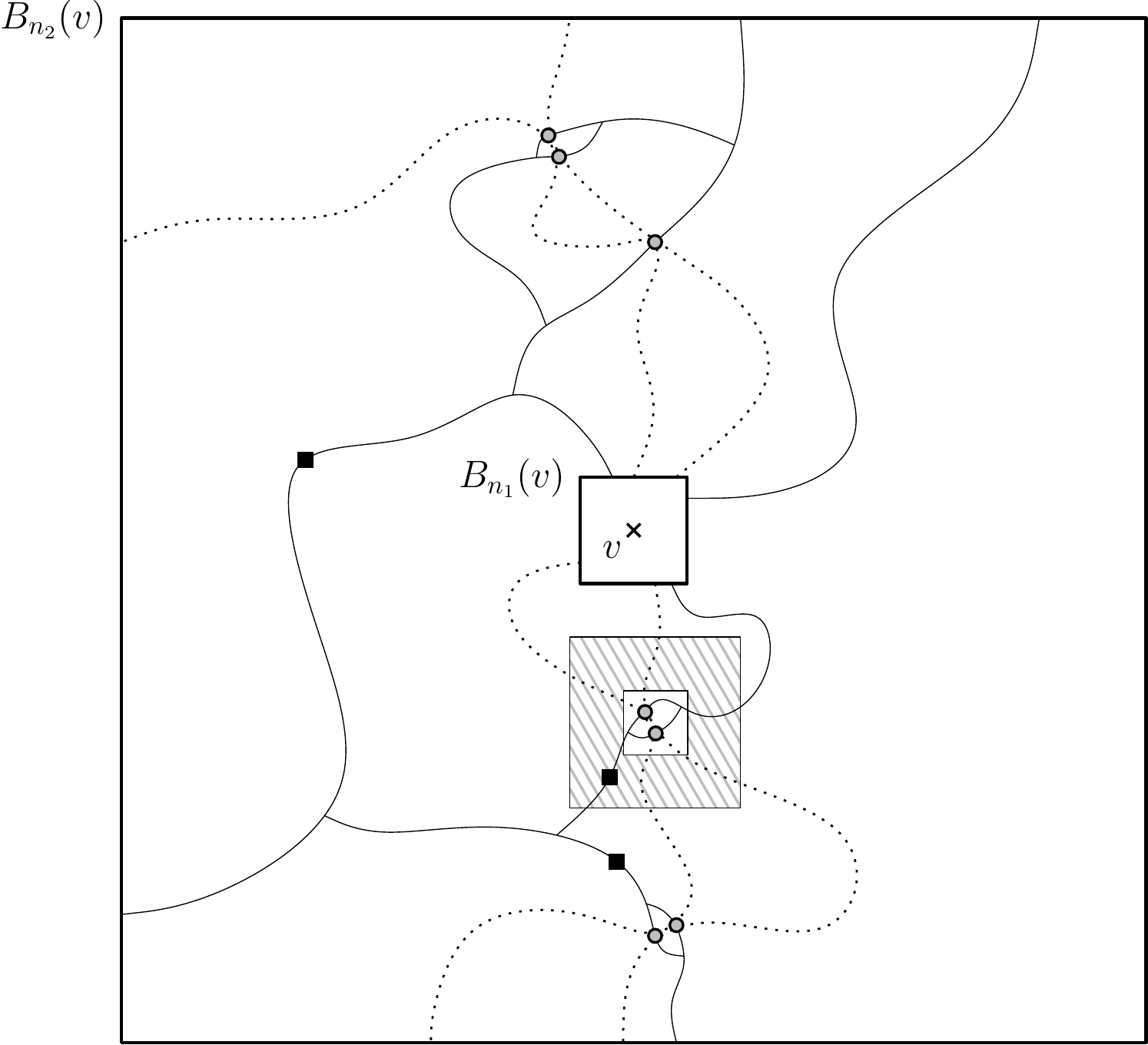}
\caption{\label{fig:six_arm_event} This figure shows a rough depiction of the $6$-arm event $\marm_6(A)$ introduced in this section, for the annulus $A = \Ann_{n_1, n_2}(v)$. Heuristically, we think of the vertices of $\calP$, marked with grey disks, as ``passage sites'', in the spirit of \cite{KMS2015}. The paths in solid and dotted line are respectively occupied and vacant, the black squares indicating defects (with size at most $2$) along the occupied arms. The annulus shaded in grey is $\Ann_{m_1, m_2}(w) \subseteq A$: it satisfies $\Ball_{m_1}(w) \cap \calP \neq \emptyset$ and $\Ann_{m_1, m_2}(w) \cap \calP = \emptyset$. In this annulus, the $6$-arm event with a defect $\darm_6(\Ann_{m_1, m_2}(w))$ should occur.}

\end{center}
\end{figure}

We now introduce a further modification of the $6$-arm event, inspired by Lemma~16(i) in \cite{KMS2015} (which is a deterministic result). We illustrate it on Figure~\ref{fig:six_arm_event}.

\begin{definition} \label{def:6_arms_modified}
Let $0 \leq n_1 < n_2$, $v \in V$, and $A := \Ann_{n_1, n_2}(v)$. The modified $6$-arm event $\marm_6(A)$ is the set of all pairs $(\omega, \sigma) \in \{0,1\}^{\Ball_{n_2}(v)} \times \{0,1\}^{\Ball_{n_2}(v)}$ such that there exists a (possibly empty) set of vertices $\calP \subseteq A$ for which the two properties below hold.
\begin{enumerate}[(i)]
\item For all $w \in \calP$, $\sigma_w = 1$.

\item For all annuli $\Ann_{m_1, m_2}(w) \subseteq A$, with $w \in \Ball_{n_2}(v)$ and $0 \leq m_1 < m_2$, such that
\begin{itemize}
\item $\Ball_{m_1}(w) \supseteq \Ball_{n_1}(v)$ or $\Ball_{m_1}(w) \cap \calP \neq \emptyset$,

\item and $\Ann_{m_1, m_2}(w) \cap \calP = \emptyset$,
\end{itemize}
we have that $\omega \in \darm_6(\Ann_{m_1, m_2}(w))$.
\end{enumerate}
\end{definition}

\begin{remark} \label{rem:6_arms_modified}
We will make use later of the following observations about the above definition, still denoting $A = \Ann_{n_1, n_2}(v)$ for some $0 \leq n_1 < n_2$ and $v \in V$.
\begin{enumerate}
\item Clearly, $\marm_6(A) \supseteq \darm_6(A)$ (by considering the particular case $\calP = \emptyset$). \label{rem:6_arms_modified1}

\item The event $\marm_6(A)$ depends only on the $\omega$- and $\sigma$-configurations in $A$ (but not on those in $\Ball_{n_1}(v)$). \label{rem:6_arms_modified2}

\item If $\marm_6(A)$ holds with respect to a certain set $\calP$, then $\marm_6(\Ann_{m_1, m_2}(w))$ holds for every $w \in \calP$ with $\Ann_{m_1, m_2}(w) \subseteq A$. \label{rem:6_arms_modified3}

\item We have $\marm_6(A) \subseteq \marm_6(\Ann_{m_1, m_2}(v))$ for all $n_1 \leq m_1 < m_2 \leq n_2$. \label{rem:6_arms_modified4}
\end{enumerate}
\end{remark}

In the following, we assume that $\sigma = (\sigma_v)_{v \in V} \sim \PP_{\deltaKMS}$ for some $\deltaKMS > 0$, and that it is independent of $\omega$ ($\sim \PP_p$). We write $\PP_{p,\deltaKMS}$ for the associated probability distribution. As noted above, $\marm_6(\Ann_{n_1, n_2}(v))$ contains $\darm_6(\Ann_{n_1, n_2}(v))$, and the following key estimate shows that it is possible to choose $\deltaKMS > 0$ small enough (independently of all parameters) so that the probabilities of these two events remain comparable.

\begin{theorem} \label{thm:modif_6_arms}
There exist universal constants $\deltaKMS > 0$ and $C > 0$ such that the following holds. For all $0 \leq n_1 < n_2$, and $v \in V$,
\begin{equation} \label{eq:modif_6_arms}
\PP_{p_c,\deltaKMS} \big( \marm_6(\Ann_{n_1, n_2}(v)) \big) \leq C \bigg( 1 + \log \bigg( \frac{n_2}{n_1 \vee 1} \bigg) \bigg) \pi_6(n_1, n_2).
\end{equation}
\end{theorem}

We establish it by induction. Our reasoning has a similar flavor as some proofs of ``stability results'' for near-critical percolation, see in particular Lemma~8.4 of \cite{GPS2018a} and Theorem~4.1 of \cite{BN2018}.

\begin{proof}[Proof of Theorem~\ref{thm:modif_6_arms}]
We can assume, without loss of generality, that $v = 0$. In order to ease notation, we also write $\PP$ for $\PP_{p_c, \deltaKMS}$ in this proof (only). First, we observe that it is sufficient to consider $n_1$ and $n_2$ of the form $2^i$ and $2^j$, respectively, for any two integers $0 \leq i < j$. The proof is by induction on $j$ and, for fixed $j$, on the difference $j-i$.

Note that for each given $M \geq 1$, we have: for all $i$ and $j$ with $j-i \leq M$, \eqref{eq:modif_6_arms} is satisfied with some suitable $C = C(M) > 0$. Here, for convenience, we take $M = 8$, and we assume that $C > C(M)$ (near the end of the proof, we will see which stronger condition on $C$ and $\deltaKMS$ is needed so that the induction works). Now, suppose that the pair $(i,j)$ with $j-i \geq M+1$ is such that \eqref{eq:modif_6_arms} holds for all pairs $(i',j')$ with $j' < j$, and for all pairs $(i',j)$ with $i' > i$. We show that \eqref{eq:modif_6_arms} also holds for the pair $(i,j)$ itself.

Trivially,
\begin{align}
\PP \big( \marm_6(\Ann_{2^i, 2^j}) \big) & \leq \PP \big( \darm_6(\Ann_{2^{i+3}, 2^{j-3}}) \big) + \PP \big( \marm_6(\Ann_{2^i, 2^j}) \setminus \darm_6(\Ann_{2^{i+3}, 2^{j-3}}) \big) \nonumber\\[1mm]
& =: (\text{Term } 1) + (\text{Term } 2). \label{eq:modif_6_arms_pf1}
\end{align}
First, \eqref{eq:arm_defect} (combined with \eqref{eq:extendability}) implies that
\begin{equation} \label{eq:modif_6_arms_pf2}
(\text{Term } 1) \leq c_0 (1+ j - i) \pi_6(2^i, 2^j),
\end{equation}
for some universal constant $c_0$. Consider now $(\text{Term } 2)$, and suppose that $(\omega, \sigma)$ belongs to the event appearing in it. Let $\calP \subseteq \Ann_{2^i, 2^j}$ be a set of vertices such that with this choice of $\calP$, the pair $(\omega, \sigma)$ satisfies the conditions in the definition of $\marm_6(\Ann_{2^i, 2^j})$. Since $\omega \notin \darm_6(\Ann_{2^{i+3}, 2^{j-3}})$, we have necessarily (see (ii) in the definition of $\marm_6$)
$$\calP \cap \Ann_{2^{i+3}, 2^{j-3}} \neq \emptyset.$$
Choose a vertex $w$ in that set (so in particular, $\sigma_w = 1$), and let $k \geq 3$ be such that $w \in \Ann_{2^k, 2^{k+1}}$. We observe that the following facts hold.
\begin{itemize}
\item The event $\marm_6(\Ann_{0, 2^{k-1}}(w))$ occurs (from Remark~\ref{rem:6_arms_modified}(\ref{rem:6_arms_modified3}), since $\Ann_{0, 2^{k-1}}(w) \subseteq \Ann_{2^i, 2^j}$), as well as $\marm_6(\Ann_{2^i, 2^{k-2}})$ and $\marm_6(\Ann_{2^{k+2}, 2^j})$ (using Remark~\ref{rem:6_arms_modified}(\ref{rem:6_arms_modified4})).

\item Moreover, these three events involve vertices in disjoint annuli, so they are independent (see Remark~\ref{rem:6_arms_modified}(\ref{rem:6_arms_modified2})). In addition, they are independent of $\{ \sigma_w = 1\}$.
\end{itemize}
Hence, we get
\begin{align}
(\text{Term } 2) & \leq \sum_{k=i+3}^{j-4} \sum_{w \in \Ann_{2^k, 2^{k+1}}} \PP \big( \{ \sigma_w = 1 \} \cap \marm_6(\Ann_{0, 2^{k-1}}(w)) \cap \marm_6(\Ann_{2^i, 2^{k-2}}) \cap \marm_6(\Ann_{2^{k+2}, 2^j}) \big) \nonumber\\
& = \sum_{k=i+3}^{j-4} \sum_{w \in \Ann_{2^k, 2^{k+1}}} \PP \big( \sigma_w = 1 \big) \cdot \PP \big( \marm_6(\Ann_{0, 2^{k-1}}(w)) \big) \cdot \PP \big( \marm_6(\Ann_{2^i, 2^{k-2}}) \big) \cdot \PP \big( \marm_6(\Ann_{2^{k+2}, 2^j}) \big) \nonumber\\
& \leq \sum_{k=i+3}^{j-4} c_1 2^{2k} \cdot \deltaKMS \cdot \PP \big( \marm_6(\Ann_{0, 2^{k-1}}) \big) \cdot \PP \big( \marm_6(\Ann_{2^i, 2^{k-2}}) \big) \cdot \PP \big( \marm_6(\Ann_{2^{k+2}, 2^j}) \big), \label{eq:modif_6_arms_rem}
\end{align}
using the independence of the events for the equality, and then $|\Ann_{2^k, 2^{k+1}}| \leq c_1 2^{2k}$ for the second inequality, where $c_1$ is universal (as well as translation invariance). We can now apply the induction hypothesis to each of the three probabilities in the product inside this sum, which gives
\begin{align}
(\text{Term } 2) & \leq c_1 \deltaKMS \sum_{k=i+3}^{j-4} 2^{2k} \cdot C \big( 1+ \log (2^{k-1}) \big) \pi_6(0, 2^{k-1}) \cdot C \bigg( 1 + \log \bigg( \frac{2^{k-2}}{2^i} \bigg) \bigg) \pi_6(2^i, 2^{k-2}) \nonumber\\
& \hspace{5cm} \cdot C \bigg( 1 + \log \bigg( \frac{2^j}{2^{k+2}} \bigg) \bigg) \pi_6(2^{k+2}, 2^j) \nonumber\\
& \leq c_1 C^3 \deltaKMS \sum_{k=i+3}^{j-4} 2^{2k} \cdot \pi_6(2^{k-1}) \cdot c_2 \pi_6(2^i, 2^j) \cdot k (k-i) (1+j-i) \nonumber\\
& \leq c_3 C^3 \deltaKMS (1 + j - i) \pi_6(2^i, 2^j) \sum_{k=i+3}^{j-4} 2^{2k} (2^k)^{-(2 + \beta)} k (k-i) \nonumber\\
& \leq c_3 C^3 \deltaKMS (1 + j - i) \pi_6(2^i, 2^j) \sum_{k=3}^{\infty} k^2 (2^{-\beta})^k \nonumber\\
& = c_4 C^3 \deltaKMS (1 + j - i) \pi_6(2^i, 2^j), \label{eq:modif_6_arms_pf3}
\end{align}
where we used \eqref{eq:extendability} and \eqref{eq:quasi_mult} for the second inequality, and the a-priori upper bound \eqref{eq:6arm} on $6$-arm events for the third one. Combining \eqref{eq:modif_6_arms_pf3} with \eqref{eq:modif_6_arms_pf1} and \eqref{eq:modif_6_arms_pf2}, we get
\begin{equation} \label{eq:modif_6_arms_pf4}
\PP \big( \marm_6(\Ann_{2^i, 2^j}) \big) \leq \big( c_0 + c_4 C^3 \deltaKMS \big) (1 + j - i) \pi_6(2^i, 2^j).
\end{equation}
We now explain how to fix $C$ and $\deltaKMS$ (independently of $i$ and $j$, of course) in order to establish the induction step. Besides the lower bound on $C$ that we had already imposed in the beginning of the proof, we also require $C > 4 c_0$ (so that $c_0 < \frac{1}{4} C$). For such a choice of $C$, we then take $\deltaKMS > 0$ small enough so that $c_4 C^2 \deltaKMS < \frac{1}{4}$. This implies that the right-hand side of \eqref{eq:modif_6_arms_pf4} is at most
$$\bigg( \frac{1}{4} C + \frac{1}{4} C \bigg) (1 + j - i) \pi_6(2^i, 2^j) = \frac{1}{2} C (1 + j - i) \pi_6(2^i, 2^j).$$
This completes the induction step (using that $\log 2 > \frac{1}{2}$), and thus the proof of Theorem~\ref{thm:modif_6_arms}.
\end{proof}

\begin{remark} \label{rem:modif_6_arms_indep}
Strictly speaking, the hypothesis that the two configurations $\omega$ and $\sigma$ are independent is not fully needed, and it can be weakened, rather, to the requirement that the pairs $(\omega_v, \sigma_v)$, $v \in  V$, are independent (thus potentially allowing for a complicated dependence between $\omega_v$ and $\sigma_v$, for each $v \in V$). Indeed, the equality in \eqref{eq:modif_6_arms_rem} comes from the independence of four events, which still holds true under the weaker assumption. Hence, the induction step remains valid in this case, and thus the whole proof of Theorem~\ref{thm:modif_6_arms}. This observation turns out be important later, when we apply such results to the $N$-parameter forest fire process, as we explain in Section~\ref{sec:stab_FF}.
\end{remark}

\subsection{Applications of Theorem~\ref{thm:modif_6_arms}} \label{sec:appl_6arms}

In Section~\ref{sec:thm_sdp}, we first derive a crossing estimate (after burning and recovery), Theorem~\ref{thm:sdp}, from the stability result for six arms, before discussing classical consequences of this estimate in Section~\ref{sec:csq_sdp}.

\subsubsection{Crossing estimate} \label{sec:thm_sdp}

In this section, we show how Theorem~\ref{thm:modif_6_arms}, together with the above-mentioned part (i) of Lemma~16 in \cite{KMS2015}, quickly gives a version of Theorem~4 in \cite{KMS2015}, which is slightly weaker, but sufficient for the purposes in that paper. Roughly speaking, this result provides a useful upper bound for the existence of a crossing in a rectangle after both a ``macroscopic'' fire and then partial recovery took place.

To this end, let us recall notations from \cite{KMS2015}. First, we let
$$R_n := [-2n, 2n] \times [0, n] \quad \text{and} \quad S_n := [-3n, 3n] \times [0, n],$$
and we denote by $\din_L S_n$ and $\din_R S_n$ the left and right sides of $\din S_n$, respectively. Also, exactly as in \cite{KMS2015}, we introduce the set of vertices
\begin{equation} \label{eq:def_chi}
\chi := \big \{ v \in S_n \: : \: v \stackrel{\omega}{\llra} \din_L S_n \text{ and } v \stackrel{\omega}{\llra} \din_R S_n \text{ in } S_n \big \}
\end{equation}
(here, we write $\stackrel{\omega}{\llra}$ to stress that occupied connections are considered in the configuration $\omega$). Finally, we define the configurations $\tilde{\omega}$ and $\tilde{\omega}^{\sigma}$, obtained from $\omega$ as follows. For each $v \in S_n$,
$$\tilde{\omega}_v :=
\begin{cases}
0 & \text{ if } v \in \ol{\chi} := \chi \cup \dout \chi,\\[1mm]
1 & \text{ otherwise,}
\end{cases}$$
and
$$\tilde{\omega}^{\sigma}_v :=
\begin{cases}
\tilde{\omega}_v \vee \sigma_v & \text{ if } v \in R_n,\\[1mm]
\tilde{\omega}_v & \text{ otherwise}
\end{cases}$$
(in other words, only the vertices in the middle rectangle $R_n$ are affected by $\sigma$). A vertex $v \in S_n$ is said to be \emph{enhanced} if $\tilde{\omega}_v = 0$ and $\tilde{\omega}^{\sigma}_v = 1$ (in particular, $v$ belongs necessarily to $R_n$).

\begin{figure}[t]
\begin{center}

\includegraphics[width=.75\textwidth]{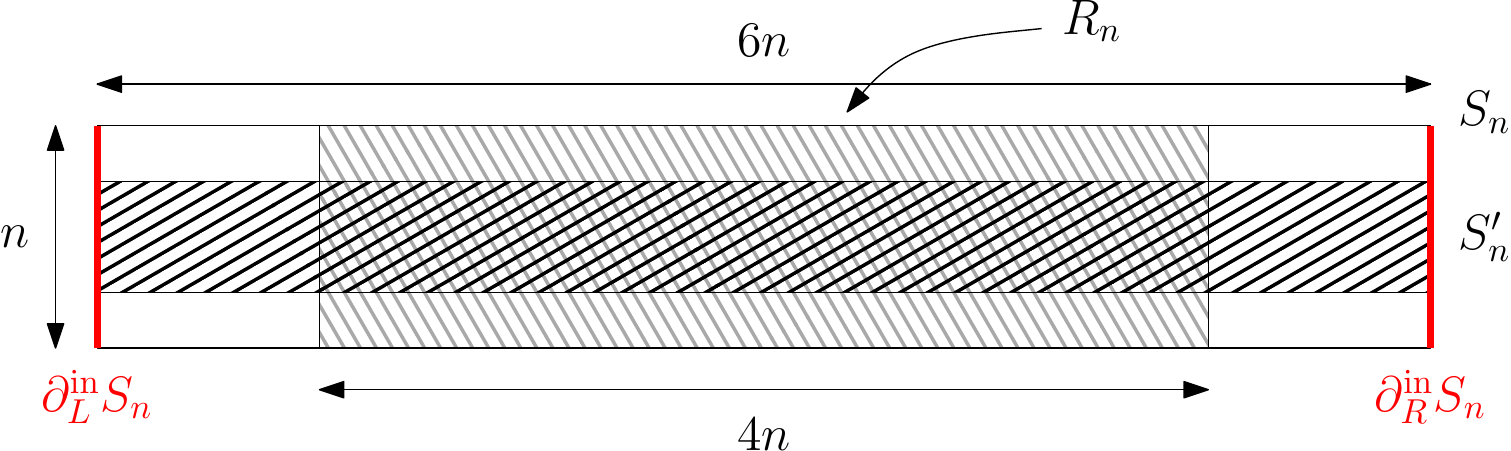}
\caption{\label{fig:rectangles} The larger rectangle is $S_n$, and its sub-rectangles $R_n$ and $S'_n$ are shaded in grey and black, respectively. We use the left and right sides of $S_n$, denoted by $\din_L S_n$ and $\din_R S_n$.}

\end{center}
\end{figure}

Now, differently from \cite{KMS2015}, we also define the sub-rectangle
$$S'_n := [-3n, 3n] \times \bigg[ \frac{n}{4}, \frac{3n}{4} \bigg]$$
of $S_n$ (Figure~\ref{fig:rectangles} shows the various rectangles involved). We establish the following result, analogous to Theorem~4 in \cite{KMS2015}.

\begin{theorem} \label{thm:sdp}
Let $\deltaKMS > 0$ be as in the statement of Theorem~\ref{thm:modif_6_arms}. There exist $\beta > 0$ and $c > 0$ such that: for all $n \geq 1$,
\begin{equation} \label{eq:sdp}
\PP_{p_c, \deltaKMS} \big( \{ \omega \in \Ch(S'_n) \} \cap \{ \tilde{\omega}^{\sigma} \in \Cv(R_n) \} \big) \leq c n^{-\beta}.
\end{equation}
\end{theorem}

Note that the difference with \cite{KMS2015} is that we require the existence of a horizontal crossing in $S'_n$ (instead of $S_n$), which is a stronger condition. We do so in order to avoid boundary effects along the top and bottom sides of $S_n$, which would require additional care (as in \cite{KMS2015}). As we mentioned earlier, this slightly weaker version turns out to be sufficient for all further applications in that paper, see Section~\ref{sec:csq_sdp}.

\begin{proof}[Proof of Theorem~\ref{thm:sdp}]
We will see that the proof is an almost immediate consequence of part(i) of Lemma~16 in \cite{KMS2015}, combined with Theorem~\ref{thm:modif_6_arms} above (and \eqref{eq:6arm}). As in \cite{KMS2015}, we first remark that it is obviously sufficient to show \eqref{eq:sdp} for the larger event obtained by replacing $R_n$ by $S_n$ in \eqref{eq:sdp}.

Suppose that $\omega \in \Ch(S'_n)$ and $\tilde{\omega}^{\sigma} \in \Cv(S_n)$. Let $\gamma$ be a vertical $\tilde{\omega}^{\sigma}$-occupied crossing of $S_n$ with a minimal number of enhanced vertices, which, from now on, we call \emph{passage sites} (recall that by the definition of $\tilde{\omega}^{\sigma}$, all such vertices lie in $R_n$). We want to mention that in \cite{KMS2015}, $\gamma$ is chosen to be the leftmost such path, but this was only done to make the choice of $\gamma$ unique, and played no role in the proof.

We observe that from Lemma~16(i) in \cite{KMS2015} and our definition of $\marm_6$, for each passage site $v$ and each $m$ such that $\Ball_m(v) \subseteq S_n$, we have $(\omega, \sigma) \in \marm_6(\Ann_{0,m}(v))$. In addition, note that, clearly, $\gamma$ has a passage site which lies in $R_n \cap S'_n$, and is thus at a distance at least $\frac{n}{4}$ from $\din S_n$. It follows that for the event in the left-hand side of \eqref{eq:sdp},
$$\{ \omega \in \Ch(S'_n) \} \cap \{ \tilde{\omega}^{\sigma} \in \Cv(R_n) \} \subseteq \bigcup_{v \in R_n \cap S'_n} \marm_6 \big( \Ann_{0,\frac{n}{4}}(v) \big).$$
Hence, we deduce from the union bound, and then Theorem~\ref{thm:modif_6_arms} (recall that we chose $\deltaKMS > 0$ as in the statement of this result), that
\begin{align*}
\PP_{p_c, \deltaKMS} \big( \{ \omega \in \Ch(S'_n) \} \cap \{ \tilde{\omega}^{\sigma} \in \Cv(R_n) \} \big) & \leq \sum_{v \in R_n \cap S'_n} \PP_{p_c, \deltaKMS} \big( \marm_6 \big( \Ann_{0,\frac{n}{4}}(v) \big) \big)\\
& \leq \big| R_n \cap S'_n \big| \cdot C \bigg( 1 + \log \bigg( \frac{n}{4} \bigg) \bigg) \pi_6 \bigg( \frac{n}{4} \bigg),
\end{align*}
where $C$ is universal. This allows us to conclude, since the $6$-arm exponent is $> 2$ (see \eqref{eq:6arm}) and $| R_n \cap S'_n | = O(n^2)$ as $n \to \infty$.
\end{proof}

\subsubsection{Classical consequences of Theorem~\ref{thm:sdp}} \label{sec:csq_sdp}

We now discuss properties which are quite directly implied by Theorem~\ref{thm:sdp}. In \cite{KMS2015}, the authors point out that Theorem~4 in that paper implies the ``conditional'' results in \cite{BB2004} and \cite{BB2006}, especially the following ones.
\begin{enumerate}[(i)]
\item There exists $\deltaKMS > 0$ small enough such that for self-destructive percolation,
$$\text{for all } p > p_c, \quad \theta(p, \deltaKMS) := \PP_{p,\deltaKMS}(0 \lra \infty) = 0$$
(when we first make vacant all vertices in the infinite occupied cluster $\Cinf$ at $p > p_c$, before performing a $\deltaKMS$-enhancement over the whole lattice: see the notations in \cite{BB2004}, as well as Theorem~3.3 there).

\item There exists a time $t > t_c$ such that for the $N$-parameter forest fire process (recall the notation from Section~\ref{sec:processes}),
\begin{equation} \label{eq:csqFF1}
\text{for all } m \geq 1, \quad \liminf_{N \to \infty} \PP_N \big( \text{a fire occurs in } \Ball_m \text{ before time } t \big) \leq \frac{1}{2}
\end{equation}
(see Theorem~4.2 in \cite{BB2006}).
\end{enumerate}

The results in \cite{BB2004} are called conditional because they are proved there under the assumption that Conjecture~3.2 in that paper holds. In \cite{KMS2015}, it is observed that, although Theorem~4 of \cite{KMS2015} does not (easily) imply Conjecture~3.2 of \cite{BB2004}, it does imply its ``main consequence'', Lemma~3.4 of \cite{BB2004}, which itself implies the mentioned conditional results in \cite{BB2004}. In fact, it is shown in \cite{KMS2015} that their Corollary~5 (which is equivalent to Lemma~3.4 of \cite{BB2004}) follows from their Theorem~4.

Although, as we said before, our Theorem~\ref{thm:sdp} is weaker than Theorem~4 of \cite{KMS2015}, it also implies Corollary~5 in that paper (and hence, the mentioned conditional results in \cite{BB2004}). Indeed, it suffices to replace, in the proof of Corollary~5 from Theorem~4 in \cite{KMS2015}, the event $\Ch(S_B)$ (i.e. the event that there is a horizontal crossing of the rectangle $S_B = S_B(n) = [-3n, 3n] \times [-2n, -n]$ considered in that paper) by the event $\Ch(S'_B)$ that there exists a horizontal crossing of the narrower rectangle $S'_B = [-3n, 3n] \times [- \frac{7}{4} n, - \frac{5}{4} n]$, and make similar replacements for the events $\Ch(S_T)$, $\Cv(S_L)$ and $\Cv(S_R)$ (which are rotated / translated versions of $\Ch(S_B)$, see Figure~\ref{fig:circuit}).

\begin{figure}[t]
\begin{center}

\includegraphics[width=.65\textwidth]{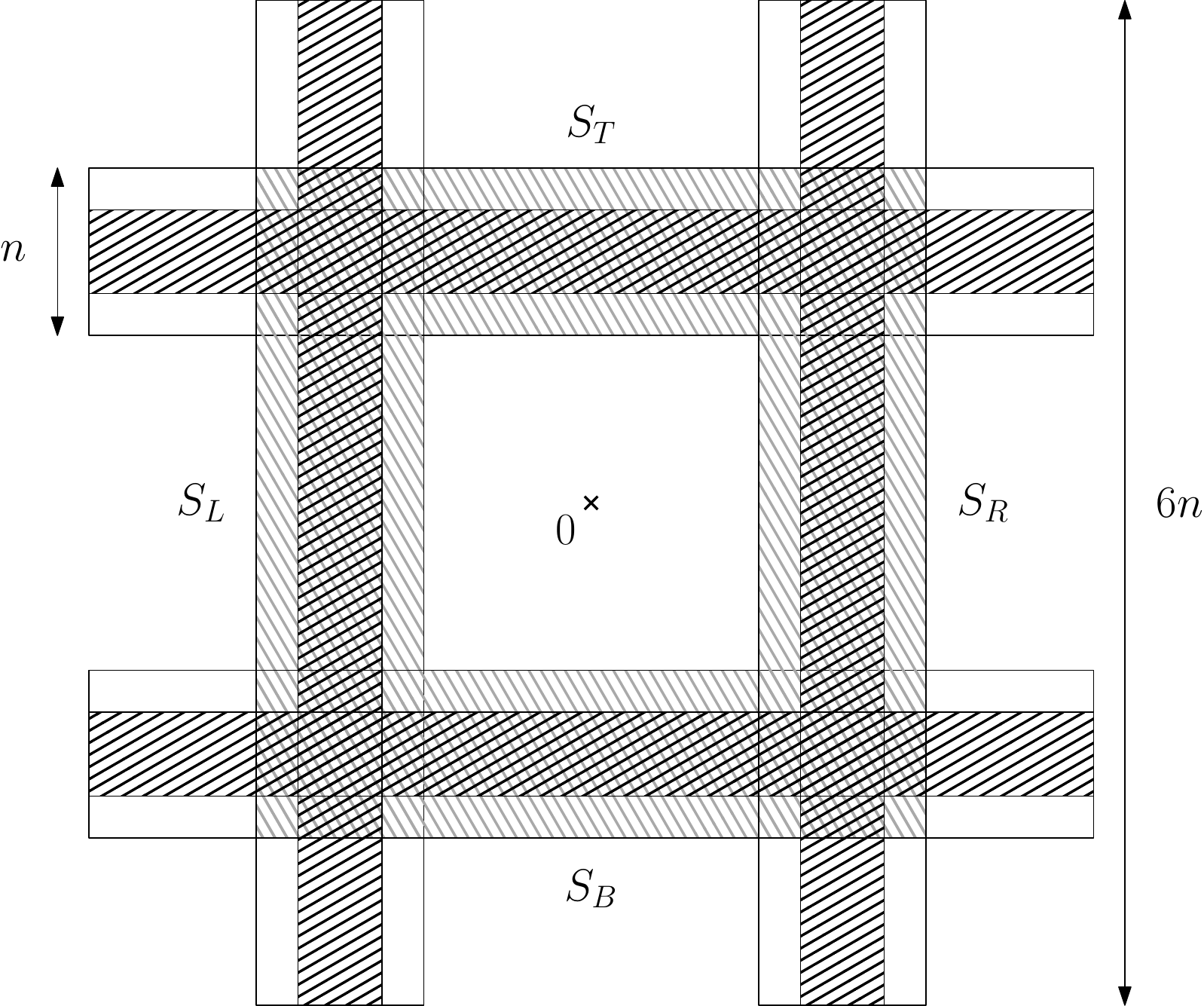}
\caption{\label{fig:circuit} We consider the rectangles $S_B$, $S_T$, $S_L$ and $S_R$, obtained by translating and rotating $S_n$, and the corresponding sub-rectangles $R_.$ (shaded in grey) and $S'_.$ (in black).}

\end{center}
\end{figure}

The authors of \cite{KMS2015} also remark that their Theorem~4 implies Conjecture~2.1 in \cite{BB2006}, from which the conditional results in \cite{BB2006} follow. Again, although our Theorem~\ref{thm:sdp} is weaker than Theorem~4 of \cite{KMS2015}, it also implies Conjecture~2.1 in \cite{BB2006}.

Concluding, our Theorem~\ref{thm:sdp}, that we just established, allows us to obtain the conditional results in \cite{BB2004} and \cite{BB2006} (as Theorem~4 of \cite{KMS2015} did).

Moreover, this result, or, rather, an extension / generalization of Theorem~\ref{thm:modif_6_arms}, combined with the results and reasonings in \cite{BKN2018}, solves one of the main open problems in \cite{BB2006}. More precisely, we will use it to show that, for the $N$-parameter forest fire process, the following (much) stronger version of \eqref{eq:csqFF1} holds: there exists $t > t_c$ such that
\begin{equation} \label{eq:csqFF2}
\text{for all } m \geq 1, \quad \PP_N \big( \text{a fire occurs in } \Ball_m \text{ before time } t \big) \stackrel[N \to \infty]{}{\longrightarrow} 0.
\end{equation}
The next section presents such an extension of Theorem~\ref{thm:modif_6_arms}, as well as a useful adaptation of Theorem~\ref{thm:sdp} that it implies. We will then explain, in Sections~\ref{sec:stab_recov} and \ref{sec:dec_FF}, how these generalizations may be used in the analysis of the $N$-parameter forest fire process.

\subsection{Generalizations} \label{sec:gen_6arms}

We now prove various extensions of Theorems~\ref{thm:modif_6_arms} and \ref{thm:sdp}, which will be useful later in the paper.

\subsubsection{Burning the cluster of an arbitrary circuit}

We first establish a version of Theorem~\ref{thm:sdp} which is needed later, where we burn circuits in an annulus, instead of horizontal crossings in a rectangle. Moreover, this result provides more flexibility than Theorem~\ref{thm:sdp}, in the sense that only a single circuit is destroyed (together with its connected component in the annulus), and we can choose freely this circuit (it does not matter which of them is burnt). This property is useful in Section~\ref{sec:dec_proof}, where at some point we burn only the vertices connected to the outermost circuit in a given annulus. It is easy to convince oneself that Theorem~\ref{thm:sdp} can be strengthened in a similar fashion, i.e. if we add the possibility to choose which of the horizontal crossings to burn.

For this purpose, we introduce the following notation, which resembles that of Section~\ref{sec:thm_sdp}.

\begin{definition}
Let $A := \Ann_{n_1, n_2}(v)$, for some $0 \leq n_1 < n_2$ and $v \in V$, and consider $\omega$, $\sigma \in \{0,1\}^A$. For all $w \in A$, we define the configurations $\tilde{\omega}^w$ and $\tilde{\omega}^{w, \sigma}$ as: for each $u \in A$,
$$\tilde{\omega}_u^w :=
\begin{cases}
0 & \text{ if } u \in \ol{\cluster(w)} := \cluster(w) \cup \dout \cluster(w),\\[1mm]
1 & \text{ otherwise,}
\end{cases}$$
where $\cluster(w)$ denotes the $\omega$-occupied cluster of $w$ in $A$ (i.e. the set of all vertices which are connected to $w$ by an $\omega$-occupied path in $A$), and
$$\tilde{\omega}_u^{w, \sigma} := \tilde{\omega}_u^w \vee \sigma_u.$$
\end{definition}

We are now in a position to state the desired analog of Theorem~\ref{thm:sdp}.

\begin{theorem} \label{thm:sdp_circ}
Let $\deltaKMS > 0$ be as in the statement of Theorem~\ref{thm:modif_6_arms}. There exists $\beta > 0$ such that the following holds. For all $0 < a < a' < b' < b$, there exists $c > 0$ such that for all $v \in V$, and $n \geq 1$, we have: in the annulus $A := \Ann_{a n, b n}(v)$,
\begin{align}
\PP_{p_c, \deltaKMS} \big( \text{there exists } w \in V \text{ on an } \omega\text{-occupied circuit in } \Ann_{a' n, b' n}(v) \text{ s.t. } \tilde{\omega}^{w, \sigma} & \in \arm_1(\Ann_{a n, b n}(v)) \big) \nonumber \\[1mm]
& \leq c n^{-\beta}. \label{eq:sdp_circ}
\end{align}
\end{theorem}

\begin{proof}[Proof of Theorem~\ref{thm:sdp_circ}]
Let $A' := \Ann_{a' n, b' n}(v)$. Suppose that $\omega$, $\sigma \in \{0,1\}^A$ have the property that for some vertex $w \in A'$, $w$ belongs to an $\omega$-occupied circuit in $A'$ and $\tilde{\omega}^{w, \sigma} \in \arm_1(A)$ (as in the left-hand side of \eqref{eq:sdp_circ}). We pick such a $w$, so that there is an $\tilde{\omega}^{w, \sigma}$-occupied arm in $A$, and we consider such a path $\gamma$ containing a minimal number of enhanced sites, i.e. of vertices $u$ which are $\tilde{\omega}^w$-vacant (in addition to being $\tilde{\omega}^{w, \sigma}$-occupied). Similarly to the proof of Theorem~\ref{thm:sdp}, these enhanced vertices on $\gamma$ are called \emph{passage sites}.

From the definitions, it is clear that there is at least one passage site in $A'$, and we pick such a vertex $u$. Clearly, it lies at a distance at least $d n$ from the boundary of $A$, where $d = \min(a' - a, b - b')$. By practically the same arguments as in the deterministic Lemma~16(i) in \cite{KMS2015}, we get that $(\omega, \sigma) \in \marm_6(\Ann_{0, d n}(u))$. The proof can now be completed in the same way as the last part of the proof of Theorem~\ref{thm:sdp}.
\end{proof}

\subsubsection{Near-critical extension}

We now present a generalization of Theorem~\ref{thm:modif_6_arms}, from critical to near-critical percolation, which follows from by-now classical results. We first need to give suitable modifications of the earlier definitions.

We generalize Definition~\ref{def:6_arms_modified} as follows. Let $A := \Ann_{n_1, n_2}(v)$, for some $0 \leq n_1 < n_2$ and $v \in V$. As for $\marm_6(A)$, we still take $\sigma \in \{0,1\}^{\Ball_{n_2}(v)}$, but we now consider $\omega \in [0,1]^{\Ball_{n_2}(v)}$ instead. For all $p \in [0,1]$, we let $\omega^{(p)} \in \{0,1\}^{\Ball_{n_2}(v)}$ be the configuration defined by: for each vertex $w$, $\omega^{(p)}_w = 1$ if $\omega_w \leq p$ (in this case, $w$ is said to be $p$-occupied), and $\omega^{(p)}_w = 0$ otherwise ($w$ is $p$-vacant).

For any two parameters $p, p' \in (0,1)$, we then define $\darm_6^{p,p'}$ as the extension of $\darm_6$ where ``occupied'' and ``vacant'' are replaced by ``$p$-occupied'' and ``$p'$-vacant'', respectively. In other words, $\omega \in \darm_6^{p,p'}(A)$ means that the annulus $A$ is crossed by $6$ disjoint arms with types $(ovvovv)$, where $o$ now stands for ``$\omega$-value $\leq p$'' and $v$ for ``$\omega$-value $> p'$'', and where one of the arms with type $o$ may have a defect, in the same sense as before. This leads us to the following generalized definition of $\marm_6(A)$.

\begin{definition} \label{def:6_arms_modified2}
Let $0 \leq n_1 < n_2$, $v \in V$, and $A := \Ann_{n_1, n_2}(v)$. The modified ``fuzzy'' $6$-arm event $\marm_6^{p,p'}(A)$ is the set of all pairs $(\omega, \sigma) \in [0,1]^{\Ball_{n_2}(v)} \times \{0,1\}^{\Ball_{n_2}(v)}$ such that there exists a (possibly empty) set of vertices $\calP \subseteq A$ for which the two properties below hold.
\begin{enumerate}[(i)]
\item For all $w \in \calP$, $\sigma_w = 1$.

\item For all $w \in \Ball_{n_2}(v)$ and $0 \leq m_1 < m_2$ with $\Ann_{m_1, m_2}(w) \subseteq A$, we have the following property. If
\begin{itemize}
\item $\Ball_{m_1}(w) \supseteq \Ball_{n_1}(v)$ or $\Ball_{m_1}(w) \cap \calP \neq \emptyset$,

\item and $\Ann_{m_1, m_2}(w) \cap \calP = \emptyset$,
\end{itemize}
then necessarily $\omega \in \darm_6^{p,p'}(\Ann_{m_1, m_2}(w))$.
\end{enumerate}
\end{definition}

For this new definition, Theorem~\ref{thm:modif_6_arms} can be generalized as follows.

\begin{theorem} \label{thm:modif_6_arms2}
For all $K \geq 1$, there exist $\deltaKMS = \deltaKMS(K) > 0$ and $C = C(K) > 0$ such that the following holds. For all $p, p' \in (0,1)$, $0 \leq n_1 < n_2 \leq K (L(p) \wedge L(p'))$, and $v \in V$,
\begin{equation} \label{eq:modif_6_arms2}
\PP_{\deltaKMS} \big( \marm_6^{p,p'}(\Ann_{n_1, n_2}(v)) \big) \leq C \bigg( 1 + \log \bigg( \frac{n_2}{n_1 \vee 1} \bigg) \bigg) \pi_6(n_1, n_2).
\end{equation}
\end{theorem}

\begin{proof}[Proof of Theorem~\ref{thm:modif_6_arms2}]
This result can be obtained from a quite straightforward adaptation of the proof of Theorem~\ref{thm:modif_6_arms}, using standard results about near-critical percolation (such as, in particular, (2.10) in \cite{BKN2018}).
\end{proof}

\subsubsection{Geodesics} \label{sec:geodesics}

In this section, we use Theorem~\ref{thm:modif_6_arms2} to derive a near-critical generalization of Theorem~\ref{thm:sdp}. This is a generalization also in the sense that the set of vertices which is made vacant is more ``flexible''. In addition, the geometric setup is more general, as we now explain.

For this purpose, let $\calD = (\calD, E_{\calD})$ be a domain, i.e. a finite subgraph of $\TT$ (with some abuse of notation, we also use $\calD$ for its set of vertices), $\omega \in [0,1]^{\calD}$ and $\sigma \in \{0,1\}^{\calD}$. For each vertex $w \in \calD$ and each $p \in (0,1)$, we consider the $\omega^{(p)}$-occupied cluster of $w$ in $\calD$, denoted by $\cluster^{(p)}(w) = \cluster^{(p)}(w; \calD)$ (i.e. each vertex $v \in \cluster^{(p)}(w)$ has $\omega_v \leq p$). We need the following modifications of $\tilde{\omega}$ and $\tilde{\omega}^{\sigma}$.

\begin{definition}
Let $w \in \calD$ and $p \in (0,1)$. We define the configurations $\tilde{\omega}^{(p),w}$ and $\tilde{\omega}^{(p), w, \sigma}$ as follows. For each $v \in \calD$,
$$\tilde{\omega}_v^{(p),w} :=
\begin{cases}
0 & \text{ if } v \in \ol{\cluster^{(p)}(w)} := \cluster^{(p)}(w) \cup \dout \cluster^{(p)}(w),\\[1mm]
1 & \text{ otherwise,}
\end{cases}$$
and then $\tilde{\omega}_v^{(p), w, \sigma} := \tilde{\omega}_v^{(p),w} \vee \sigma_v$.
\end{definition}

A path $\gamma$ in $\calD$ is said to be a \emph{geodesic} (with respect to $\tilde{\omega}^{(p),w}$ and $\sigma$) if
\begin{itemize}
\item every vertex on $\gamma$ is $\tilde{\omega}^{(p), w, \sigma}$-occupied (i.e. it has $\tilde{\omega}^{(p), w, \sigma}$-value $1$),

\item and there is no $\tilde{\omega}^{(p), w, \sigma}$-occupied path in $\calD$ with the same endpoints as $\gamma$, but strictly fewer enhanced sites (i.e. vertices $v$ such that $\tilde{\omega}^{(p), w}_v = 0$).
\end{itemize}
If $\gamma$ is a geodesic, the vertices $v$ on $\gamma$ where $\tilde{\omega}^{(p), w}_v = 0$ are called \emph{passage sites} of $\gamma$.

The following purely deterministic lemma is an extension of part (i) of Lemma~16 in \cite{KMS2015}.

\begin{lemma} \label{lem:det_6arms}
Let $u \in \calD$, and $R \geq 1$ be such that $\Ball_R(u) \subseteq \calD$. Let $w \in \calD$, $p \in (0,1)$, and let $\gamma$ be a geodesic in $\calD$ with respect to $\tilde{\omega}^{(p), w}$ and $\sigma$, of which both endpoints lie outside $\Ball_R(u)$. If for some $0 \leq r < R$, $\Ball_r(u)$ contains a passage site of $\gamma$, but $\Ann_{r, R}(u)$ does not contain any such site, then $\omega^{(p)} \in \darm_6(\Ann_{r,R}(u))$.
\end{lemma}

\begin{proof}[Proof of Lemma~\ref{lem:det_6arms}]
This result follows from very similar arguments as for Lemma~16(i) in \cite{KMS2015}. We sketch them briefly, and we refer the reader to Figure~\ref{fig:geodesics} for an illustration.

First, we choose (in an arbitrary way) one passage site $v$ of $\gamma$ in $\Ball_r(u)$, and we extract from $\gamma$ the maximal sub-path $\bar{\gamma}$ which is completely contained in $\Ball_R(u)$, and such that $v \in \bar{\gamma}$. Note that $\bar{\gamma}$ is also a geodesic, and it gives rise to a partition of $A = \Ann_{r, R}(u)$ into two regions, that we can denote by $A^{\text{L}}$ and $A^{\text{R}}$.

In the same way as for Lemma~16(i) in \cite{KMS2015}, the passage site $v$ yields six disjoint arms, three inside each of $A^{\text{L}}$ and $A^{\text{R}}$: one $\omega^{(p)}$-occupied arm, and two $\omega^{(p)}$-vacant ones (both obtained from the outer boundary of the cluster of the occupied arm). However, one should be careful that one of the two occupied arms may have a defect, with size at most $2$ (exactly as in the definition of the event $\darm_6(A)$).
\end{proof}

\begin{figure}[t]
\begin{center}

\includegraphics[width=.65\textwidth]{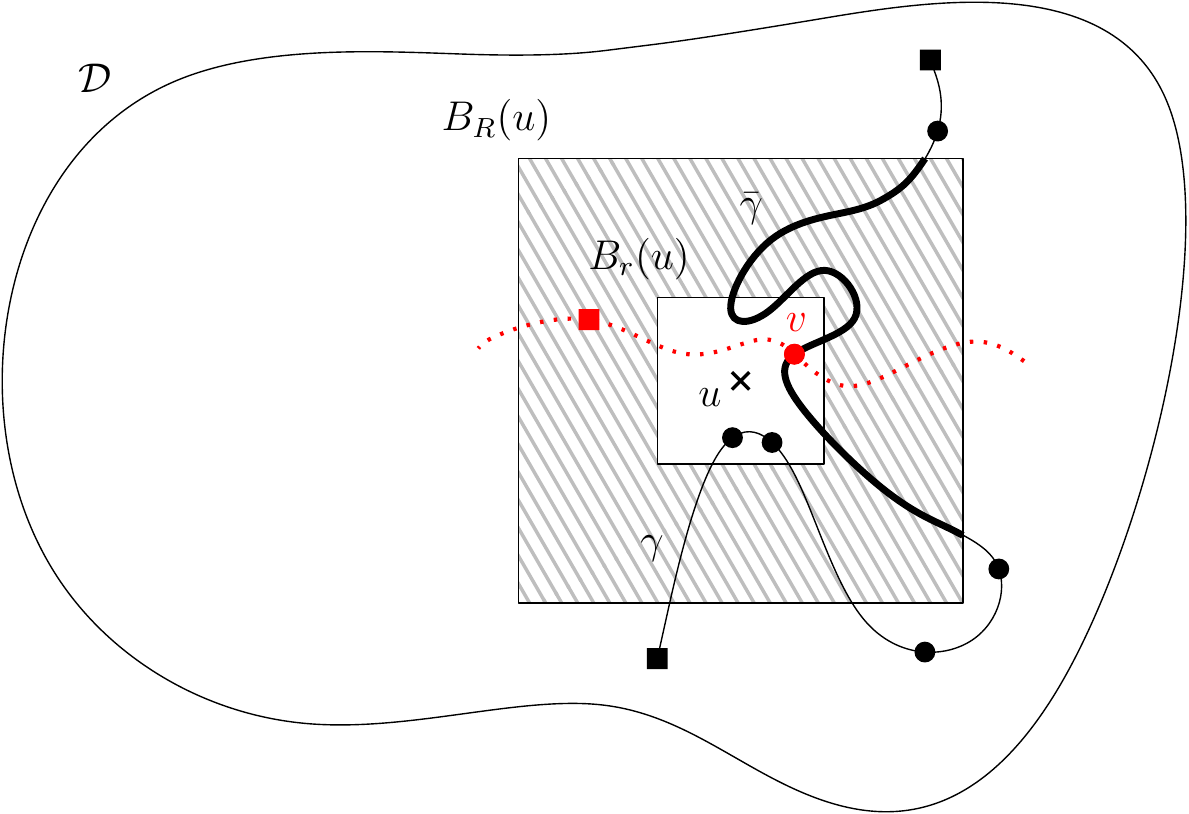}
\caption{\label{fig:geodesics} The path $\gamma$ is a geodesic, and both of its endpoints (shown as black squares) lie outside $\Ball_R(u)$. Its passage sites are marked with disks, and at least one such vertex $v$ belongs to $\Ball_r(u)$. The sub-path $\bar{\gamma}$, which contains $v$, is drawn in a thicker line. On the other hand, the annulus $\Ann_{r, R}(u)$, shaded in grey, does not contain any passage site, and $\bar{\gamma}$ subdivides it into two parts. The two red arms are disjoint and they are both included in $\cluster^{(p)}(w)$ (so $\omega^{(p)}$-occupied), except for a potential defect (with size at most $2$) on one of them, shown with a red square.}

\end{center}
\end{figure}

From Lemma~\ref{lem:det_6arms} and the definition of $\marm_6$, we obtain immediately the following.

\begin{lemma} \label{lem:det_6arms2}
Let $u \in \calD$, and $R \geq 1$ be such that $\Ball_R(u) \subseteq \calD$. Let $w \in \calD$, $p \in (0,1)$, and let $\gamma$ be a geodesic in $\calD$ with respect to $\tilde{\omega}^{(p), w}$ and $\sigma$, of which both endpoints lie outside $\Ball_R(u)$. If for some $0 \leq r < R$, $\Ball_r(u)$ contains a passage site of $\gamma$, then $(\omega, \sigma) \in \marm_6^{p,p}(\Ann_{r,R}(u))$.
\end{lemma}

This suggests to introduce the following event.

\begin{definition} \label{def:geod}
Let $w \in \calD$, $p \in (0,1)$, and $M > 0$. We define $G^{(p)}(w; M) = G^{(p)}(w; M, \calD) := \{$there exists a geodesic in $\calD$ with respect to $\tilde{\omega}^{(p), w}$ and $\sigma$, of which both endpoints lie at a distance $> M$ from $w$, and which contains at least one vertex among $w$ and its neighbors, i.e. in $\{w\} \cup \dout \{w\} \}$.
\end{definition}

\begin{theorem} \label{thm:geod}
For all $K \geq 1$, there exist $\deltaKMS = \deltaKMS(K) > 0$ and $C = C(K) > 0$ such that: for all $p \in (0,1)$, $2 \leq M \leq K L(p)$, and $w \in \calD$ with $\Ball_M(w) \subseteq \calD$,
\begin{equation}
\PP_{\deltaKMS}\big( G^{(p)}(w; M, \calD) \big) \leq C \cdot \log M \cdot \pi_6(M).
\end{equation}
\end{theorem}

\begin{proof}[Proof of Theorem~\ref{thm:geod}]
Assume that the event $G^{(p)}(w; M, \calD)$ occurs. Then by definition, there is a geodesic $\gamma$ (with respect to $\tilde{\omega}^{(p), w}$ and $\sigma$) in $\calD$ with both endpoints at a distance $> M$ from $w$, and containing a vertex in $\{w\} \cup \dout \{w\}$. Consider such a vertex $v$ on $\gamma$, which is thus a passage site of $\gamma$. Hence, by Lemma~\ref{lem:det_6arms2}, the event $\marm_6^{p,p}(\Ann_{0,M}(v))$ occurs. We can now conclude by using Theorem~\ref{thm:modif_6_arms2} (and the fact that $|\dout \{w\}| \leq 6$).
\end{proof}

For applications, we also need (a form of) the following result, which slightly generalizes the previous theorem.

\begin{theorem} \label{thm:geod2}
For all $K \geq 1$, there exist $\deltaKMS = \deltaKMS(K) > 0$ and $C = C(K) > 0$ such that: for all $0 < p < p' < 1$, $2 \leq M \leq K (L(p) \wedge L(p'))$, and $w \in V$,
\begin{align*}
\PP_{\deltaKMS}\big( \text{there exist a domain } \calD \supseteq \Ball_M(w) \text{ and some } \tilde{p} \in (p, p') & \text{ s.t. } G^{(\tilde{p})}(w; M, \calD) \text{ holds}  \big)\\[1mm]
& \leq C \cdot \log M \cdot \pi_6(M).
\end{align*}
\end{theorem}

\begin{proof}[Proof of Theorem~\ref{thm:geod2}]
As in the proof of Theorem~\ref{thm:geod}, we have that if $G^{(\tilde{p})}(w; M, \calD)$ occurs for some $\tilde{p} \in (p,p')$ (where $\calD$ and $w$ satisfy $\Ball_M(w) \subseteq \calD$), then $\marm_6^{\tilde{p},\tilde{p}}(\Ann_{0,M}(v))$ holds for some $v \in \{w\} \cup \dout \{w\}$ and the same $\tilde{p} \in (p,p')$. Clearly, for such a vertex $v$, the event $\marm_6^{p',p}(\Ann_{0,M}(v))$ holds, so that it suffices to use Theorem~\ref{thm:modif_6_arms2} again (and the fact that the number of neighbors of $w$ is smaller than some universal constant).
\end{proof}

\section{Evolution of the percolation holes} \label{sec:stab_recov}

We now explain more how to use the results from Section~\ref{sec:positive_rec_time} in the context of forest fires. Percolation holes (introduced in Section~\ref{sec:holes}) play a central role in the proofs of Section~\ref{sec:dec_FF}, and we study their evolution through recoveries in Section~\ref{sec:robust}. We then make the connection with forest fires in Section~\ref{sec:stab_FF}. 

\subsection{Robustness through recoveries} \label{sec:robust}

In this section, we study the effect of recoveries on the strong hole $\holes$. Note that this percolation hole is also, with high probability, the hole obtained when burning all vertices connected to the boundary in a box centered on $0$ with a large enough side length (or, alternatively, when burning the largest cluster in this box). Moreover, this holds true in approximable domains as well.

First, recall the following consequence \eqref{eq:holes_apriori} of Lemma~\ref{lem:bound_holes} on the weak and strong holes: for all $\ve > 0$, there exist $\ol{\lambda} > \ul{\lambda} > 0$ such that
$$\PP_p \big( \Ball_{\ul{\lambda} L(p)} \subseteq \holes \subseteq \holew \subseteq \Ball_{\ol{\lambda} L(p)} \big) \geq 1 - \ve$$
for all $p > p_c$ sufficiently close to $p_c$.

As already mentioned, the successive holes around $0$ play an important role when studying the dynamics of the frozen percolation and forest fire processes. In the latter case, the hole produced by a burning may be modified, before the next burning takes place inside it, by subsequent recoveries. More specifically, some vacant vertices (e.g. vertices which just burnt, or neighbors of such vertices) may become occupied, which can have the effect of enlarging the hole. This would affect the time at which the next burning in the hole occurs (this time being decreased).

We now prove that it is likely that such modifications of the hole are in fact very small, in a sense that we make precise. Furthermore, the reasonings even show that if, after the first burning, we stop the ignitions so that there is only recovery, the resulting modified hole (which is potentially even larger) still differs very little from $\holes$ after some ``macroscopic'' time $\deltaKMS > 0$.

For any configuration $\sigma = (\sigma_v)_{v \in V} \in \{0,1\}^V$, we introduce a modification of $\holes$ associated with $\sigma$. It is obtained by adding to $\holes$ all vertices in $V \setminus \holes$ which can be reached from $\holes$ by an ``appropriate'' path, of which all vertices belonging to $\Cinf \cup \dout \Cinf$ have been ``recovered'', i.e. have $\sigma$-value $1$. Note that the actual hole after recovery is smaller than this modification, since typically many vertices in $V \setminus (\holes \cup \Cinf \cup \dout \Cinf)$ have both $\sigma$-value $0$ and $\omega$-value $0$.

Let us give more precise definitions.

\begin{definition} \label{def:hole_recov}
Let $\omega$, $\sigma \in \{0,1\}^V$. First, the configurations $\tilde{\tilde{\omega}}$ and $\tilde{\tilde{\omega}}^{\sigma}$ are defined as follows: for each $v \in V$, we let
$$\tilde{\tilde{\omega}}_v :=
\begin{cases}
0 & \text{ if } v \in \ol{\Cinf} := \Cinf \cup \dout \Cinf,\\[1mm]
1 & \text{ otherwise,}
\end{cases}$$
and $\tilde{\tilde{\omega}}_v^{\sigma} := \tilde{\tilde{\omega}}_v \vee \sigma_v$. The \emph{$\sigma$-modified} strong hole $\holes^{\sigma}$ $(\supseteq \holes)$ is then obtained as
\begin{equation} \label{eq:def_hole_recov}
\holes^{\sigma} := \holes \cup \Big\{ v \in V \setminus \holes \: : \: v \stackrel{\tilde{\tilde{\omega}}^{\sigma}}{\llra} \holes \Big\}.
\end{equation}
\end{definition}

Note that in the above definition, all vertices in $\holes$ have, in particular, $\tilde{\tilde{\omega}}$-value $1$ (from the definition of $\holes$).

Since the $6$-arm exponent is $> 2$ (see \eqref{eq:6arm}), we can fix a universal $\upsilon < 1$ such that
\begin{equation} \label{eq:choice_upsilon}
n^2 \cdot \log n \cdot \pi_6(n^{\upsilon}) \stackrel[n \to \infty]{}{\longrightarrow} 0
\end{equation}
(actually, the exact value $\alpha_6 = \frac{35}{12}$ is known on $\TT$ \cite{SW2001}, from computations using the connection between critical percolation and the Schramm-Loewner Evolution process with parameter $6$, which shows that $\upsilon$ can be chosen arbitrarily in the interval $(\frac{24}{35},1)$). We then take
\begin{equation} \label{eq:def_M}
M = M(p) := L(p)^{\upsilon}
\end{equation}
($\ll L(p)$ as $p \searrow p_c$). We show that with high probability, $\holes^{\sigma}$ is contained in the $M$-thickening of $\holes$, i.e.
\begin{equation} \label{eq:hole_thick}
\big[ \holes \big]^{(M)} := \holes \cup \Big\{ z \in \CC \setminus \holes \: : \: d(z, \holes) \leq M \Big\}
\end{equation}
(as defined in \eqref{eq:def_thick}), where, as before, $d$ is the distance induced by $\|.\|$. Recall that $\holes$ has typically a radius of order $L(p)$ (see Lemma~\ref{lem:bound_holes}), so that this thickening can be considered as a negligible perturbation of $\holes$, as $p \searrow p_c$.

We actually prove the slightly more uniform version below. Recall that $\dul{p}^{(K)} \geq \dol{p}^{(K)} > p_c$ were defined in \eqref{eq:def_pK}, for $p > p_c$ and $K \geq 1$.

\begin{theorem} \label{thm:M_thick}
There exist universal constants $\deltaKMS > 0$ and $\upsilon \in (0,1)$ such that the following holds. If $\sigma = (\sigma_v)_{v \in V}$ is independent of $\omega$, and $\sigma \sim \PP_{\deltaKMS}$, then for all $K \geq 1$, we have
\begin{equation} \label{eq:M_thick}
\PP_{\deltaKMS} \big(\text{for all } \tilde{p} \in [\dol{p}^{(K)}, \dul{p}^{(K)}], \: \holes(\tilde{p})^{\sigma} \subseteq \big[ \holes(\tilde{p}) \big]^{(M)} \big) \stackrel[p \searrow p_c]{}{\longrightarrow} 1
\end{equation}
with $M = L(p)^{\upsilon}$.
\end{theorem}

A close inspection of the proof below shows that the interval $[\dol{p}^{(K)}, \dul{p}^{(K)}]$ is not optimal, and could easily be enlarged. However, this result is enough for all our applications. As a matter of fact, we often do not apply Theorem~\ref{thm:M_thick} directly, but rather in combination with the observation that, although the region under consideration is not infinite (and not exactly known), it is so large that the burnt cluster coincides there with the infinite cluster in the whole lattice (which follows essentially from Lemma~\ref{lem:hole_domain}). This can be justified precisely by the results this theorem is based on (in particular Theorem~\ref{thm:geod2}).

\begin{proof}[Proof of Theorem~\ref{thm:M_thick}]
Let $\ve >0$, and consider $\ol{\lambda}(\ve)$ and $\ul{\lambda}(\ve)$ associated with it through \eqref{eq:holes_apriori}. Hence, for each $\tilde{p} \in [\dol{p}^{(K)}, \dul{p}^{(K)}]$, we have: with a probability at least $1 - \ve$, $\Ball_{\ul{\lambda} L(\tilde{p})} \subseteq \holes(\tilde{p}) \subseteq \holew(\tilde{p}) \subseteq \Ball_{\ol{\lambda} L(\tilde{p})}$. In particular, we get from the obvious monotonicity properties (in $\tilde{p}$) of the holes, \eqref{eq:def_pK}, and the union bound, that
\begin{equation} \label{eq:bound_holes_unif}
\PP_{\deltaKMS} \big(\text{for all } \tilde{p} \in [\dol{p}^{(K)}, \dul{p}^{(K)}], \: \Ball_{\ul{\lambda} K^{-1} L(p)} \subseteq \holes(\tilde{p}) \subseteq \holew(\tilde{p}) \subseteq \Ball_{\ol{\lambda} K L(p)} \big) \geq 1 - 2 \ve.
\end{equation}
We assume it to be the case in the remainder of the proof, and we write $\ul{\ul{\lambda}} := \ul{\lambda} K^{-1}$ and $\ol{\ol{\lambda}} := \ol{\lambda} K$. In addition, we consider the finite domain
$$\calD := \Ball_{2 \ol{\ol{\lambda}} L(p)}.$$
We can also require $\ol{\lambda}$ to be large enough so that with a probability at least $1 - \ve$, the event $\circuitevent(\Ann_{\ol{\ol{\lambda}} L(p), 2 \ol{\ol{\lambda}} L(p)})$ occurs at $\dol{p}^{(K)}$ (which is assumed in the following). In particular, this event implies that at each $\tilde{p} \in [\dol{p}^{(K)}, \dul{p}^{(K)}]$, all vertices in $\Cinf \cap \Ball_{\ol{\ol{\lambda}} L(p)}$ belong to the same occupied cluster in $\calD$.

\begin{figure}[t]
\begin{center}

\includegraphics[width=.85\textwidth]{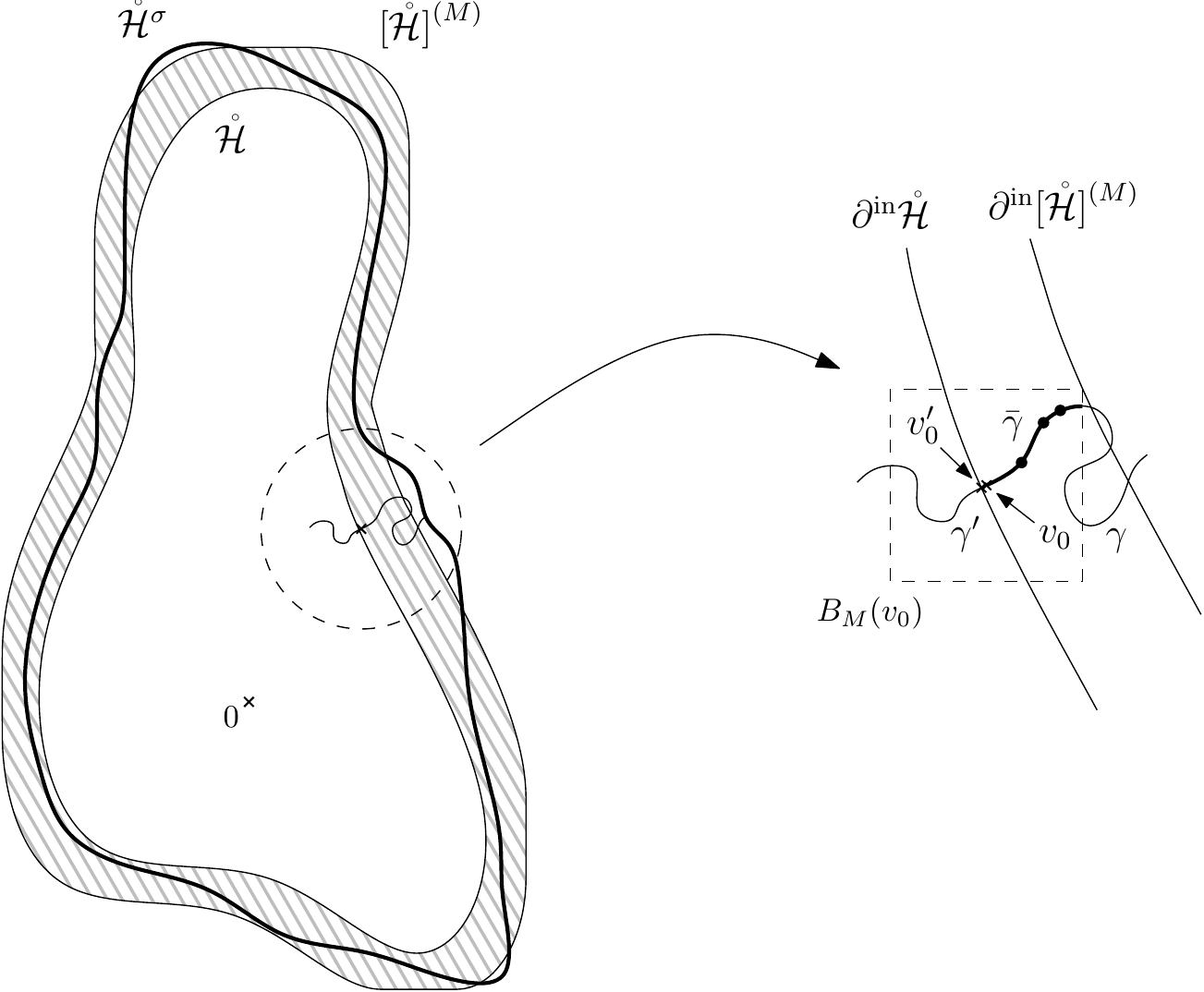}
\caption{\label{fig:thickening} This figure shows the strong hole $\holes$, together with its $M$-thickening $\big[ \holes \big]^{(M)}$ (with respect to the distance induced by $\|.\|$), and its enlargement $\holes^{\sigma}$ after partial recovery (specified by $\sigma \sim \PP_{\deltaKMS}$). The path $\gamma$ is $\tilde{\tilde{\omega}}^{\sigma}$-occupied, with a minimal number of vertices having $\tilde{\tilde{\omega}}$-value $0$ (marked with disks), and we denote by $v_0$ ($\in \dout \holes$) its starting vertex. The path $\gamma'$ is entirely contained in $\holes$, so in particular $\tilde{\tilde{\omega}}$-occupied, and it ends at some vertex $v'_0$ on $\din \holes$, with $v'_0 \sim v_0$.}

\end{center}
\end{figure}

Suppose that for some $\tilde{p} \in [\dol{p}^{(K)}, \dul{p}^{(K)}]$, $\holes(\tilde{p})^{\sigma}$ is not contained in $\big[ \holes(\tilde{p}) \big]^{(M)}$. We fix such a $\tilde{p}$, and we use it to find a geodesic, as considered in Section~\ref{sec:geodesics} (the construction is depicted on Figure~\ref{fig:thickening}). In the remainder of the proof, we often drop the dependence of $\holes$ on $\tilde{p}$, in order to simplify notation.

First, from \eqref{eq:def_hole_recov} and \eqref{eq:hole_thick}, there exists necessarily an $\tilde{\tilde{\omega}}^{\sigma}$-occupied path from $\dout \holes$ to a vertex at a distance at least $M$ from $\holes$. Let $\gamma$ be such a path having a minimal number of ``enhanced'' vertices, i.e. with $\tilde{\tilde{\omega}}$-value $0$, and let $v_0 \in \dout \holes$ be the first vertex along $\gamma$.

We extract from this path the sub-path $\bar{\gamma}$, obtained by following $\gamma$ starting from $v_0$, until one reaches a vertex at a distance at least $M$ from $v_0$. Let $\gamma'$ be a path which lies entirely in $\holes$, starts at a site at a distance at least $M$ from $v_0$, and ends at a neighbor $v'_0$ (in $\holes$) of $v_0$. Note that such a path exists as soon as $\holes \supseteq \Ball_M$, which occurs for all $p$ sufficiently close to $p_c$. Indeed, we assumed that $\Ball_{\ul{\ul{\lambda}} L(p)} \subseteq \holes$, and we chose $M = L(p)^{\upsilon}$ (see \eqref{eq:def_M}), which is $\ll L(p)$ as $p \searrow p_c$ (since $\upsilon < 1$). In addition, $\gamma'$ is automatically $\tilde{\tilde{\omega}}$-occupied (from the earlier observation, below Definition~\ref{def:hole_recov}, that all vertices in $\holes$ have $\tilde{\tilde{\omega}}$-value $1$). Finally, let $\tilde{\gamma}$ be the concatenation of $\gamma'$ and $\bar{\gamma}$, in this order.

We have $v_0 \in \dout \Cinf$, since $v_0$ lies in $\dout \holes$, so we can find a vertex $w$ which is a neighbor of $v_0$ and belongs to $\Cinf$. From the way it was chosen, it is then clear that $\tilde{\gamma}$ is a geodesic with respect to $\tilde{\omega}^{(\tilde{p}), w}$ and $\sigma$. Hence, the event $G^{(\tilde{p})}(w; M, \calD)$ occurs for some $w \in \Ball_{\ol{\ol{\lambda}} L(p)+2}$ and some $\tilde{p} \in [\dol{p}^{(K)}, \dul{p}^{(K)}]$.

We now take $\deltaKMS = \deltaKMS(1)$, in the notation of Theorem~\ref{thm:geod2}. We have, for all $p$ close enough to $p_c$ so that the requirement on $M$ in this result is satisfied (recall from \eqref{eq:def_M} that $M = L(p)^{\upsilon} \ll L(p)$),
\begin{align*}
\PP_{\deltaKMS} \big(\text{for some } & \tilde{p} \in [\dol{p}^{(K)}, \dul{p}^{(K)}], \: \holes(\tilde{p})^{\sigma} \nsubseteq \big[ \holes(\tilde{p}) \big]^{(M)} \big)\\[1mm]
& \leq 3 \ve + \PP_{\deltaKMS} \big( G^{(\tilde{p})}(w; M, \calD) \text{ holds for some } w \in \Ball_{\ol{\ol{\lambda}} L(p)+2} \text{ and } \tilde{p} \in [\dol{p}^{(K)}, \dul{p}^{(K)}] \big)\\[1mm]
& \leq 3 \ve + c_1 \big( \ol{\ol{\lambda}} L(p)+2 \big)^2 \cdot C \cdot \log \big( L(p)^{\upsilon} \big) \cdot \pi_6 \big( L(p)^{\upsilon} \big),
\end{align*}
where Theorem~\ref{thm:geod2} is used in the second line, as well as the bound $|\Ball_{\ol{\ol{\lambda}} L(p)+2}| \leq  c_1 (\ol{\ol{\lambda}} L(p)+2)^2$ for some universal constant $c_1$. From our choice of $\upsilon$ (see \eqref{eq:choice_upsilon}), the second term in the right-hand side tends to $0$ as $p \searrow p_c$ (using that $\ol{\ol{\lambda}}$ depends only on $K$ and $\ve$), so
\begin{equation} \label{eq:M_thick_pf}
\PP_{\deltaKMS} \big(\text{for some } \tilde{p} \in [\dol{p}^{(K)}, \dul{p}^{(K)}], \: \holes(\tilde{p})^{\sigma} \nsubseteq \big[ \holes(\tilde{p}) \big]^{(M)} \big) \leq 4 \ve
\end{equation}
for all $p$ sufficiently close to $p_c$. Since \eqref{eq:M_thick_pf} holds for all $\ve > 0$, this completes the proof of Theorem~\ref{thm:M_thick}.
\end{proof}

\subsection{Persistent barriers in forest fires} \label{sec:stab_FF}

In this section, we explain how our results for percolation with recoveries are applied in Section~\ref{sec:dec_FF}. More specifically, we will use Theorem~\ref{thm:sdp_circ} and Theorem~\ref{thm:M_thick}. Recall that the successive birth times at a vertex $v \in V$ are denoted by $(\tau_v^i)_{i \geq 1}$.

First, we let the configuration $\omega$ depend on $t$, with
\begin{equation} \label{eq:def_omega_FF}
\omega_v(t) := \ind_{\tau_v^1 \leq t} \quad (v \in V, \: t \in [0,\infty)).
\end{equation}
On the other hand, we fix some $\deltaKMS > 0$ small enough so that the conclusions of both Theorem~\ref{thm:sdp_circ} and Theorem~\ref{thm:M_thick} hold, and we let: for each $v \in V$,
\begin{equation} \label{eq:def_sigma_FF}
\sigma_v := \ind_{\exists i \geq 1 \: : \: \tau_v^i \in [t_c, t_c + \deltaKMS]}
\end{equation}
(observe that $\PP(\sigma_v = 1) \leq \deltaKMS$, since births occur at rate $1$). From now on, we stick to this choice for the configurations $\omega$ and $\sigma$, i.e. we always assume that they are given by \eqref{eq:def_omega_FF} and \eqref{eq:def_sigma_FF}.

Clearly, $\omega$ and $\sigma$ are not independent, so that we are not in a position to apply our earlier results directly. However, the weaker assumption that the pairs $(\omega_v, \sigma_v)$, $v \in V$, are independent, mentioned in Remark~\ref{rem:modif_6_arms_indep}, is verified, which is enough to obtain \eqref{eq:sdp_circ} and \eqref{eq:M_thick} (of course, after replacing $\PP_{p_c, \deltaKMS}$ and $\PP_{\deltaKMS}$, respectively, by the probability measures produced by \eqref{eq:def_omega_FF} and \eqref{eq:def_sigma_FF}).

\section{Deconcentration for forest fires and proof of main results} \label{sec:dec_FF}

As mentioned in the introduction (in particular, see the discussion in Section~\ref{sec:backgr}, as well as the general outline in Section~\ref{sec:intro_org}), we now harness the deconcentration property established in \cite{BKN2018} for volume-frozen percolation. This is done by applying our results in Sections~\ref{sec:positive_rec_time} and \ref{sec:stab_recov} to control the effect of recoveries. In Section~\ref{sec:dec_iteration}, we prove a version of Theorem~\ref{thm:main_result} for large (but not too large, as a function of $N$) finite domains. Finally, in Section~\ref{sec:dec_proof}, we extend it to arbitrary large enough finite domains, as well as the full lattice. This will complete the proofs of Theorems~\ref{thm:main_result} and \ref{thm:main_result_finite}.

\subsection{Iteration procedure and deconcentration in a finite domain} \label{sec:dec_iteration}

In this section, we extend Section~6 of \cite{BKN2018} to the $N$-parameter forest fire process, by explaining how to handle the difficulties created by recoveries. Since this extension has substantial overlap with \cite{BKN2018}, we only sketch some parts and omit several details, in order to emphasize two new aspects of the arguments. As we just mentioned, these aspects come from the recovery mechanism, which makes the sequence of consecutive holes (denoted by $\tilde{\Lambda}_i$) more difficult to control.

We will use (this is the first aspect) the generalized results from Section~\ref{sec:stab_recov} to get such a control: roughly, the growth of a hole due to recovery is typically ``much smaller'' than the hole itself (see in particular Theorem~\ref{thm:M_thick}). But this still leaves another difficulty, which is the second aspect mentioned above: even though the growth of $\tilde{\Lambda}_i$ is small, it may grow outside the ``idealized'' hole (called $\Delta_i$ in Section~6 of \cite{BKN2018}) which had a nice and natural spatial Markov property, that we referred to as being a \emph{stopping set}. This aspect somewhat perturbs the reasonings in \cite{BKN2018} which are based on identifying successive stopping sets.

Moreover, it is not a priori clear that the recovered part does not contribute too much to the cluster sizes. If this contribution was larger than expected, this would lead to an earlier occurrence of the next burning, and hence a bigger than expected corresponding hole. We will handle this difficulty by using Lemma~4.1 in \cite{BKN2018} in a ``global'' way (i.e. not iteratively, but rather simultaneously from the beginning), thus showing that each $\tilde{\Lambda}_i$ is contained in $\Delta_i$.

All these issues are addressed in the proof of the following result, which is an analog (for $N$-forest fires) of Theorem~6.2 in \cite{BKN2018}. Its statement involves the exceptional scales $(m_k(N))_{k \geq 0}$ defined in Section~\ref{sec:map_next} (see \eqref{eq:def_tk} and \eqref{eq:def_mk}). The requirement $K \in (m_{k+2}(N), m_{k+5}(N))$ is the condition that the starting domain $\Lambda$ is ``neither too small nor too large''. It ensures that the function $\next_N$ can be iterated $k + O(1)$ times.

We want to emphasize that in this section, the time line is $[0,\infty)$, unlike in \cite{BKN2018} where it is $[0,1]$. We use implicitly a translation of the results of that paper to our setting, through the formula $p(t) = 1 - e^{-t}$ from \eqref{eq:relation_p_t}.

\begin{proposition} \label{prop:finite_domain}
Consider the $N$-forest fire process on the triangular lattice $\TT$, and the associated critical time $t_c = t_c(\TT)$. There exists $\deltaKMS > 0$ such that the following holds. For all $\ve > 0$, there exists $\eta = \eta(\ve) > 0$ such that for all $c_2 > c_1 > \alpha > 0$, there exists $k_0 = k_0(\ve, c_1, c_2, \alpha) \geq 1$ such that for all $k \geq k_0$, we have: for all sufficiently large $N$, all $K \in (m_{k+2}(N), m_{k+5}(N))$, and all simply connected $(\alpha K, \eta)$-approximable sets $\Lambda$ with $\Ball_{c_1 K} \subseteq \Lambda \subseteq \Ball_{c_2 K}$,
\begin{equation}
\PP_N^{(\Lambda)} \big( 0 \text{ burns before time } t_c + \deltaKMS \big) < \ve.
\end{equation}
\end{proposition}

\begin{proof}[Proof of Proposition~\ref{prop:finite_domain}]
Consider a level number $k \geq 1$, a starting scale $K \in (m_{k+2}(N), m_{k+5}(N))$, and a domain $\Lambda$ with $\Ball_{c_1 K} \subseteq \Lambda \subseteq \Ball_{c_2 K}$ for some $0 < c_1 < c_2$. We define the sequence $(\Lambda_i)_{i \geq 0}$, which are the successive holes around $0$ for the $N$-parameter frozen percolation process in $\Lambda$, starting from $\Lambda_0 = \Lambda$. We denote by $(t_i)_{i \geq 0}$ the corresponding freezing times. We also introduce the deterministic sequence $(K_i)_{i \geq 0}$, with $K_0 = K$ and $K_{i+1} = \next_N(K_i)$.

We need additional sequences from \cite{BKN2018}. First, we use the sequences $(\Lambda_i^{**})_{i \geq 0}$ and $(t_i^{**})_{i \geq 0}$ of random sets and times, respectively, which ``approximate'' the $\Lambda_i$'s and $t_i$'s (see (6.1) in \cite{BKN2018}). We also make use of another auxiliary sequence $(\Delta_i)_{i \geq 0}$ which is used in \cite{BKN2018} to show that $\Lambda_i^{**}$ indeed approximates $\Lambda_i$. Its construction involves a parameter $\bar{\delta}$: we take $\Delta_0 = \Lambda$, and
\begin{equation} \label{eq:def_Delta}
\Delta_{i+1} = \holes_{\Delta_i}(t_i^-), \quad \text{with $t_i^-$ defined by } |\Delta_i| \theta(t_i^-) = N (1 - \bar{\delta})
\end{equation}
(recall the definition of the strong hole of the origin in a finite domain, given just below \eqref{eq:def_hole_domain}). A useful feature of these sets $\Delta_i$ is that they are stopping sets, i.e. they can be determined by an exploration ``from outside''. Moreover, we introduce the sequence $(\tilde{\Lambda}_i)_{i \geq 0}$ of successive holes ($\tilde{\Lambda}_i$ is the hole immediately after the $i$th burning around $0$) in the $N$-parameter forest fire process under consideration (again, with $\tilde{\Lambda}_0 = \Lambda$). The corresponding times of burning are denoted by $(\tilde{t}_i)_{i \geq 0}$.

In Section~6 of \cite{BKN2018}, it was shown that for arbitrarily small $\bar{\delta}$ (the parameter appearing in the definition of $(\Delta_i)_{i \geq 0}$, see \eqref{eq:def_Delta}), $\hat{\eta}$ and $\hat{\beta}$, there exist $\eta$, $\delta$ and $\hat{\delta}$ such that: if $\Lambda$ is $(\delta K, \eta)$-approximable, then (for $N$ sufficiently large) with high probability, each $\Delta_i$ is $(\hat{\delta} K_i, \hat{\eta})$-approximable, and each $\Lambda_i$ and $\Lambda^{**}_i$, $1 \leq i \leq k$, stays ``close'' to $\Delta_i$ (and hence to each other). More precisely,
\begin{equation} \label{eq:Lambda_Delta_close}
\Big[ \Dint{\Delta_i}{\hat{\delta} K_i} \Big]_{(\hat{\beta} \hat{\delta} K_i)} \subseteq \Lambda_i, \Lambda^{**}_i \subseteq \Delta_i.
\end{equation}

Earlier in \cite{BKN2018} (see Section~5 there), it was shown that if $k$ and $N$ are large, $\Lambda^{**}_k$ and $t^{**}_k$ satisfy a deconcentration property. By \eqref{eq:Lambda_Delta_close}, such a deconcentration then also holds for $\Lambda_k$, which was the key of the proof in \cite{BKN2018} that for $N$-frozen percolation, $\PPP_N^{(\Lambda)}(0 \text{ freezes})$ is very small if $k$ and $N$ are large. We will show that also each $\tilde{\Lambda}_i$, $1 \leq i \leq k$, is close to $\Delta_i$, and hence to $\Lambda^{**}_i$, which will yield a deconcentration property for $\tilde{\Lambda}_k$. For this purpose, we use two results, stated as Lemma~\ref{lem:approx_BCKS_unif} and Lemma~\ref{lem:iteration_thick} below. The first one follows from the now-classical arguments in \cite{BCKS2001} on the size of the largest cluster in a box, while the second one relies on Theorem~\ref{thm:M_thick} about robustness of percolation holes through recoveries.

First, we note that from the proof of Lemma~6.1 in \cite{BKN2018}, we can see that we may assume that
\begin{equation} \label{eq:cond_eta_delta}
1 + \hat{\eta} < \frac{1 - \bar{\delta}/2}{1 - \bar{\delta}}
\end{equation}
(note that this is equivalent to $\hat{\eta} < \frac{\bar{\delta}}{2(1 - \bar{\delta})}$, so it is sufficient that $\hat{\eta} < \frac{\bar{\delta}}{2}$). From now on, we assume that \eqref{eq:cond_eta_delta} indeed holds.

We now introduce yet another parameter $\lambda > 0$, of which the use will become clear later (it is related to Theorem~\ref{thm:M_thick}). We take it small enough so that
\begin{equation} \label{eq:cond_eta_delta_lambda}
(1 + 8 \lambda)(1 + \hat{\eta}) < \frac{1 - \bar{\delta}/2}{1 - \bar{\delta}},
\end{equation}
which is possible thanks to \eqref{eq:cond_eta_delta} above.

We now turn to a uniform version of Lemma~4.1 in \cite{BKN2018} (recall that this lemma is a generalization of \eqref{eq:largest_cluster}, from boxes to approximable domains). For our purpose, we only need the upper bound from this result. For any domain $\calD$ and $t > t_c$, we define $\tilde{t} = \tilde{t}(\calD)$ by
\begin{equation} \label{eq:def_ptildeD}
\frac{1}{(1 + 8 \lambda)(1 + \hat{\eta})} | \calD | \theta(\tilde{t}) = (1 - \bar{\delta}) N.
\end{equation}

\begin{lemma} \label{lem:approx_BCKS_unif}
Let $k \geq 1$, $K \geq 1$. With probability tending to $1$ as $N \to \infty$, the following holds. For all $i \in \{1, \ldots, k\}$, and every domain $\calD \subseteq \Ball_{K \cdot K_i}$ which is the $\lambda \delta K_i$-thickening of a union of $\delta K_i$-blocks, the largest $\tilde{t}$-occupied cluster in $\calD$ has a volume $< N$.
\end{lemma}

\begin{proof}[Proof of Lemma~\ref{lem:approx_BCKS_unif}]
Note that, for any given $K$, $\delta$ and $i$, the number of candidates for $\calD$ is uniformly bounded (for instance, by the number of subsets of a set of, roughly, $\big( \frac{K}{\delta} \big)^2$ blocks). Hence, it suffices to show that for a given $\calD$ which is the $\lambda \delta K_i$-thickening of a union of $\delta K_i$-blocks, as in the statement of the lemma, its largest $\tilde{t}$-occupied cluster has a volume $< N$ with high probability as $N \to \infty$.

This result for a fixed $\calD$ then follows easily from the reasonings in \cite{BCKS2001}, as for Lemma~4.1 in \cite{BKN2018}, since
$$| \calD | \theta(\tilde{t}) = (1 + 8 \lambda)(1 + \hat{\eta}) (1 - \bar{\delta}) N \leq \bigg( 1 - \frac{\bar{\delta}}{2} \bigg) N$$
(where the equality follows from \eqref{eq:def_ptildeD}, and the inequality from \eqref{eq:cond_eta_delta_lambda}). This completes the proof of Lemma~\ref{lem:approx_BCKS_unif}.
\end{proof}

The other lemma that we will use is the following. In part (ii) of it, we use the notion of $\sigma$-modified strong hole defined in \eqref{eq:def_hole_recov}, where $\sigma$ is given by \eqref{eq:def_sigma_FF}.

\begin{lemma} \label{lem:iteration_thick}
Let $C \geq 1$. With probability tending to $1$ as $N \to \infty$, we have that the following holds for each $1 \leq i \leq k$, each domain $\calD$ with $\Ball_{\frac{K_i}{C}} \subseteq \calD \subseteq \Ball_{C K_i}$, and each $t > t_c$ with $L(t) \in \big( \frac{K_{i+1}}{C}, C K_i \big)$:
\begin{enumerate}[(i)]
\item $\holes_{\calD}(t) = \holes(t)$,

\item $\holes(t)^{\sigma} \subseteq \big[ \holes(t) \big]^{(\lambda \hat{\delta} K_{i+1})}$.
\end{enumerate}
\end{lemma}

\begin{proof}[Proof of Lemma~\ref{lem:iteration_thick}]
It suffices to prove the result for a fixed $i \in \{1, \ldots, k\}$. Further, by monotonicity (in the domain $\calD$) of $\holes_{\calD}(t)$, it is enough to prove the result for $\calD = \Ball_{\frac{K_i}{C}}$, so we fix this domain $\calD = \calD(N)$. Further, with $\ol{t} = \ol{t}(N)$ defined by $L(\ol{t}) = C K_{i+1}$, we observe that if there is a $\ol{t}$-occupied circuit in $\Ann_{\frac{K_i}{2C},\frac{K_i}{C}}$ which has a $\ol{t}$-occupied path to $\infty$, then $\holes_{\calD}(t) = \holes(t)$ for all $t$ in the mentioned interval. Finally, this event has a probability tending to $1$ as $N \to \infty$ (since $L(\ol{t}) = C K_{i+1} \ll K_i$), which proves the first part (i).

We now turn to the second part (ii). Again, we only need to prove it for a fixed $i \in \{1, \ldots, k\}$. Since the characteristic lengths of the $t$-values considered (for a fixed $i$) are of the same order (they differ by a factor at most $C^2$ from each other), and since
$$\lambda \hat{\delta} K_{i+1} \asymp L(\ol{t}) \gg (L(\ol{t}))^{\upsilon}$$
(for $\upsilon$ defined in \eqref{eq:choice_upsilon}), the result follows from Theorem~\ref{thm:M_thick} (also using the observations in Section~\ref{sec:stab_FF} about how to apply this theorem to forest fires). The proof is thus complete.
\end{proof}

Lemma~\ref{lem:iteration_thick} (i) immediately implies that with high probability, we have
$$\Delta_{i+1} = \holes(t_i^-).$$
We start with $i=0$. Recall that we have already taken $\Delta_0 = \Lambda_0 = \tilde{\Lambda}_0 = \Lambda_0^{**} = \Lambda$. Clearly, we get $\tilde{\Lambda}_1 = \Lambda_1$ so by Lemma~6.1 of \cite{BKN2018}, we have, with high probability,
\begin{equation} \label{eq:Lambda_tilde_1}
\Big[ \Dint{\Delta_1}{\hat{\delta} K_1} \Big]_{(\hat{\beta} \hat{\delta} K_1)} \subseteq \tilde{\Lambda}_1 \subseteq \Delta_1,
\end{equation}
where (see \eqref{eq:def_Delta})
$$\Delta_1 = \holes(t_0^-), \quad \text{with } |\Delta_0| \theta(t_0^-) = N (1 - \bar{\delta}).$$
From the proof of Lemma~6.1 in \cite{BKN2018}, we also have that w.h.p.,
$$\tilde{\Lambda}_1 = \Lambda_1 = \holes_{\Lambda}(\tilde{t}_0) = \holes(\tilde{t}_0),$$
with $\tilde{t}_0 \in (t_0^-,t_0^+)$, where $t_0^+$ is given by
$$\theta(t_0^+) \bigg| \Big[ \Dint{\Lambda}{\hat{\delta} K_0} \Big]_{(\hat{\beta} \hat{\delta} K_0)} \bigg| = N(1 + \bar{\delta}).$$
We also have
$$L(t_0^-) \asymp L(t_0^+) \asymp K_1.$$
Now, let $\tilde{t}_1$ be the next burning time. We would like to get an analog of \eqref{eq:Lambda_tilde_1}. A difficulty for the forest-fire process (compared with frozen percolation) is that between times $\tilde{t}_0$ and $\tilde{t}_1$, recovery takes place, by which $\tilde{\Lambda}_1$ typically ``expands''. For the analog of the first inclusion in \eqref{eq:Lambda_tilde_1}, this is not a problem: very similar arguments as in \cite{BKN2018} (see around (6.15) there) give that
\begin{equation} \label{eq:Lambda_2}
\tilde{\Lambda}_2 \supseteq \Big[ \Dint{\Delta_2}{\hat{\delta} K_2} \Big]_{(\hat{\beta} \hat{\delta} K_2)}.
\end{equation}
For the analog of the other inclusion in \eqref{eq:Lambda_tilde_1}, we use Lemmas~\ref{lem:approx_BCKS_unif} and \ref{lem:iteration_thick}. First, note that the burning occurs within a ``macroscopic'' time after $\tilde{t}_0$ (before entering the supercritical regime), so in particular before time $t_c + \deltaKMS$. From Lemma~\ref{lem:iteration_thick} (ii), we have that, by then, the domain $\tilde{\Lambda}_1$ together with the extra region due to recovery (i.e. a subset of $\holes(\tilde{t}_0)^{\sigma}$) is w.h.p. contained in $\big[ \holes(\tilde{t}_0) \big]^{(\lambda \hat{\delta} K_1)} = \big[ \tilde{\Lambda}_1 \big]^{(\lambda \hat{\delta} K_1)}$, which by \eqref{eq:Lambda_tilde_1} is contained in $\big[ \Dout{\Delta_1}{\hat{\delta} K_1} \big]^{(\lambda \hat{\delta} K_1)}$. Recall that $t_1^-$ satisfies by definition (see \eqref{eq:def_Delta})
\begin{equation} \label{eq:t_1-}
|\Delta_1| \theta(t_1^-) = N(1 - \bar{\delta}).
\end{equation}
Let $\tilde{t} = \tilde{t}(\calD)$ be as in Lemma~\ref{lem:approx_BCKS_unif}, with
$$\calD = \Big[ \Dout{\Delta_1}{\hat{\delta} K_1} \Big]^{(\lambda \hat{\delta} K_1)}$$
(which contains $\tilde{\Lambda}_1$, together with its recovered area). From that lemma, we get that w.h.p., the largest $\tilde{t}$-occupied cluster in the above domain $\calD$ has volume $< N$. Moreover, by the definition of $\tilde{t}$, combined with \eqref{eq:t_1-}, we have
$$|\Delta_1| \theta(t_1^-) = \frac{1}{(1 + 8 \lambda)(1 + \hat{\eta})} | \calD | \theta(\tilde{t}),$$
so
\begin{equation} \label{eq:Delta_t_1}
\frac{\theta(\tilde{t})}{\theta(t_1^-)} = \frac{|\Delta_1|}{|\calD|} (1 + 8 \lambda)(1 + \hat{\eta}).
\end{equation}
However, we have
\begin{equation} \label{eq:Delta_D_1}
|\calD| = \Big| \Big[ \Dout{\Delta_1}{\hat{\delta} K_1} \Big]^{(\lambda \hat{\delta} K_1)} \Big| \leq (1 + 8 \lambda) \big| \Dout{\Delta_1}{\hat{\delta} K_1} \big| \leq (1 + 8 \lambda) (1 + \hat{\eta}) |\Delta_1|,
\end{equation}
where the first inequality comes from \eqref{eq:thick_bound}, and the second one uses that $\Delta_1$ is $(\hat{\delta} K_1, \hat{\eta})$-approximable. Combining \eqref{eq:Delta_t_1} and \eqref{eq:Delta_D_1} yields $\theta(\tilde{t}) \geq \theta(t_1^-)$, and hence $\tilde{t} \geq t_1^-$. From this, and the earlier observation about the largest $\tilde{t}$-occupied cluster in $\calD$, we get that w.h.p., the largest $t_1^-$-occupied cluster in $\calD$ has volume $< N$, so $\tilde{t}_1 > t_1^-$, and hence (using also Lemma~\ref{lem:iteration_thick} (i))
$$\tilde{\Lambda}_2 \subseteq \holes(t_1^-) = \Delta_2.$$
Together with \eqref{eq:Lambda_2}, this gives
$$\Big[ \Dint{\Delta_2}{\hat{\delta} K_2} \Big]_{(\hat{\beta} \hat{\delta} K_2)} \subseteq \tilde{\Lambda}_2 \subseteq \Delta_2,$$
as desired. Continuing in this way, we get that w.h.p., each $\tilde{\Lambda}_i$ is close to $\Delta_i$, for $i = 1, \ldots, k$, and hence also close to $\Lambda_i$ and $\Lambda_i^{**}$. In particular, we get the analog of the deconcentration result in \cite{BKN2018} (6.21), with $\tilde{\Lambda}_k$ instead of $\Lambda_k^{**}$.

The proof of Theorem~6.2 in \cite{BKN2018} was then completed by applying a result (for $N$-frozen percolation) from \cite{BN2017a}. An alternative way to complete that proof is, roughly speaking, to iterate a few more steps, showing that (with high probability) the origin is eventually in a (non-empty) hole with a radius much smaller than $m_1(N)$ (which is roughly $\sqrt{N}$), and hence that its cluster does not freeze. This alternative way, i.e. adding a few steps (of course, now using Lemmas~\ref{lem:approx_BCKS_unif} and \ref{lem:iteration_thick} at each step, as earlier in this proof), also works for our situation, thus completing the proof of Proposition~\ref{prop:finite_domain}.
\end{proof}

\subsection{Comparison to the full lattice and completion of the proof} \label{sec:dec_proof}

We now prove Theorem~\ref{thm:main_result}, and we will then explain quickly how to get Theorem~\ref{thm:main_result_finite}. The strategy is of the same spirit as that in Section~7 of \cite{BKN2018}; however, special care is (again) needed to deal with the recoveries. We show that the forest fire process produces, at some stage, a very large suitable hole around $0$, so that we are then more or less in the situation of Proposition~\ref{prop:finite_domain}.

We start with a technical lemma which is used several times in the proof.

\begin{lemma} \label{lem:apriori_L_theta}
There exist $C'_1$, $C'_2 > 0$ such that: for all $t_c < t_1 < t_2$,
\begin{equation} \label{eq:apriori_L_theta}
C'_1 \bigg( \frac{L(t_2)}{L(t_1)} \bigg)^{\frac{1}{2}} \leq \frac{\theta(t_1)}{\theta(t_2)} \leq C'_2 \bigg( \frac{L(t_2)}{L(t_1)} \bigg)^{\beta},
\end{equation}
where $\beta > 0$ is from \eqref{eq:1arm}, as well as
\begin{equation} \label{eq:apriori_L}
C'_1 \bigg( \frac{t_1 - t_c}{t_2 - t_c} \bigg)^{\frac{1}{\beta}} \leq \frac{L(t_2)}{L(t_1)} \leq C'_2 \bigg( \frac{t_1 - t_c}{t_2 - t_c} \bigg)^{\frac{1}{2-\beta}}.
\end{equation}
\end{lemma}

\begin{proof}[Proof of Lemma~\ref{lem:apriori_L_theta}]
On the one hand, \eqref{eq:apriori_L_theta} can be obtained easily by combining \eqref{eq:equiv_theta}, \eqref{eq:quasi_mult} and \eqref{eq:1arm}. On the other hand, \eqref{eq:apriori_L} can be deduced from \eqref{eq:equiv_L}, \eqref{eq:quasi_mult}, \eqref{eq:1arm} and \eqref{eq:4arm}.
\end{proof}

\begin{proof}[Proof of Theorem~\ref{thm:main_result}]
Let $\ve > 0$, and consider $\eta$ associated with $\frac{\ve}{100}$ through Proposition~\ref{prop:finite_domain}. We then consider some $K = K(\ve)$ large enough, that we explain how to choose toward the end of the proof (see \eqref{eq:uncertainty_t+} below). Roughly speaking, $K$ is produced by the quantiles at level $\frac{\ve}{100}$ for the largest cluster in sufficiently large boxes. For this value of $K$, we let $\delta = \delta(\frac{\ve}{100},\eta,K)$ coming from Lemma~\ref{lem:approx_hole_unif}, so that the following holds: for all $t_c < t < t_c + \ol{\delta}$,
\begin{equation} \label{eq:starting_hole1}
\PP \big( \text{for all } \tilde{t} \in [\dol{t}^{(K)}, \dul{t}^{(K)}], \: \holes(\tilde{t}) \text{ is } (\delta L(t), \eta)\text{-approximable} \big) \geq 1 - \frac{\ve}{100}.
\end{equation}
Here, we adapt the earlier notations $\dol{p}^{(K)}$ and $\dul{p}^{(K)}$ (see \eqref{eq:def_pK}) in an obvious way (from a percolation parameter $p > p_c$ to a ``time'' $t > t_c$). In addition, we can deduce from \eqref{eq:bound_holes_unif} the existence of $c_1$ and $c_2$ (depending on $K$) so that
\begin{equation} \label{eq:starting_hole2}
\PP \big(\text{for all } \tilde{t} \in [\dol{t}^{(K)}, \dul{t}^{(K)}], \: \Ball_{c_1 L(\tilde{t})} \subseteq \holes(\tilde{t}) \subseteq \Ball_{c_2 L(\tilde{t})} \big) \geq 1 - \frac{\ve}{100}.
\end{equation}
Applying Proposition~\ref{prop:finite_domain} with $\frac{\ve}{100}$ and the corresponding $\eta$, and this choice of $c_1$, $c_2$, and $\delta$, produces a $k_0 \geq 1$. We then consider $r_0 = r_0(N) := m_{k_0 + 5}(N)$, and the corresponding time $t_0 = t_0(N) > t_c$, satisfying $L(t_0) = r_0$ (in the remainder of the proof, we often drop the dependence on $N$ for notational convenience). We let
\begin{equation} \label{eq:def_t1t2}
t_0 = t_c + \ve_0, \: \: t_1 := \widehat{t_0} = t_c + \ve_1, \: \: \text{and } t_2 := \widehat{\widehat{t_0}} = t_c + \ve_2.
\end{equation}
We then introduce
\begin{equation} \label{eq:def_ult0}
\ul{t}_0 = t_c + \ul{\ve}_0 := t_c + (\ve_0)^{\frac{1}{2}} (\ve_1)^{\frac{1}{2}},
\end{equation}
and 
\begin{equation} \label{eq:def_olult1}
\ol{t}_1 = t_c + \ol{\ve}_1 := t_c + \ol{\kappa}_1 \, \ve_1 \quad \text{and} \quad \ul{t}_1 = t_c + \ul{\ve}_1 := t_c + \ul{\kappa}_1 \, \ve_1,
\end{equation}
for some $\ol{\kappa}_1 \in (0,1)$ and $\ul{\kappa}_1 > 1$ which depend only on $\ve$, chosen so as to ensure that for the largest cluster in $\Ball_{r_0}$, the following holds:
\begin{equation}
\PP \big( \text{at $\ol{t}_1$, } |\lclus_{\Ball_{r_0}}| < N \big) \geq 1 - \frac{\ve}{100} \quad \text{and} \quad \PP \big( \text{at $\ul{t}_1$, } |\lclus_{\Ball_{r_0}}| \geq N \big) \geq 1 - \frac{\ve}{100}.
\end{equation}
This can be done thanks to Lemma~\ref{lem:apriori_L_theta}, since
\begin{equation} \label{eq:t0_t1}
c_{\TT} L(t_0)^2 \theta(t_1) = N
\end{equation}
(from $t_1 = \widehat{t_0}$ and \eqref{eq:t_that}) and $r_0 = L(t_0)$. Note that
$$t_0 < \ul{t}_0 < \ol{t}_1 < t_1 < \ul{t}_1 < t_2$$
for all $N$ sufficiently large.

We consider the event
$$E_0 := \big\{ \text{at $t_0$, } \circuitevent^* \big( \Ann_{\mu L(t_0), L(t_0)} \big) \text{ occurs} \big\},$$
where we choose $\mu = \mu(\ve) > 0$ sufficiently small so that
\begin{equation}
\PP(E_0) \geq 1 - \frac{\ve}{100}.
\end{equation}
Indeed, this is possible thanks to \eqref{eq:RSW}.

In addition, we want to create a $t_c$-occupied circuit in $\Ann_{2 \mu \mu' L(t_0), \frac{\mu}{2} L(t_0)}$, for some $\mu' > 0$ sufficiently small, which will burn and create a barrier up to time $t_c + \deltaKMS$. For this purpose, Theorem~\ref{thm:sdp_circ} turns out to be enough, and the subsequent results are not needed at this stage. Clearly, $\mu'$ can be chosen small enough (and $< \frac{1}{8}$) so that
$$E'_0 := \big\{ \text{at $t_c$, } \circuitevent \big( \Ann_{2 \mu \mu' L(t_0), \frac{\mu}{2} L(t_0)} \big) \text{ occurs} \big\}$$
has a probability at least $1 - \frac{\ve}{100}$ (again, from \eqref{eq:RSW}). We thus introduce the outermost $t_c$-occupied circuit $\circuit$ in $\Ann_{2 \mu \mu' L(t_0), \frac{\mu}{2} L(t_0)}$ (we let $\circuit = \dout \Ball_{2 \mu \mu' L(t_0)}$ if no such circuit exists).

We now consider temporarily a modified forest fire process, where clusters entirely contained in $\Ball_{r_0}$ do not burn (even if they reach a volume larger than $N$). From estimates on $|\lclus_{\Ball_{r_0}}|$, we can see that this new process coincides with the original process up to time $\ol{t}_1$ (with probability at least $1 - \frac{\ve}{100}$). It is introduced (as well as a similar modified process considered later) to ensure that the hole created at the end of the proof is a stopping set.

Note that all vertices of $\circuit$ have to burn together (if they do burn) at a time $> t_0$.

\smallskip

\underline{Two cases:} We distinguish the following cases, depending on whether $\circuit$ has already burnt at the intermediate time $\ul{t}_0$ in the modified process.

\begin{itemize}
\item Case~1: $\circuit$ burns at or before time $\ul{t}_0$ (so it must be connected to $(\Ball_{r_0})^c$). In this case, we denote by $t'$ its burning time, and we let $t_0^* = \ul{t}_0$. We consider the event
$$\ul{E}_0 := \big\{ \text{at $\ul{t}_0$, } \circuitevent^* \big( \Ann_{\mu L(\ul{t}_0), L(\ul{t}_0)} \big) \text{ occurs} \big\} ,$$
which satisfies
$$\PP(\ul{E}_0) \geq 1 - \frac{\ve}{100}$$
(from the way in which $\mu$ was chosen). Note that if $E_0$, $E'_0$ and $\ul{E}_0$ occur, then at time $t'$, $0$ lies in an ``island'' whose boundary is contained in $\Ann_{\mu L(\ul{t}_0), \frac{\mu}{2} L(t_0)}$. We denote this island by $\holes'$, i.e. the (strong) hole containing $0$ in the complement of the cluster (just before time $t'$) of $\circuit$.

In addition, Theorem~\ref{thm:sdp_circ} (applied with $\mu \mu' L(t_0) < 2 \mu \mu' L(t_0) < \frac{\mu}{2} L(t_0) < \mu L(t_0)$) ensures that
$$\big\{ \text{at all $t \in [t', t_c + \deltaKMS]$, } \arm_1( \Ann_{\mu \mu' L(t_0), \mu L(t_0)} ) \text{ does not occur} \big\}$$
has a probability at least $1 - \frac{\ve}{100}$ (for all $N$ large enough). This means that $0$ is surrounded by a circuit lying in $\Ann_{\mu \mu' L(t_0), \mu L(t_0)}$ which remains vacant (in the forest fire process) throughout the whole interval $[t', t_c + \deltaKMS]$ (so at time $t_0^*$, in particular). We let
\begin{equation}
\ul{r}_0^* := \mu L(\ul{t}_0) \quad \text{and} \quad \ol{r}_0^* := \mu L(t_0).
\end{equation}

\item Case~2: $\circuit$ has not burnt yet at time $\ul{t}_0$. We consider the event
$$\ul{E}_1 := \big\{ \text{at $\ul{t}_1$, } \circuitevent^* \big( \Ann_{\mu L(\ul{t}_1), L(\ul{t}_1)} \big) \text{ occurs} \big\},$$
and we have
$$\PP(\ul{E}_1) \geq 1 - \frac{\ve}{100}.$$
We will show that $\circuit$ burns before time $\ul{t}_1$, so $\ul{E}_1$ will ensure that when this burning occurs, the island around $0$ contains $\Ball_{\mu L(\ul{t}_1)}$. We also introduce
$$E'_1 := \big\{ \text{at $t_c$, } \circuitevent \big( \Ann_{2 C L(\ul{t}_0), \frac{C^2}{2} L(\ul{t}_0)} \big) \text{ occurs} \big\} \cap \big\{ \text{at $\ul{t}_0$, } \dout \Ball_{2 C L(\ul{t}_0)} \lra \circuit \big\},$$
for some $C = C(\ve) \geq 1$ sufficiently large so that
$$\PP(E'_1) \geq 1 - \frac{\ve}{100}.$$
It is indeed possible to choose such a $C$, thanks to \eqref{eq:RSW} and \eqref{eq:exp_decay}. Additionally, we consider a $\ul{t}_0$-occupied net with mesh $\kappa_0 = (L(t_0) L(\ul{t}_0))^{\frac{1}{2}}$ ($\gg L(\ul{t}_0)$ as $N \to \infty$) covering $\Ball_{r_0}$. Observe that
$$\frac{\kappa_0}{L(\ul{t}_0)} = \frac{L(t_0)}{\kappa_0} = \bigg( \frac{L(t_0)}{L(\ul{t}_0)} \bigg)^{\frac{1}{2}} \gg N^{\upsilon_0}$$
as $N \to \infty$, for some $\upsilon_0 > 0$ small enough (depending on $k_0$). Hence, Lemma~\ref{lem:exist_net} ensures that at time $\ul{t}_0$, $\net(L(t_0), \kappa_0)$ occurs with probability at least $1 - \frac{\ve}{100}$ (and $r_0 = L(t_0)$). Moreover, all cells of such a net have a diameter at most $4 \kappa_0$, and it follows from Lemma~\ref{lem:largest_cluster_exp} that at time $\ul{t}_1$, none of them contains a cluster with a volume at least $N$ (also with probability at least $1 - \frac{\ve}{100}$).

We again consider a modified forest fire process, now where clusters entirely contained in $\Ball_{C^2 L(\ul{t}_0)}$ do not burn. Clearly, \eqref{eq:largest_cluster} implies that this new process coincides with the original process up to time $\ul{t}_1$ (with high probability), and that necessarily, $\circuit$ burns in $(\ul{t}_0, \ul{t}_1]$. When this happens, the $t_c$-occupied circuit in $\Ann_{2 C L(\ul{t}_0), \frac{C^2}{2} L(\ul{t}_0)}$ provided by $E'_1$ (which is occupied and connected to $\circuit$ at time $\ul{t}_0$) burns together with it. We denote by $t'$ the corresponding burning time, at which $0$ lies in an island with a boundary contained in $\Ann_{\mu L(\ul{t}_1), \frac{C^2}{2} L(\ul{t}_0)}$. In this case, we let $t_0^* = \ul{t}_1$, and we again denote by $\holes'$ the associated hole of $0$ (in the complement of the cluster of $\circuit$ just before time $t'$).

Furthermore, Theorem~\ref{thm:sdp_circ} (similarly to Case~1) ensures that for all $N$ sufficiently large,
$$\big\{ \text{at all $t \in [t', t_c + \deltaKMS]$, } \arm_1( \Ann_{C L(\ul{t}_0), C^2 L(\ul{t}_0)} ) \text{ does not occur} \big\}$$
has a probability at least $1 - \frac{\ve}{100}$. In other words, there exists a circuit around $0$ which lies in $\Ann_{C L(\ul{t}_0), C^2 L(\ul{t}_0)}$ and remains vacant during the whole interval $[t', t_c + \deltaKMS]$ (which contains $t_0^*$). We denote
\begin{equation}
\ul{r}_0^* := \mu L(\ul{t}_1) \quad \text{and} \quad \ol{r}_0^* := C^2 L(\ul{t}_0).
\end{equation}

\end{itemize}

\underline{Conclusion for both cases:} We observe that with high probability, the following properties hold true, with of course different definitions for $t_0^*$, $\ul{r}_0^*$ and $\ol{r}_0^*$ (and, strictly speaking, also for $t'$ and $\holes'$) in the two cases.
\begin{itemize}
\item The cluster of $\circuit$ burns at time $t'$, where $t' \in (t_0, \ul{t}_1]$, leaving $0$ in the island $\holes'$.

\item On the one hand, $\Ball_{\mu L(t_0^*)} = \Ball_{\ul{r}_0^*} \subseteq \holes'$. This ensures that we have enough free space around $0$ to continue the construction after time $t'$.

\item On the other hand, there exists a circuit $\circuit^*$ in $\Ball_{\ol{r}_0^*}$ which contains $\holes'$ in its interior (so in particular, it surrounds $0$), and remains vacant (in the forest fire process) throughout the interval $[t', t_c + \deltaKMS]$ ($\ni t_0^*$).
\end{itemize}

Moreover, we claim that $\ol{r}_0^*$ is ``small'' in the following sense: for some $\upsilon_1 > 0$ (depending on $k_0$),
\begin{equation} \label{eq:small_vol_Lambda*1}
( \ol{r}_0^* )^2 \theta ( t_0^* ) N^{\upsilon_1} \ll N \quad \text{as $N \to \infty$,}
\end{equation}
which can be seen from our choice of times, using that $L(t_0) = r_0 = m_{k_0 + 5}$. Indeed, in Case~1, we have
$$( \ol{r}_0^* )^2 \theta ( t_0^* ) \asymp L(t_0)^2 \theta(\ul{t}_0),$$
while on the other hand,
$$c_{\TT} L(t_0)^2 \theta(t_1) = N$$
(from \eqref{eq:t0_t1}). It then suffices to compare $\theta(\ul{t}_0)$ and $\theta(t_1)$, by using that $\ul{t}_0$ is ``much earlier'' than $t_1 = \widehat{t_0}$. More precisely, it follows from Lemma~\ref{lem:apriori_L_theta}, combined with \eqref{eq:def_t1t2} and \eqref{eq:def_ult0}, that
$$\frac{\theta(\ul{t}_0)}{\theta(t_1)} \leq C' \bigg( \frac{\ve_0}{\ve_1} \bigg)^{\frac{\beta}{2(2-\beta)}},$$
and
$$\frac{\ve_0}{\ve_1} \ll N^{-\upsilon_2}$$
for some $\upsilon_2 > 0$ small enough (e.g. using again Lemma~\ref{lem:apriori_L_theta}, and $\delta_{k_0 + 4} < \delta_{k_0 + 5}$, in the notation of the paragraph below \eqref{eq:def_mk}). This yields \eqref{eq:small_vol_Lambda*1} in this case. In Case~2,
$$( \ol{r}_0^* )^2 \theta ( t_0^* ) \asymp L(\ul{t}_0)^2 \theta(\ul{t}_1),$$
and we can proceed in a similar way as in Case~1, using now that $\ul{t}_1$ satisfies $L(\ul{t}_1) \asymp L(t_1)$ (from \eqref{eq:def_olult1}), so it is much earlier than $\widehat{\ul{t}_0}$ (since $t_0$ is much earlier than $\ul{t}_0$, and $t_1 = \widehat{t_0}$). Hence, the claim \eqref{eq:small_vol_Lambda*1} holds true in this case as well.

Next, we let $t_1^* = \widehat{t_0^*} = t_c + \ve_1^*$, and we consider $\ul{t}_1^* = t_c + \ul{\ve}_1^* := t_c + \ul{\kappa}_1^* \, \ve_1^*$, for some $\ul{\kappa}_1^*$ chosen large enough so that
\begin{equation} \label{eq:def_ult1*}
\big| \Ball_{\frac{\mu}{2} L(t_0^*)} \big| \theta(\ul{t}_1^*) \geq 2N.
\end{equation}
Lemma~\ref{lem:apriori_L_theta} indeed allows us to do so, since $c_{\TT} L(t_0^*)^2 \theta(t_1^*) = N$ (from $t_1^* = \widehat{t_0^*}$ and \eqref{eq:t_that}).

We then introduce $t^{**}$ ($> t_0^*$) via the relation
\begin{equation} \label{eq:def_t**}
\theta ( t^{**} ) = \theta ( t_0^* ) N^{\upsilon_1}.
\end{equation}
It follows immediately from \eqref{eq:small_vol_Lambda*1} and \eqref{eq:def_t**} that $| \Ball_{\ol{r}_0^*} | \theta ( t^{**} ) \ll N$ as $N \to \infty$. Hence, \eqref{eq:largest_cluster} implies that with probability at least $1 - \frac{\ve}{100}$, no cluster burns inside the aforementioned vacant circuit $\circuit^*$ before time $t^{**}$ (for all $N$ large enough). We also let $\Gamma = \Gamma(N)$ be such that
\begin{equation} \label{eq:def_Gamma}
\mu L(t_0^*) = \Gamma L(t^{**}),
\end{equation}
and \eqref{eq:def_t**} (combined with \eqref{eq:apriori_L_theta}) implies that for some $C_1 > 0$,
\begin{equation} \label{eq:GammaN}
\Gamma(N) \geq C_1 N^{2 \upsilon_1}.
\end{equation}
Hence, the event
\begin{equation}
E_2 := \Big\{ \text{at $t^{**}$, } \circuitevent \Big( \Ann_{\frac{\Gamma}{2} L(t^{**}), \Gamma L(t^{**})} \Big) \cap \circuitevent \Big( \Ann_{\frac{\sqrt{\Gamma}}{2} L(t^{**}), \sqrt{\Gamma} L(t^{**})} \Big) \cap \arm_1 \Big( \Ann_{\frac{\sqrt{\Gamma}}{4} L(t^{**}), \Gamma L(t^{**})} \Big) \text{ occurs} \Big\}
\end{equation}
satisfies, for all $N$ large enough,
\begin{equation}
\PP(E_2) \geq 1 - \frac{\ve}{100}.
\end{equation}
We also consider, using Lemma~\ref{lem:exist_net}, a $t^{**}$-occupied net covering $\Ball_{\mu L(t_0^*)}$, with mesh $\kappa^{**} = \sqrt{\Gamma} L(t^{**})$. For this purpose, note that
$$\frac{\kappa^{**}}{L(t^{**})} = \frac{\mu L(t_0^*)}{\kappa^{**}} = \sqrt{\Gamma} \geq \sqrt{C_1} N^{\upsilon_1}$$
(using \eqref{eq:def_Gamma} and \eqref{eq:GammaN}). In combination with the exponential upper tail provided by Lemma~\ref{lem:largest_cluster_exp}, it will ensure that with probability at least $1 - \frac{\ve}{100}$, no cluster in $\Ball_{\frac{\mu}{2} L(t_0^*)}$ burns outside of the big ``structure'' produced by the occupied paths in $E_2$. Indeed, in any of the cells of the net, all clusters have a volume $\ll N$ up to time $\ul{t}_1^*$.

Let $\circuit^{**}$ be the outermost occupied circuit in $\Ann_{\frac{\sqrt{\Gamma}}{2} L(t^{**}), \sqrt{\Gamma} L(t^{**})}$ at time $t^{**}$, and denote
\begin{equation}
\text{for all } t \geq t^{**}, \quad Y_t := \big| \big\{ v \in V \setminus \ol{\circuit^{**}} \: : \: v \lra \circuit^{**} \big\}\big|
\end{equation}
in the modified $N$-forest fire process where the cluster of $\circuit^{**}$ does not burn. This quantity is nondecreasing in $t$, and it coincides with the analogous quantity in the original process, until $\circuit^{**}$ burns in that process (together with all occupied vertices connected to it). Moreover, we have $Y_{\ul{t}_1^*} \geq N$ with probability at least $1 - \frac{\ve}{100}$ (from \eqref{eq:def_ult1*} and \eqref{eq:largest_cluster}), which we can assume. This will ensure that the random times $\ol{t}$ and $\ul{t}$ defined below are smaller than $\ul{t}_1^*$, so that the cells in the net cannot reach a volume at least $N$. In the following, we condition on $\circuit^{**}$ and the birth process outside $\circuit^{**}$, so that $Y$ can be considered as deterministic.

We also consider $\ul{x}$, $\ol{x}$, $X > 0$ associated with $\frac{\ve}{100}$ by Lemma~\ref{lem:largest_cluster_quantiles}, so that for all $p > p_c$, and $n \geq X L(p)$,
\begin{equation} \label{eq:lclus_quant}
\PP_p \bigg( \bigg\{ \frac{\big| \lclus_{\Ball_n} \big|}{n^2 \theta(p)} \in (\ul{x}, \ol{x}) \bigg\} \cap \circuitarm \Big( \Ann_{\frac{1}{2} n, n} \, \big| \, \lclus_{\Ball_n} \Big) \bigg) \geq 1 - \frac{\ve}{100}.
\end{equation}
We will apply it at times $t \geq t^{**}$, for $n = \frac{\sqrt{\Gamma}}{2} L(t^{**})$ and $n = 2 \sqrt{\Gamma} L(t^{**})$. Observe that the condition on $n$ is satisfied for all $N$ large enough (because $\Gamma$ grows with $N$, see \eqref{eq:GammaN}).

We introduce the times
\begin{align*}
& \ol{t} := \inf \big\{ t \geq t^{**} \: : \: Y_t + \ol{x} \big( 2 \sqrt{\Gamma} L(t^{**}) \big)^2 \theta(t) \geq N \big\}\\[1mm]
\text{and} \quad & \ul{t} := \inf \Big\{ t \geq t^{**} \: : \: Y_t + \ul{x} \Big( \frac{\sqrt{\Gamma}}{2} L(t^{**}) \Big)^2 \theta(t) \geq N \Big\}.
\end{align*}
Clearly, $\ol{t} < \ul{t}$, and they are deterministic due to the earlier conditioning.

Moreover, it follows from \eqref{eq:lclus_quant} that the burning time $t^+$ of $\circuit^{**}$ satisfies
\begin{equation} \label{eq:t+1}
\PP (\ol{t} \leq t^+ \leq \ul{t}) \geq 1 - 2 \cdot \frac{\ve}{100}.
\end{equation}
Indeed, we have in particular
\begin{align*}
\PP \Big( \text{at $\ul{t}$, } & \Big| \lclus_{\Ball_{\frac{\sqrt{\Gamma}}{2} L(t^{**})}} \Big| \geq \ul{x} \Big( \frac{\sqrt{\Gamma}}{2} L(t^{**}) \Big)^2 \theta(\ul{t})\\
& \text{ and } \circuitarm \Big( \Ann_{\frac{\sqrt{\Gamma}}{4} L(t^{**}), \frac{\sqrt{\Gamma}}{2} L(t^{**})} \, \big| \, \lclus_{\Ball_{\frac{\sqrt{\Gamma}}{2} L(t^{**})}} \Big) \text{ occurs} \Big) \geq 1 - \frac{\ve}{100}
\end{align*}
and if this event occurs, the number of occupied vertices connected to $\circuit^{**}$ at time $\ul{t}$ is $\geq N$ (using in addition the occupied arm provided by $\arm_1 \big( \Ann_{\frac{\sqrt{\Gamma}}{4} L(t^{**}), \Gamma L(t^{**})} \big)$), so $t^+ \leq \ul{t}$. On the other hand,
\begin{align*}
\PP \Big( \text{at $\ol{t}$, } & \Big| \lclus_{\Ball_{2 \sqrt{\Gamma} L(t^{**})}} \Big| \leq \ol{x} \big( 2 \sqrt{\Gamma} L(t^{**}) \big)^2 \theta(\ol{t})\\
& \text{ and } \circuitarm \Big( \Ann_{\sqrt{\Gamma} L(t^{**}), 2 \sqrt{\Gamma} L(t^{**})} \, \big| \, \lclus_{\Ball_{2 \sqrt{\Gamma} L(t^{**})}} \Big) \text{ occurs} \Big) \geq 1 - \frac{\ve}{100},
\end{align*}
and if this event occurs, the number of vertices inside $\circuit^{**}$ connected to this circuit, at time $\ol{t}$, is $< \ol{x} \big( 2 \sqrt{\Gamma} L(t^{**}) \big)^2 \theta(\ol{t})$. Indeed, this number is at most the number of vertices in $\Ball_{\sqrt{\Gamma} L(t^{**})}$ connected to the boundary of this box, which is itself $< \big| \lclus_{\Ball_{2 \sqrt{\Gamma} L(t^{**})}} \big|$ (from the occurrence of $\circuitarm \big( \Ann_{\sqrt{\Gamma} L(t^{**}), 2 \sqrt{\Gamma} L(t^{**})} \, \big| \, \lclus_{\Ball_{2 \sqrt{\Gamma} L(t^{**})}} \big)$). We deduce that at time $\ol{t}$, the number of vertices connected to $\circuit^{**}$ is $< N$, so $t^+ \geq \ol{t}$. This completes the proof of \eqref{eq:t+1}, and we assume from now on that
\begin{equation} \label{eq:t+2}
\ol{t} \leq t^+ \leq \ul{t}.
\end{equation}

Using the fact that $Y_t$ is nondecreasing in $t$, we can obtain immediately that for some constant $\xi_1 = \xi_1(\ve)$, we have $\theta(\ul{t}) \leq \xi_1 \theta(\ol{t})$ (this reasoning is similar to the beginning of Step~6 in the proof of Proposition~7.2 in \cite{BKN2018}). Hence (using \eqref{eq:apriori_L_theta}), for some $\xi_2 = \xi_2(\ve)$,
\begin{equation} \label{eq:uncertainty_t+}
L(\ol{t}) \leq \xi_2 L(\ul{t}).
\end{equation}
In words, this constant $\xi_2$ (associated with $\frac{\ve}{100}$) quantifies the uncertainty on $L(t^+)$. We thus choose $K = \sqrt{\xi_2}$ in the beginning of the proof, so that \eqref{eq:starting_hole1} and \eqref{eq:starting_hole2} can be applied with $t$ defined by $L(t) = (L(\ol{t}) L(\ul{t}))^{\frac{1}{2}}$. Indeed, \eqref{eq:t+2} and \eqref{eq:uncertainty_t+} then ensure that
\begin{equation} \label{eq:bounds_t+}
t^+ \in [\dol{t}^{(K)}, \dul{t}^{(K)}].
\end{equation}

In particular, we conclude that the hole created around $0$ by the burning at time $t^+$ is $(\delta L(t), \eta)$-approximable. Also, note that by the construction it is a stopping set. Moreover, we can, by \eqref{eq:bounds_t+}, use Theorem~\ref{thm:M_thick}. So, the creation of this hole brings us in the same position as in the second step in the proof of Proposition~\ref{prop:finite_domain}. We can now repeat the same iteration argument as in that proof to complete the proof of Theorem~\ref{thm:main_result}.
\end{proof}

Theorem~\ref{thm:main_result_finite} can actually be obtained from the same reasonings, as we now briefly explain.

\begin{proof}[Proof of Theorem~\ref{thm:main_result_finite}]
It suffices to observe the following. Since the exponents associated with the scales $(m_k)$ satisfy $\delta_k \nearrow \delta_{\infty} = \frac{48}{91}$ as $k \to \infty$ (see the paragraph below \eqref{eq:def_mk}), we have: for every $k \geq 0$, there exists $\eta = \eta(k) > 0$ such that $N^{\frac{48}{91} - \eta} \gg m_{k+5}(N)$ as $N \to \infty$. Hence, for this specific value $\eta$, the condition $m(N) \geq N^{\frac{48}{91} - \eta}$ implies that $\Ball_{m(N)} \supseteq \Ball_{r_0(N)}$ for all $N$ large enough (with $r_0$ as in the beginning of the proof of Theorem~\ref{thm:main_result}), so that exactly the same construction can be repeated.
\end{proof}

\subsection{Brief discussion on frozen percolation with modified boundary rules} \label{sec:FP_modified}

As a concluding remark, we want to mention that the proofs of Theorems~\ref{thm:main_result} and \ref{thm:main_result_finite} can be adapted to other processes with a similar flavor. In particular, our reasonings can be applied to an alternative frozen percolation process, with \emph{modified boundary rules}. For this process, each vertex $v \in V_G$ can be in three states: vacant ($\eta_v = 0$), occupied ($\eta_v = 1$), or frozen ($\eta_v = -1$). Its definition is similar to that of volume-frozen percolation in Section~\ref{sec:processes}, except that when an occupied cluster reaches a volume $N$ and freezes, all its vertices become frozen, while the vertices along its outer boundary remain unaffected. In other words, these boundary vertices stay vacant immediately after the freezing time, and they may become occupied (and then possibly freeze) at later times (while in the original frozen percolation process, these vertices were kept in a vacant state). Hence, when a vacant vertex $v$ tries to change its state, say at time $t$, we consider the union of the occupied clusters adjacent to $v$ at time $t^-$:
$$\ol{\cluster}_{t^-}(v) := \bigcup_{\substack{v' \in V_G\\ v' \sim v}} \cluster_{t^-}(v').$$
If $|\{v\} \cup \ol{\cluster}_{t^-}(v)|\geq N$, then all vertices in $\{v\} \cup \ol{\cluster}_{t^-}(v)$ become frozen at time $t$. Otherwise, $v$ simply becomes occupied at time $t$, and in this case we have obviously $\cluster_t(v) = \{v\} \cup \ol{\cluster}_{t^-}(v)$.

This process is the $N$-parameter analog of the forest fire process without recovery, considered in \cite{BN2018}, where clusters are ignited according to a Poisson process with rate $\zeta \searrow 0$. It can also clearly be seen as an intermediate process
between volume-frozen percolation and the $N$-parameter forest fire process which is the focus of the present paper. The same reasonings can be followed practically step by step in this case, but redoing them seems to be required: because of the lack of monotonicity, we cannot simply use a ``sandwiching'' of the process between frozen percolation and $N$-forest fire.

Finally, note that Theorems~\ref{thm:main_result} and \ref{thm:main_result_finite} could also be obtained in a more direct way from the deconcentration results for frozen percolation, without using the properties derived in Sections~\ref{sec:positive_rec_time} and \ref{sec:stab_recov}. Indeed, this could be achieved from geometric considerations about the weak and strong holes, although it would require a whole new proof. Roughly speaking, modifying the boundary rules makes accessible from $0$ many additional vertices, and we have to take into account these vertices, which lie in the weak hole but not in the strong hole. In principle, we should thus follow the successive weak holes around $0$, instead of the strong holes.

However, we can observe (from a similar $6$-arm argument as in the proof of Lemma~\ref{lem:approx_hole}) that the extra vertices lie in ``dangling ends'' of the weak hole, and can only be reached through narrow corridors, so that they typically do not have a significant effect on the volume of the largest cluster in the strong hole. Hence, we could obtain an upper bound on the expected number of such vertices which are connected to the strong hole, and combining this bound with Markov's inequality should be enough to control their contribution. Once again, the strong holes play a central role, and not the weak holes.

\appendix

\section{Appendix: uniform approximability of percolation holes} \label{sec:proof_largest}

We now prove Lemma~\ref{lem:approx_hole_unif}.

\begin{proof}[Proof of Lemma~\ref{lem:approx_hole_unif}]
Let $\ve$, $\eta$ and $K$ be as in the statement. We follow the proof of Lemma~3.7 in \cite{BKN2018}. As in \eqref{eq:bound_holes_unif}, there exist $0 < \ul{\lambda} < \ol{\lambda}$, which depend only on $\ve$ and $K$, such that: for all $p > p_c$,
\begin{equation} \label{eq:approx_hole_unif_pf1}
\PP \big(\text{for all } \tilde{p} \in [\dol{p}^{(K)}, \dul{p}^{(K)}], \: \Ball_{\ul{\lambda} L(p)} \subseteq \holes(\tilde{p}) \subseteq \holew(\tilde{p}) \subseteq \Ball_{\ol{\lambda} L(p)} \big) \geq 1 - \frac{\ve}{4}.
\end{equation}
Let $p > p_c$. We denote by $A$ the event appearing in \eqref{eq:approx_hole_unif_pf1}, so that
\begin{equation} \label{eq:approx_hole_A1}
\PP(A^c) \leq \frac{\ve}{4},
\end{equation}
and from now on we assume that it holds. We also consider a small $\delta > 0$, that we explain how to choose appropriately later.

Let $\tilde{p} \in [\dol{p}^{(K)}, \dul{p}^{(K)}]$. In order to prove that $\holes(\tilde{p})$ is $(\delta L(p), \eta)$-approximable, we need to derive an upper bound on $\Dout{\holes(\tilde{p})}{\delta L(p)} \setminus \Dint{\holes(\tilde{p})}{\delta L(p)}$, which is a union of $\delta L(p)$-blocks. As in the proof of Lemma~3.7 in \cite{BKN2018}, we decompose this set into two disjoint subsets of blocks:
$$\Dout{\holes(\tilde{p})}{\delta L(p)} \setminus \Dint{\holes(\tilde{p})}{\delta L(p)} = \Lambda_1(\tilde{p}) \cup \Lambda_2(\tilde{p}),$$
where $\Lambda_1(\tilde{p})$ is the union of all $\delta L(p)$-blocks intersecting $\dout \holes(\tilde{p})$, and $\Lambda_2(\tilde{p})$ is the union of all $\delta L(p)$-blocks which are entirely contained in $\holes(\tilde{p})$, but are not connected inside $\holes(\tilde{p})$ to the block $S^{(\delta L(p))}_{0,0}$ centered on $0$. In addition, we let, for $i = 1, 2$,
\begin{equation}
\ol{\Lambda}_i := \bigcup_{\tilde{p} \in [\dol{p}^{(K)}, \dul{p}^{(K)}]} \Lambda_i(\tilde{p}).
\end{equation}

In order to handle $\Lambda_1(\tilde{p})$, we note that from the boundary of each block that it contains, there is a $\tilde{p}$-occupied arm to distance $L(p)$. This arm is also $\dul{p}^{(K)}$-occupied, so around the center $z$ of the block, the event $\arm_1(\Ann_{\frac{\delta}{2} L(p), L(p)}(z))$ occurs at $\dul{p}^{(K)}$. Note that under the event $A$, the number of choices for a block in $\ol{\Lambda}_1$ is at most $\big( \frac{2 \ol{\lambda} L(p)}{\delta L(p)} \big)^2 = \big( \frac{2 \ol{\lambda}}{\delta} \big)^2$. Hence,
\begin{align}
\EE \Big[ \big| \ol{\Lambda}_1 \big| \ind_A \Big] & \leq \bigg( \frac{2 \ol{\lambda}}{\delta} \bigg)^2 \PP_{\dul{p}^{(K)}} \big( \arm_1(\Ann_{\frac{\delta}{2} L(p), L(p)}(z)) \big) \cdot |\Ball_{\delta L(p)}| \nonumber \\
& \leq c_1 \delta^{-2} \cdot \bigg( \frac{\frac{\delta}{2} L(p)}{L(p)} \bigg)^{\beta} \cdot (\delta L(p))^2 = c_2 \delta^{\beta} L(p)^2, \label{eq:upper_bd_Lambda1}
\end{align}
where we used the a-priori upper bound from \eqref{eq:1arm} (as well as the fact that $L(\dul{p}^{(K)}) \asymp L(p)$) for the second inequality, and the constants $c_1$ and $c_2$ depend on $K$, and also, through $\ol{\lambda}$, on $\ve$.

We now turn to $\Lambda_2(\tilde{p})$, and derive a similar upper bound for the expectation of $|\ol{\Lambda}_2|$. As in \cite{BKN2018}, by max-flow min-cut and since $\holes(\tilde{p})$ is simply connected, we have that for every $v \in \Lambda_2(\tilde{p})$, there is a (not necessarily unique) $\delta L(p)$-block $B$, with center $z_B$, such that every path in $\holes(\tilde{p})$ from $v$ to $0$ intersects $\tilde{B} = \Ball_{\delta L(p)}(z_B)$. So, on the event $A$, $v$ is separated from $\Ball_{\ul{\lambda} L(p)}$ in $\holes(\tilde{p})$ by $\tilde{B}$. We say that ``$\tilde{B}$ separates $v$ in $\holes(\tilde{p})$ from $0$''.

By Remark~\ref{rem:arms_holes}, we have, in a similar way as in \cite{BKN2018}, that the following claim holds. If $A$ occurs, $v \in \holes(\tilde{p})$, and $B$ separates $v$ from $0$ in $\holes(\tilde{p})$, then the two events $\arm^{\dol{p}^{(K)}, \dul{p}^{(K)}}_6(\Ann_{\delta L(p), \ell \wedge \ul{\lambda} L(p)}(z_B))$ and $\arm^{\dol{p}^{(K)}, \dul{p}^{(K)}}_3(\Ann_{\ell \wedge \ul{\lambda} L(p), \ell \vee \ul{\lambda} L(p)}(z_B))$ hold, where $\ell = d(v, z_B)$ (and as above, $z_B$ is the center of $B$). In this proof only, the $6$- and $3$-arm events that we consider correspond to the sequences $\sigma = (ovvovv)$ and $\sigma = (ovv)$, respectively, where the arms are required to be $\dul{p}^{(K)}$-occupied (type $o$) or $\dol{p}^{(K)}$-vacant (type $v$). We will use the fact that their probabilities are of the same order (uniformly in $p > p_c$, once $K$ is fixed) as the corresponding arm events (in the same annuli, of course) at $p_c$.

As in \cite{BKN2018}, we first rule out the case when $z$ and $B$ are far apart. We let $\tilde{A} := \{$there exists a $\delta L(p)$-block $B$ and a vertex $v$ with $d(v, z_B) \geq \ul{\lambda} L(p)$ such that for some $\tilde{p} \in [\dol{p}^{(K)}, \dul{p}^{(K)}]$, $B$ separates $v$ in $\holes(\tilde{p})$ from $0\}^c$. By the above claim, we have
$$\PP(A \cap \tilde{A}^c) \leq \bigg( \frac{2 \ol{\lambda} L(p)}{\delta L(p)} \bigg)^2 \PP \big( \arm^{\dol{p}^{(K)}, \dul{p}^{(K)}}_6(\Ann_{\delta L(p), \ul{\lambda} L(p)}(z_B)) \big),$$
where the first factor is an upper bound on the number of choices for $B$ under the event $A$. Hence,
$$\PP(A \cap \tilde{A}^c) \leq c_1 \delta^{-2} \cdot \bigg( \frac{\delta L(p)}{\ul{\lambda} L(p)} \bigg)^{2 + \beta} = c_2 \delta^{\beta},$$
using the a-priori upper bound from \eqref{eq:6arm} (and again the fact that $L(\dul{p}^{(K)}) \asymp L(p)$). As before, the constants which appear only depend on $K$ and $\ve$. This probability can thus be made $\leq \frac{\ve}{4}$ by taking $\delta$ sufficiently small, and we assume it to be the case:
\begin{equation} \label{eq:approx_hole_A2}
\PP(A \cap \tilde{A}^c) \leq \frac{\ve}{4}.
\end{equation}

Now, we have
\begin{align*}
\EE \Big[ \big| \ol{\Lambda}_2 \big| \ind_{A \cap \tilde{A}} \Big] & \leq \sum_B \sum_{v : d(v,z_B) \leq \ul{\lambda} L(p)} \PP \big(B \text{ separates } v \text{ from } 0 \text{ in } \holes(\tilde{p}) \text{ for some } \tilde{p} \big),
\end{align*}
where the first sum is over all $\delta L(p)$-blocks $B$ contained in $\Ball_{\ol{\lambda} L(p)}$. From the same upper bound as before on the number of choices for $B$, and a summation, we get
\begin{align}
\EE \Big[ \big| \ol{\Lambda}_2 \big| \ind_{A \cap \tilde{A}} \Big] & \leq \bigg( \frac{2 \ol{\lambda} L(p)}{\delta L(p)} \bigg)^2 \sum_{r=0}^{\ul{\lambda} / \delta} \sum_{v : \frac{d(v,z_B)}{\delta L(p)} \in [r,r+1)} \PP \big(B \text{ separates } v \text{ from } 0 \text{ in } \holes(\tilde{p}) \text{ for some } \tilde{p} \big) \nonumber \\
& \leq c_3 \bigg( \frac{2 \ol{\lambda}}{\delta} \bigg)^2 \bigg( (\delta L(p))^2 \bigg( \frac{\delta L(p)}{\ul{\lambda} L(p)} \bigg)^{\beta/2} + \sum_{r=1}^{\ul{\lambda} / \delta} r (\delta L(p))^2 \bigg( \frac{\delta L(p)}{r \delta L(p)} \bigg)^{2+\beta} \bigg( \frac{r \delta L(p)}{\ul{\lambda} L(p)} \bigg)^{\beta/2} \bigg) \nonumber \\
& \leq c_3 \delta^{\beta/2} L(p)^2 \bigg( 1 + \sum_{r=1}^{\infty} r^{-1-\beta+\beta/2} \bigg) \leq c_4 \delta^{\beta/2} L(p)^2, \label{eq:upper_bd_Lambda2}
\end{align}
where the second inequality uses the above claim, which implies, for a vertex $v$ with $\frac{d(v,z_B)}{\delta L(p)} \in [r,r+1)$ (and $r \leq \ul{\lambda} / \delta$), that
\begin{align*}
\PP \big(B \text{ separates } & v \text{ from } 0 \text{ in } \holes(\tilde{p}) \text{ for some } \tilde{p} \big)\\
& \leq \PP \big( \arm^{\dol{p}^{(K)}, \dul{p}^{(K)}}_6(\Ann_{\delta L(p), r \delta L(p)}(z_B)) \big) \PP \big( \arm^{\dol{p}^{(K)}, \dul{p}^{(K)}}_3(\Ann_{(r+1) \delta L(p) \wedge \ul{\lambda} L(p), \ul{\lambda} L(p)}(z_B)) \big),
\end{align*}
as well as \eqref{eq:1arm} (which also holds with $\beta$ replaced by $\frac{\beta}{2}$) and \eqref{eq:6arm}. Finally, we can write
\begin{align*}
\PP \big( \big| & \Dout{\holes(\tilde{p})}{\delta L(p)} \setminus \Dint{\holes(\tilde{p})}{\delta L(p)} \big| \geq \eta \big| \holes(\tilde{p}) \big| \text{ for some } \tilde{p} \big)\\
& \leq \PP \big( \big| \Dout{\holes(\tilde{p})}{\delta L(p)} \setminus \Dint{\holes(\tilde{p})}{\delta L(p)} \big| \geq \eta \big| \holes(\tilde{p}) \big| \text{ for some } \tilde{p}, \text{ and } A \cap \tilde{A} \text{ occurs} \big) + \PP((A \cap \tilde{A})^c).
\end{align*}
Since $\PP((A \cap \tilde{A})^c) = \PP(A^c) + \PP(A \cap \tilde{A}^c) \leq \frac{\ve}{2}$ (from \eqref{eq:approx_hole_A1} and \eqref{eq:approx_hole_A2}), we get that the desired probability is at most
\begin{align*}
\PP \big( \big| & \Lambda_1(\tilde{p}) \big| + \big| \Lambda_2(\tilde{p}) \big| \geq \eta \big| \holes(\tilde{p}) \big| \text{ for some } \tilde{p}, \text{ and } A \cap \tilde{A} \text{ occurs} \big) + \frac{\ve}{2}\\
& \leq \PP \big( \big| \ol{\Lambda}_1 \big| + \big| \ol{\Lambda}_2 \big| \geq \eta \big| \Ball_{\ul{\lambda} L(p)} \big|, \text{ and } A \cap \tilde{A} \text{ occurs} \big) + \frac{\ve}{2}.
\end{align*}
Using Markov's inequality, combined with \eqref{eq:upper_bd_Lambda1} and \eqref{eq:upper_bd_Lambda2}, we obtain
that this is at most
$$\frac{1}{\eta \big| \Ball_{\ul{\lambda} L(p)} \big|} \Big( \EE \Big[ \big| \ol{\Lambda}_1 \big| \ind_{A \cap \tilde{A}} \Big] + \EE \Big[ \big| \ol{\Lambda}_2 \big| \ind_{A \cap \tilde{A}} \Big] \Big) \leq c \frac{\delta^{\beta/2}}{\eta} + \frac{\ve}{2},$$
which can be made $\leq \ve$ by taking $\delta$ sufficiently small (depending on $\eta$, $\ve$ and $K$). This completes the proof of Lemma~\ref{lem:approx_hole_unif}.
\end{proof}

\bibliographystyle{plain}
\bibliography{FF_deconcentration}

\end{document}